\documentclass[a4paper, 11pt, twoside]{article}
\usepackage{amsthm}
\usepackage{fancyhdr}
\usepackage{amssymb}
\usepackage{mathtools}
\usepackage{indentfirst}
\usepackage{color}
\usepackage{bm}
\usepackage{extarrows}
\usepackage[pagewise]{lineno}

\def \dive {\,{\rm div}\,}
\theoremstyle{plain}
\newtheorem{proposition}{Proposition}
\newtheorem{lemma}[proposition]{Lemma}
\newtheorem{theorem}[proposition]{Theorem}

\theoremstyle{definition}

\newtheorem{remark}[proposition]{Remark}

\theoremstyle{definition}

\theoremstyle{plain}

\numberwithin{equation}{section}
\numberwithin{proposition}{section}
\numberwithin{conj}{section}

\usepackage{amsmath}

\usepackage{geometry}
\begin{document}

\title{Local existence of classical solution to the chemotaxis-shallow water system with vacuum in $\mathbb{R}^2$}
\author{
\sl{Li Chen}\\
   \small\emph{School of Business Informatics and Mathematics, University of Mannheim,} \\
   \small\emph{Mannheim, 68159, Germany}\\
\sl{Zhen Luo}\\
   \small\emph{School of Mathematical Sciences, Xiamen University,} \\
   \small\emph{Xiamen, 361005, China}\\
\sl{Yucheng Wang}\\
  \small\emph{School of Business Informatics and Mathematics, University of Mannheim,} \\
   \small\emph{Mannheim, 68159, Germany}\\
    }
\date{}
\footnotetext[1]{2020MSC: 35Q35; 35Q92; 92C17; 76N10.}
\footnotetext[3]{E-mail addresses: li.chen@uni-mannheim.de (L.Chen), zluo@xmu.edu.cn (Z.Luo),\\
yucheng.wang@uni-mannheim.de (Y.Wang).}

\maketitle

\section*{Abstract}

In this paper, we consider the chemotaxis-shallow water system in $\mathbb{R}^2$. We establish the local existence of classical solution without assuming the initial height is small or has a small perturbation near a constant. The far field behavior of the height is a constant which could be either vacuum or non-vacuum. The initial data is allowed vacuum and the spatial measure of the set of vacuum can be arbitrarily large.

\section{Introduction}

In this paper, we consider the chemotaxis-shallow water system which was proposed in \cite{C2016}
\begin{equation}\label{A}
\begin{cases}
n_{t}+\dive(n\mathbf{u})=D_{n}\Delta n-\nabla\cdot(n\chi(c)\nabla c),\\
c_{t}+\dive(c\mathbf{u})=D_{c}\Delta c-nf(c),\\
h_{t}+\dive(h\mathbf{u})=0,\\
h\mathbf{u}_{t}+h\mathbf{u}\cdot\nabla\mathbf{u}+h^2\nabla n+\frac{1}{2}(1+n)\nabla h^2=\mu\Delta\mathbf{u}+(\mu+\lambda)\nabla(\dive\mathbf{u}),
\end{cases}
\end{equation}
where $x\in\mathbb{R}^2,\ t\geq 0$. Here $n, c, h, \mathbf{u}$ represent the bacterial density, substrate concentration, the fluid height and the fluid velocity respectively. The constants $\mu$ and $\lambda$ are the shear viscosity and the bulk viscosity coefficients respectively which satisfy $\mu>0,\ \mu+\lambda\geq 0$. The two constants $D_{n}$ and $D_{c}$ are the corresponding diffusion coefficients for the cells and substrate. The constant $\chi(c)$ represents the chemotactic sensitivity and the given function $f(c)$ represents the consumption rate of the substrate by the cells. 

Chemotaxis is a well-known biological phenomenon which describes the movement of cells, bacteria or microorganism in response to the chemical signal. Mathematically, the best-studied model to describe the chemotaxis is the Keller-Segel system \cite{K1970, K1971}. This system is very interesting as it has two different mechanisms, namely attraction and repulsion. For the attractive Keller-Segel system, as the competition between the dissipation and non-local aggregation, the solutions have different behavior depending on the space dimension. For one dimension, \cite{O2001} established the global classical solution. For two and three dimensions, a threshold of critical mass between the global existence and finite-time blow-up has been conjectured in \cite{H1996, H1997, H2001, N1997, W2010}. For the repulsive Keller-Segel system, the global existence of the solution was established in \cite{C2008}.

In real nature, the chemotaxis process often occurs in a fluid environment. Recently, the chemotaxis coupled with fluid equations have attracted many mathematicians. Based on some rational hypotheses, we can derive the system (\ref{A}) from the three-dimensional chemotaxis-incompressible Navier-Stokes system. The details of the derivation can refer to \cite{C2016}. The chemotaxis-Navier-Stokes system was first proposed in \cite{T2005} to describe the interaction between bacteria and the surrounding fluid in which their nutrient is dissolved. \cite{L2010} constructed the local-in-time weak solutions in a bounded domain in $\mathbb{R}^{2}$ and $\mathbb{R}^{3}$. For the Cauchy problem in two dimensions, \cite{D2010} established the global weak solutions to the chemotaxis-Stokes system i.e. the nonlinear convective term $\mathbf{u}\cdot\nabla\mathbf{u}$ does not appear in the fluid equation of chemotaxis-Navier-Stokes system. \cite{D2017} and \cite{L2011} obtained the global weak solutions to the chemotaxis-Navier-Stokes systems. \cite{D2017} also established the global existence of classical solution. For some small initial data, the global existence of smooth solution was established in \cite{C2014}. For the Cauchy problem in three dimensions, the global existence of classical solution near constant steady state was established in \cite{D2010}. For the initial boundary value problem, in \cite{W2012}, the global existence of classical solution was established in two dimensions. Later, \cite{W2014} proved that this solution stabilizes to the equilibrium $(\overline{n}_{0}, 0, \mathbf{0})$ with respect to $L^{\infty}$ norm. \cite{W2016, W2017} established the global existence and stabilization of this system in three dimensions. 

For the chemotaxis-shallow water system (\ref{A}), \cite{C2016} established the local existence of solutions and the blow-up criterion. Later, \cite{T2018, W2019, Y2021} established the global existence and decay estimates of the Cauchy problem. For the initial boundary value problem, the global existence of the solution can refer to \cite{W2020, W2022}.

In this paper, we consider the case
\begin{equation*}
D_{n}=D_{c}=1,\ \chi(c)=1,\ f(c)=c,
\end{equation*}
and establish the existence of classical solution with far field behavior
\begin{equation}\label{A100B}
(n, c, \mathbf{u})(x, t)\to (0, 0, \mathbf{0}),\ h(x, t)\to\widetilde{h}\geq 0,\ \text{as}\ |x|\to\infty.
\end{equation}
The initial data
\begin{equation}\label{A2}
(n, c, h, \mathbf{u})(x, 0)=(n_{0}, c_{0}, h_{0}, \mathbf{u}_{0})(x),
\end{equation}
where $n_{0}\geq 0,\ c_{0}\geq 0,\ h_{0}\geq 0$ satisfy the compatibility condition
\begin{equation}\label{A3}
-\mu\Delta\mathbf{u}_{0}-(\mu+\lambda)\nabla(\dive\mathbf{u}_{0})+h_{0}^2\nabla n_{0}+\frac{1}{2}(1+n_{0})\nabla h_{0}^2=h_{0}g,
\end{equation}
for some $g\in D^1$ with $h_{0}^{1/2}g\in L^2$. Before stating the main theorem, we introduce the notations and conventions used throughout this paper. For $\Omega\subset\mathbb{R}^2$, we write
\begin{equation*}
\int f\, dx=\int_{\Omega} f\, dx.
\end{equation*}
For $1<r<\infty, k\in\mathbb{N}$, we denote the standard homogeneous and inhomogeneous Sobolev spaces as follows:
\begin{equation*}
\begin{cases}
L^r=L^r(\Omega), D^{k, r}=\{\mathbf{u}\in L^{1}_{loc}(\Omega)|\|\nabla^k\mathbf{u}\|_{L^r}<\infty\},\ \|\mathbf{u}\|_{D^{k, r}}:=\|\nabla^k\mathbf{u}\|_{L^r},\\
W^{k, r}=L^r\cap D^{k, r},\ H^k=W^{k, 2},\ D^{k}=D^{k, 2},\ D_{0}^{k}=\{\mathbf{u}\in D^k\big|\mathbf{u}|_{\partial\Omega}=\mathbf{0}\}.
\end{cases}
\end{equation*}
Moreover, the material derivative is defined as
\begin{equation*}
\dot{f}:=f_{t}+\mathbf{u}\cdot\nabla f.
\end{equation*}
The aim of this paper is to establish the local existence of classical solution without any restriction on the smallness of initial data or perturbation in the sense of $\|h_{0}-\widetilde{h}\|_{H^3}$ which allows the vacuum for the initial height. The main result in this paper can be stated as follows:

\begin{theorem}\label{B}
For $\widetilde{h}\geq 0$ and $\Omega=\mathbb{R}^2$, assume that the initial data $(n_{0}, c_{0}, h_{0}, \mathbf{u}_{0})$ satisfy
\begin{equation}\label{303}
n_{0}\in H^3,\ c_{0}\in H^3,\ h_{0}-\widetilde{h}\in H^3,\ \mathbf{u}_{0}\in D^1\cap D^3,\ \sqrt{h_{0}}\mathbf{u}_{0}\in L^2,
\end{equation}
and the compatibility condition (\ref{A3}). Moreover, if $\widetilde{h}=0$, in addition to (\ref{A3}) and (\ref{303}), suppose that
\begin{equation}\label{204}
|x|^{\alpha/2}n_{0}\in L^2,\ |x|^{\alpha/2}c_{0}\in L^2,\ |x|^{\alpha/2}h_{0}^2\in L^2,\ |x|^{\alpha/2}\sqrt{h_{0}}g\in L^2
\end{equation}
and
\begin{equation}\label{304}
|x|^{\alpha/2}\nabla n_{0}\in L^2,\ |x|^{\alpha/2}\nabla c_{0}\in L^2,\ |x|^{\alpha/2}\nabla^2 n_{0}\in L^2,\ |x|^{\alpha/2}\nabla^2 c_{0}\in L^2,\ |x|^{\alpha/2}\nabla\mathbf{u}_{0}\in L^2,
\end{equation}
where
\begin{equation}\label{200}
\alpha:=\mu/(4(2\mu+\lambda))\in (0, 1/8].
\end{equation}
Then there exists a small time $T^*>0$ and a unique strong solution $(n, c, h, \mathbf{u})$ to the Cauchy problem (\ref{A})-(\ref{A2}) on $\mathbb{R}^2\times[0, T^*]$ such that for any $\tau\in(0, T^*)$,
\begin{equation*}
\begin{cases}
n\in C([0, T^*]; H^3),\ c\in C([0, T^*]; H^3),\\
(h-\widetilde{h}, h^2-\widetilde{h}^2)\in C([0, T^*]; H^3),\\
\mathbf{u}\in C([0, T^*]; L^q\cap D^1\cap D^3)\cap L^2(0, T^*; D^4)\cap L^{\infty}(\tau, T^*; D^4),\\
\mathbf{u}_{t}\in L^{\infty}(0, T^*; D^1)\cap L^2(0, T^*; L^q\cap D^2),\\
\mathbf{u}_{t}\in L^{\infty}(\tau, T^*; L^{q_{1}}\cap D^2)\cap H^1(\tau, T^*; D^1),\\
\sqrt{h}\mathbf{u}_{t}\in L^{\infty}(0, T^*; L^2),\ \sqrt{h}\mathbf{u}_{tt}\in L^2(0, T^*; L^2),\\
\sqrt{h}\mathbf{u}_{tt}\in L^{\infty}(\tau, T^*; L^2),
\end{cases}
\end{equation*}
where
\begin{equation*}
\begin{cases}
q=q_{1}=2,\ &\text{if}\ \widetilde{h}>0,\\
q=4/\alpha,\ q_{1}=8/\alpha,\ &\text{if}\ \widetilde{h}=0.
\end{cases}
\end{equation*}
\end{theorem}

\begin{remark}
The two-dimensional case is the critical case for the standard Sobolev embedding theorem. Thus, it is difficult to get the $L^p$-norm of the velocity $\mathbf{u}$ and the material derivative $\dot{\mathbf{u}}$ depending on the gradient of it. When the far field state $\widetilde{h}>0$, in \cite{C2016}, the authors established the local existence of strong solutions to the system (\ref{A}) in $\mathbb{R}^2$, where the $L^2$-norm of $\mathbf{u}$ was achieved through the following Lemma \ref{Q1}. In this paper, we consider the case that includes $\widetilde{h}=0$, where Lemma \ref{Q1} is no longer true. How to establish the $L^p$-norm is the main difficulty of our paper. We need some new elaborate estimates to overcome this difficulty.
\end{remark}

\begin{remark}
If the initial height is not allowed to vanish, for the Cauchy problem in $\mathbb{R}^2$, the global existence of strong solutions to the system (\ref{A}) are established in \cite{T2018, W2019, Y2021}. The solution in our main result can be expected to exist globally if the initial height is allowed to vanish.
\end{remark}

The rest of this paper is organized as follows: In Section 2, we list some known results and inequalities which will be used in latter analysis. Section 3 establishes necessary the a priori estimates on strong solutions for $\widetilde{h}=0$ and obtains the local strong solution to the system (\ref{A}) in $\mathbb{R}^2$. Finally, the main result, Theorem \ref{B} is proved in Section 4.

\section {Preliminaries}

In this section, we recall some inequalities, Aubin-Lions lemma and fixed point theorem which will be used later. When the far field state is $\widetilde{h}>0$, it is easy to get the $L^2$-norm for $\mathbf{u}$ through the following lemma.

\begin{lemma}(see \cite{D2012, L2012})\label{Q1}
If $\widetilde{h}>0$, then there exists some constant $C(\widetilde{h})>0$ such that the following estimate holds for $h-\widetilde{h}\in L^2, \mathbf{u}\in D^1, h^{1/2}\mathbf{u}\in L^2$,
\begin{equation}\label{173}
\|\mathbf{u}\|_{L^2}^2\leq C(\int h|\mathbf{u}|^2\, dx+\|h-\widetilde{h}\|_{L^2}^2\|\nabla\mathbf{u}\|_{L^2}^2).
\end{equation}
\end{lemma}

When the far field state is $\widetilde{h}=0$, we can use the following Caffarelli-Kohn-Nirenberg inequality to establish the $L^p$ estimate of $\mathbf{u}$.

\begin{lemma}(see \cite{CL1984, CF2001})\label{f}
For $\alpha\in (0, 2)$, the following estimates hold for all $\mathbf{u}\in C_{0}^{\infty}(\mathbb{R}^2)$,
\begin{equation*}
\int |x|^{\alpha-2}|\mathbf{u}|^2\, dx\leq\frac{\alpha^2}{4}\int |x|^{\alpha}|\nabla\mathbf{u}|^2\, dx,\ \|\mathbf{u}\|_{L^{4/\alpha}}^2\leq C(\alpha)\int |x|^{\alpha}|\nabla\mathbf{u}|^2\, dx.
\end{equation*}
\end{lemma}

We will use frequently the following Gagliardo-Nirenberg inequality.

\begin{lemma}(see \cite{L1968})\label{a}
For $p\geq 2, q\in (1, \infty)$, and $r\in (2, \infty)$, there exists some generic constant $C>0$ which may depend on $q, r$ such that for $f\in H^1(\mathbb{R}^2)$ and $g\in L^q(\mathbb{R}^2)\cap D^{1, r}(\mathbb{R}^2)$, we have
\begin{equation*}
\begin{split}
\|f\|_{L^p}&\leq C\|f\|_{L^2}^{2/p}\|\nabla f\|_{L^2}^{(p-2)/p},\\
\|g\|_{C(\overline{\mathbb{R}^2})}&\leq C\|g\|_{L^q}^{q(r-2)/(2r+q(r-2))}\|\nabla g\|_{L^r}^{2r/(2r+q(r-2))}.
\end{split}
\end{equation*}
\end{lemma}

Next, we recall the following Aubin-Lions lemma.

\begin{lemma}(see \cite{SJ1987})\label{G1}
Let $X_{0}, X$, and $X_{1}$ be three Banach spaces with $X_{0}\subset X\subset X_{1}$. Suppose that $X_{0}$ is compactly embedded in $X$ and that $X$ is continuously embedded in $X_{1}$.
\begin{itemize}
\item [(\romannumeral 1)] Let $G$ be bounded in $L^p(0, T; X_{0})$, where $1\leq p<\infty$, and let $\partial G/\partial t$ be bounded in $L^1(0, T; X_{1})$. Then $G$ is relatively compact in $L^p(0, T; X)$.
\item [(\romannumeral 2)] Let $F$ be bounded in $L^{\infty}(0, T; X_{0})$ and $\partial F/\partial t$ be bounded in $L^r(0, T; X_{1})$, where $r>1$. Then $F$ is relatively compact in $C([0, T]; X)$.
\end{itemize}
\end{lemma}

The local existence of strong solution in a bounded domain will be obtained by the following Schauder fixed point theorem.

\begin{lemma}(see \cite{GD1983})
Let $\mathfrak{R}$ be a compact convex set in a Banach space $\mathcal{B}$ and let $\mathcal{T}$ be a continuous mapping of $\mathfrak{R}$ into itself. Then $\mathcal{T}$ has a fixed point, that is $\mathcal{T} x=x$ for some $x\in\mathfrak{R}$.
\end{lemma}

\section{Local existence of strong solution}

In this section, we establish the local existence of strong solution to the system (\ref{A}) with the far field behavior (\ref{A100B}) and initial data (\ref{A2}) in $\mathbb{R}^2$. In order to do this, we should first derive some uniform local (in time) a priori estimates for the linearized systems in a bounded domain $\Omega\subset\mathbb{R}^2$ as the specific order. Beginning with the equation for the fluid height $h$,
\begin{equation}\label{1}
\begin{cases}
h_{t}+\dive(h\mathbf{v})=0, &x\in\Omega, t>0,\\
h(x, 0)=h_{0}, &x\in\Omega,
\end{cases}
\end{equation}
we then consider the equation for $c$,
\begin{equation}\label{2}
\begin{cases}
c_{t}+\dive(c\mathbf{v})=\Delta c-mc, &x\in\Omega, t>0,\\
c=0, &x\in\partial\Omega, t>0,\\
c(x, 0)=c_{0}, &x\in\Omega.
\end{cases}
\end{equation}
Based on the estimates of (\ref{2}), we investigate the equation for the bacterial density $n$,
\begin{equation}\label{3}
\begin{cases}
n_{t}+\dive(n\mathbf{v})=\Delta n-\nabla\cdot(n\nabla c), &x\in\Omega, t>0,\\
n=0, &x\in\partial\Omega, t>0,\\
n(x, 0)=n_{0}, &x\in\Omega.
\end{cases}
\end{equation}
Finally, using the estimates of (\ref{1}) and (\ref{3}), we consider the equation for the fluid velocity $\mathbf{u}$,
\begin{equation}\label{4}
\begin{cases}
h\mathbf{u}_{t}+h\mathbf{v}\cdot\nabla\mathbf{u}+L\mathbf{u}+h^2\nabla n+\frac{1}{2}(1+n)\nabla h^2=0, & x\in\Omega, t>0,\\
\mathbf{u}=\mathbf{0}, &x\in\partial\Omega, t>0,\\
\mathbf{u}(x, 0)=\mathbf{u}_{0}, &x\in\Omega.
\end{cases}
\end{equation}
Here $L:=-\mu\Delta-(\mu+\lambda)\nabla\dive$ is a strongly elliptic operator. The given functions $\mathbf{v}$ and $m$ satisfy
\begin{equation}\label{100}
\mathbf{v}\in C([0, T]; H_{0}^1\cap H^3)\cap L^2(0, T; H^4),\ \mathbf{v}_{t}\in L^{\infty}(0, T; H_{0}^1)\cap L^2(0, T; H^2)
\end{equation}
and
\begin{equation}\label{B1011}
m\in C([0, T]; H_{0}^1\cap H^3)\cap L^2(0, T; H^4),\ m_{t}\in L^{\infty}(0, T; H_{0}^1)\cap L^2(0, T; H^2).
\end{equation}
Based on the a priori estimates that we established and the Schauder fixed point theorem, we obtain the local strong solution to the system (\ref{A}) in a bounded domain. Then, let the size of the domain tends to infinity and the lower bound of the fluid height tends to zero, we can obtain the local strong solution of the system (\ref{A}) in the whole space $\mathbb{R}^2$. Before starting to prove, we give the following lemma.

\begin{lemma}\label{G2}
Assume that the initial data $(n_{0}, c_{0}, h_{0}-\widetilde{h}, \mathbf{u}_{0})$ satisfy (\ref{303}) and $n_{0}\geq0, c_{0}\geq 0, h_{0}\geq\delta$ on $\Omega\times[0, T]$ for some $\delta>0$. $\mathbf{v}$ and $m$ satisfy the conditions (\ref{100}) and (\ref{B1011}). Then there exists a unique strong solution $(h, c, n, \mathbf{u})$ to (\ref{1})-(\ref{4}) such that for any $T>0$, we have
\begin{equation}\label{B101}
\begin{cases}
n\in C([0, T]; H_{0}^3),\ n_{t}\in L^{\infty}(0, T; H^1)\cap L^2(0, T; H^2),\\
c\in C([0, T]; H_{0}^3),\ c_{t}\in L^{\infty}(0, T; H^1)\cap L^2(0, T; H^2),\\
h-\widetilde{h}\in C([0, T]; H^3),\ h_{t}\in C(0, T; H^2),\\
\mathbf{u}\in C([0, T]; H_{0}^3),\ \mathbf{u}_{t}\in L^2(0, T; H^2).
\end{cases}
\end{equation}
\end{lemma}

\begin{proof}
The solution $h$ is represented by the formula
\begin{equation}\label{104}
h(x, t)=h_{0}(U(0; x, t))\exp\Big\{-\int_{0}^{t}\dive\mathbf{v}(U(s; x, t), s)\, ds\Big\},
\end{equation}
where $U\in C([0, T]; \Omega\times[0, T])$ is the solution to the initial value problem
\begin{equation*}
\begin{cases}
\frac{\partial}{\partial s}U(s; x, t)=\mathbf{v}(s; U(s; x, t)), & 0\leq s\leq T,\\
U(t; x, t)=x, & 0\leq s\leq T, x\in\overline\Omega.
\end{cases}
\end{equation*}
By the characteristic method, we can obtain the existence and regularity of the unique solution to the equation (\ref{1}) (see \cite{c2006}). Next, solving the linear parabolic equations (\ref{2}) and (\ref{3}), the unique solution which satisfies the regularity in (\ref{B101}) can be established. Finally, note that $L$ is a strong elliptic operator (see \cite{C2004}), by standard method such as a semi-discrete Galerkin method or the method of continuity, together with the regularity of $h$ and $n$, we can obtain the unique solution to (\ref{4}). We omit the details.
\end{proof}

\subsection{Uniform a priori estimates for the linearized problem}

In this subsection, we derive some uniform local (in time) a priori estimates for strong solution to the linear problem (\ref{1})-(\ref{4}). These estimates are independent of the lower bound $\delta$ of $h_{0}$ and the size of the domain $\Omega$. According to the Lemma \ref{Q1}, it is easy to get the $L^2$ norm of $\mathbf{u}$ when the far field state $\widetilde{h}>0$. For the far field state $\widetilde{h}=0$, the main difficulty is how to obtain the $L^p$ norm for $\mathbf{u}$ itself and the material derivatives of the velocity $\dot{\mathbf{u}}$ for some $p\geq 2$. In the rest of this subsection, we establish some local (in time) a priori estimates in a bounded domain under the assumption that $\widetilde{h}=0$. Set
\begin{equation}\label{102}
\begin{split}
c_{0}:=&1+\|n_{0}\|_{H^3}^2+\|c_{0}\|_{H^3}^2+\|(h_{0}, h_{0}^2)\|_{H^3}^2+\|h_{0}^{1/2}\mathbf{u}_{0}\|_{L^2}^2+\|\mathbf{u}_{0}\|_{D^1\cap D^3}^2\\
&+\||x|^{\alpha/2}(n_{0}, c_{0}, h_{0}^2, \sqrt{h}_{0}g)\|_{L^2}^2+\||x|^{\alpha/2}(\nabla^2 n_{0},\nabla^2 c_{0})\|_{L^2}^2\\
&+\||x|^{\alpha/2}(\nabla n_{0}, \nabla c_{0}, \nabla\mathbf{u}_{0})\|_{L^2}^2+\|g\|_{D^1}^2,
\end{split}
\end{equation}
where
\begin{equation*}
\alpha:=\mu/(4(2\mu+\lambda))\in (0, 1/8].
\end{equation*}
We define
\begin{equation}\label{103}
\begin{split}
\Phi(\mathbf{v}, m, t):=&1+\sup_{0\leq s\leq t}(\|\nabla\mathbf{v}\|_{H^1}^2+\||x|^{\alpha/2}\nabla\mathbf{v}\|_{L^2}^2+\|m\|_{H^1}^2+\|m_{t}\|_{L^2}^2)\\
&+\int_{0}^{t}(\|\mathbf{v}_{t}\|_{L^{4/\alpha}}^2+\|\nabla\mathbf{v}_{t}\|_{L^2}^2+\|\nabla\mathbf{v}\|_{W^{1, {4/\alpha}}\cap H^1}^2+\|\nabla m\|_{H^1}^2+\|\nabla m_{t}\|_{L^2}^2)\, ds.
\end{split}
\end{equation}
From now on, we always assume that $0\leq t\leq T\leq 1$ and $C$ are various constants depending on $\mu$ and $\lambda$ but not depending on the lower bound $\delta$ of $h_{0}$ and the size of the domain $\Omega$. We begin with the estimates of $h$.
\begin{lemma}\label{c}
Let $h$ be a smooth solution of (\ref{1}). For $q=4/\alpha$, we have
\begin{equation}\label{105}
\sup_{0\leq s\leq t}\|(h, h^2)\|_{L^2\cap L^{\infty}}^2\leq Cc_{0}\exp\{C\Phi(\mathbf{v}, m, t)t^{1/2}\}
\end{equation}
and
\begin{equation}\label{106}
\sup_{0\leq s\leq t}(\|\nabla h\|_{L^p}^2+\|\nabla h^2\|_{L^p}^2+\|h_{t}\|_{L^p}^2)\leq Cc_{0}\Phi(\mathbf{v}, m, t)\exp\{C\Phi(\mathbf{v}, m, t)t^{1/2}\},
\end{equation}
for all $p\in[2, q]$, where $c_{0}$ and $\Phi$ are defined in (\ref{102}) and (\ref{103}).
\end{lemma}

\begin{proof}
The equation $(\ref{1})_{1}$ implies that $h^2$ satisfies the equation
\begin{equation}\label{5}
h^2_{t}+\mathbf{v}\cdot\nabla h^2+2h^2\dive\mathbf{v}=0.
\end{equation}
It is easy to check that
\begin{equation*}
\frac{d}{dt}\|(h, h^2)\|_{L^2}^2\leq C\|\nabla\mathbf{v}\|_{L^{\infty}}\|(h, h^2)\|_{L^2}^2,
\end{equation*}
which together with (\ref{104}) gives that
\begin{equation}\label{107}
\sup_{0\leq s\leq t}\big(\|(h, h^2)\|_{L^2}^2+\|(h, h^2)\|_{L^{\infty}}^2\big)\leq Cc_{0}\exp\{C\Phi(\mathbf{v}, m, t)t^{1/2}\}.
\end{equation}
For any $p\in [2, q]$, operating $\nabla$ on both sides of the equation $(\ref{1})_{1}$, multiplying the resulting equation by $p|\nabla h|^{p-2}\nabla h$ and integrating over $\Omega$ to get
\begin{equation*}
\frac{d}{dt}\int |\nabla h|^p\, dx\leq C(\|\dive\mathbf{v}\|_{L^{\infty}}\|\nabla h\|_{L^p}^p+\|h\|_{L^{\infty}}\|\nabla^2\mathbf{v}\|_{L^p}\|\nabla h\|_{L^p}^{p-1}).
\end{equation*}
We deduce from the Gronwall's inequality that
\begin{equation}\label{108}
\sup_{0\leq s\leq t}\|\nabla h\|_{L^p}^2\leq Cc_{0}\exp\{C\Phi(\mathbf{v}, m, t)t^{1/2}\}.
\end{equation}
Following the estimate of (\ref{108}), we have
\begin{equation}\label{109}
\sup_{0\leq s\leq t}\|\nabla h^2\|_{L^p}^2\leq Cc_{0}\exp\{C\Phi(\mathbf{v}, m, t)t^{1/2}\}.
\end{equation}
Finally, the equation $(\ref{1})_{1}$ gives that
\begin{equation}\label{110}
\sup_{0\leq s\leq t}\|h_{t}\|_{L^p}^2=\sup_{0\leq s\leq t}\|-\dive(h\mathbf{v})\|_{L^p}^2\leq Cc_{0}\Phi(\mathbf{v}, m, t)\exp\{C\Phi(\mathbf{v}, m, t)t^{1/2}\}.
\end{equation}
The estimates (\ref{108})-(\ref{110}) yield (\ref{106}).
\end{proof}

\begin{lemma}\label{b}
Let $c$ be a smooth solution of the equation (\ref{2}). Then
\begin{equation}\label{111}
\sup_{0\leq s\leq t}(\|c\|_{H^1}^2+\|c_{t}\|_{L^2}^2)+\int_{0}^{t}(\|\nabla c_{t}\|_{L^2}^2+\|\nabla^2 c\|_{L^2\cap L^4}^2)\, ds\leq C c_{0}\exp\{C\Phi^2(\mathbf{v}, m, t)t^{1/4}\}
\end{equation}
and
\begin{equation}\label{112}
\begin{split}
&\sup_{0\leq s\leq t}(\|\nabla^2 c\|_{L^2\cap L^4}^2+\|\nabla c_{t}\|_{L^2}^2)+\int_{0}^{t}(\|c_{tt}\|_{L^2}^2+\|\nabla^2 c_{t}\|_{L^2}^2)\, ds\\
\leq&Cc_{0}\Phi^2(\mathbf{v}, m, t)\exp\{C\Phi^2(\mathbf{v}, m, t)t^{1/4}\}.
\end{split}
\end{equation}
\end{lemma}

\begin{proof}
Multiplying the equation $(\ref{2})_{1}$ by $c$ and integrating over $\Omega$ to get
\begin{equation*}
\frac{d}{dt}\int c^2\, dx+\int |\nabla c|^2\, dx\leq C(\|\dive\mathbf{v}\|_{L^{\infty}}+\|m\|_{L^{\infty}})\|c\|_{L^2}^2.
\end{equation*}
Using the Gronwall's inequality leads that
\begin{equation}\label{113}
\sup_{0\leq s\leq t}\|c\|_{L^2}^2+\int_{0}^{t}\|\nabla c\|_{L^2}^2\, ds\leq Cc_{0}\exp\{C\Phi(\mathbf{v}, m, t)t^{1/2}\}.
\end{equation}
Multiplying the equation $(\ref{2})_{1}$ by $c_{t}$ and integrating over $\Omega$ gives that
\begin{equation*}
\begin{split}
&\frac{1}{2}\frac{d}{dt}\int |\nabla c|^2\, dx+\int c_{t}^2\, dx\\
\leq&C(\|\mathbf{v}\|_{L^{\infty}}\|\nabla c\|_{L^2}\|c_{t}\|_{L^2}+\|c\|_{L^4}\|\dive\mathbf{v}\|_{L^4}\|c_{t}\|_{L^2}+\|m\|_{L^4}\|c\|_{L^4}\|c_{t}\|_{L^2})\\
\leq&\varepsilon\|c_{t}\|_{L^2}^2+C_{\varepsilon}(\|\mathbf{v}\|_{L^{\infty}}^2\|\nabla c\|_{L^2}^2+\|c\|_{L^4}^2\|\dive\mathbf{v}\|_{L^4}^2+\|m\|_{L^4}^2\|c\|_{L^4}^2).
\end{split}
\end{equation*}
Choosing $\varepsilon=1/2$ and using Gronwall's inequality, we have
\begin{equation}\label{114}
\sup_{0\leq s\leq t}\|\nabla c\|_{L^2}^2+\int_{0}^{t}\|c_{t}\|_{L^2}^2\, ds\leq Cc_{0}\exp\{C\Phi^2(\mathbf{v}, m, t)t^{1/2}\}.
\end{equation}
In order to estimate $\sup_{0\leq s\leq t}\|c_{t}\|_{L^2}^2$, we differentiate the equation $(\ref{2})_{1}$ with respect to $t$
\begin{equation}\label{70}
c_{tt}+\dive(c\mathbf{v})_{t}=\Delta c_{t}-(mc)_{t}.
\end{equation}
Multiplying the equation (\ref{70}) by $c_{t}$ and integrating over $\Omega$ gives that
\begin{equation}\label{175}
\begin{split}
&\frac{1}{2}\frac{d}{dt}\int c_{t}^2\, dx+\int |\nabla c_{t}|^2\, dx\\
=&-\frac{1}{2}\int\dive\mathbf{v}c_{t}^2\, dx-\int\mathbf{v}_{t}\cdot\nabla cc_{t}\, dx-\int c\dive\mathbf{v}_{t}c_{t}\, dx-\int m_{t}cc_{t}\, dx-\int mc_{t}^2\, dx\\
\leq&C(\|\dive\mathbf{v}\|_{L^{\infty}}\|c_{t}\|_{L^2}^2+\|\mathbf{v}_{t}\|_{L^{q}}\|\nabla c\|_{L^2}\|c_{t}\|_{L^p}+\|c\|_{L^4}\|c_{t}\|_{L^4}\|\dive\mathbf{v}_{t}\|_{L^2}\\
&+\|m_{t}\|_{L^4}\|c\|_{L^4}\|c_{t}\|_{L^2}+\|m\|_{L^{\infty}}\|c_{t}\|_{L^2}^2)\\
\leq&\varepsilon\|\nabla c_{t}\|_{L^2}^2+C_{\varepsilon}(\|\dive\mathbf{v}\|_{L^{\infty}}+\|\dive\mathbf{v}_{t}\|_{L^2}+\|m\|_{L^{\infty}}+1)\|c_{t}\|_{L^2}^2\\
&+C(\|\mathbf{v}_{t}\|_{L^q}^2\|\nabla c\|_{L^2}^2+\|c\|_{L^2}\|\nabla c\|_{L^2}\|\dive\mathbf{v}_{t}\|_{L^2}^{3/2}+\|m_{t}\|_{L^2}^2\|\nabla c\|_{L^2}^2+\|\nabla m_{t}\|_{L^2}^2\|c\|_{L^2}^2),
\end{split}
\end{equation}
where $q=4/\alpha, 1/p+1/q=1/2$. Choosing $\varepsilon=1/2$ and applying Gronwall's inequality to (\ref{175}), we have
\begin{equation}\label{115}
\sup_{0\leq s\leq t}\|c_{t}\|_{L^2}^2+\int_{0}^{t}\|\nabla c_{t}\|_{L^2}^2\, ds\leq Cc_{0}\exp\{C\Phi^2(\mathbf{v}, m, t)t^{1/4}\}.
\end{equation}
The classical $L^p$ theory asserts that
\begin{equation}\label{60}
\begin{split}
\|\nabla^2 c\|_{L^2}^2\leq&C(\|c_{t}\|_{L^2}^2+\|\mathbf{v}\cdot\nabla c\|_{L^2}^2+\|c\dive\mathbf{v}\|_{L^2}^2+\|mc\|_{L^2}^2)\\
\leq&C(\|c_{t}\|_{L^2}^2+\|\mathbf{v}\|_{L^\infty}^2\|\nabla c\|_{L^2}^2+\|c\|_{L^4}^2\|\dive\mathbf{v}\|_{L^4}^2+\|m\|_{L^4}^2\|c\|_{L^4}^2),
\end{split}
\end{equation}
which together with (\ref{113}), (\ref{114}), and (\ref{115}) implies that 
\begin{equation}\label{177}
\int_{0}^{t}\|\nabla^2 c\|_{L^2}^2\, ds\leq Cc_{0}\exp\{C\Phi^2(\mathbf{v},m, t)t^{1/4}\}
\end{equation}
and
\begin{equation}\label{116}
\sup_{0\leq s\leq t}\|\nabla^2 c\|_{L^2}^2\leq Cc_{0}\Phi^2(\mathbf{v}, m, t)\exp\{C\Phi^2(\mathbf{v},m, t)t^{1/4}\}.
\end{equation}
In order to estimate $\sup_{0\leq s\leq t}\|\nabla c_{t}\|_{L^2}^2$, we multiply the equation (\ref{70}) by $c_{tt}$ and integrate over $\Omega$ to get
\begin{equation}\label{176}
\begin{split}
&\frac{1}{2}\frac{d}{dt}\int |\nabla c_{t}|^2\, dx+\int c_{tt}^2\, dx\\
=&-\int\mathbf{v}\cdot\nabla c_{t}c_{tt}\, dx-\int c_{t}\dive\mathbf{v} c_{tt}\, dx-\int \mathbf{v}_{t}\cdot\nabla c c_{tt}\, dx-\int c\dive\mathbf{v}_{t}c_{tt}\, dx\\
&-\int m_{t}c c_{tt}\, dx-\int m c_{t}c_{tt}\, dx\\
\leq&\varepsilon\|c_{tt}\|_{L^2}^2+C_{\varepsilon}(\|\mathbf{v}\|_{L^{\infty}}^2+\|c_{t}\|_{L^2}^2)\|\nabla c_{t}\|_{L^2}^2\\
&+C_{\varepsilon}(\|\dive\mathbf{v}\|_{L^2}^2\|\nabla^2\mathbf{v}\|_{L^2}^2+\|\mathbf{v}_{t}\|_{L^q}^2\|\nabla c\|_{L^p}^2+\|c\|_{L^{\infty}}^2\|\dive\mathbf{v}_{t}\|_{L^2}^2\\
&+\|m_{t}\|_{L^2}^2\|\nabla m_{t}\|_{L^2}^2+\|c\|_{L^2}^2\|\nabla c\|_{L^2}^2+\|m\|_{L^2}^2\|\nabla m\|_{L^2}^2),
\end{split}
\end{equation}
where $q=4/\alpha, 1/p+1/q=1/2$. Choosing $\varepsilon=1/2$ and applying Gronwall's inequality to (\ref{176}), combining with (\ref{113}), (\ref{114}), (\ref{115}), and (\ref{116}), we have
\begin{equation}\label{117}
\sup_{0\leq s\leq t}\|\nabla c_{t}\|_{L^2}^2+\int_{0}^{t}\|c_{tt}\|_{L^2}^2\, ds\leq Cc_{0}\Phi^2(\mathbf{v}, m, t)\exp\{C\Phi^2(\mathbf{v}, m, t)t^{1/4}\}.
\end{equation}
Using the classical $L^p$ theory and equation (\ref{70}), together with (\ref{113}), (\ref{114}), (\ref{115}), and (\ref{117}), we have
\begin{equation}\label{178}
\begin{split}
&\int_{0}^{t}\|\nabla^2 c_{t}\|_{L^2}^2\, ds\\
\leq&C\int_{0}^{t}(\|c_{tt}\|_{L^2}^2+\|\mathbf{v}\|_{L^{\infty}}^2\|\nabla c_{t}\|_{L^2}^2+\|c_{t}\|_{L^4}^2\|\dive\mathbf{v}\|_{L^4}^2)\, ds\\
&+C\int_{0}^t(\|\mathbf{v}_{t}\|_{L^q}^2\|\nabla c\|_{L^p}^2+\|c\|_{L^{\infty}}^2\|\dive\mathbf{v}_{t}\|_{L^2}^2+\|m_{t}\|_{L^2}^2\|c\|_{L^\infty}^2+\|m\|_{L^4}^2\|c_{t}\|_{L^4}^2)\, ds\\
\leq&Cc_{0}\Phi^2(\mathbf{v}, m, t)\exp\{C\Phi^2(\mathbf{v}, m, t)t^{1/4}\},
\end{split}
\end{equation}
where $q=4/\alpha, 1/p+1/q=1/2$. According to the equation $(\ref{2})_{1}$ and elliptic estimates, we have
\begin{equation*}
\begin{split}
\|\nabla^2 c\|_{L^4}^2\leq&C(\|c_{t}\|_{L^4}^2+\|\mathbf{v}\cdot\nabla c\|_{L^4}^2+\|c\dive\mathbf{v}\|_{L^4}^2+\|mc\|_{L^4}^2)\\
\leq&C(\|c_{t}\|_{L^2}\|\nabla c_{t}\|_{L^2}+\|\mathbf{v}\|_{L^{\infty}}^2\|c\|_{L^2}^{\frac{1}{2}}\|\nabla^2 c\|_{L^2}^{\frac{3}{2}})\\
&+C(\|c\|_{L^{\infty}}^2\|\mathbf{v}\|_{L^2}^{\frac{1}{2}}\|\nabla^2\mathbf{v}\|_{L^2}^{\frac{3}{2}}+\|m\|_{L^2}^{\frac{1}{2}}\|\nabla m\|_{L^2}^{\frac{3}{2}}\|c\|_{L^2}^{\frac{1}{2}}\|\nabla c\|_{L^2}^{\frac{3}{2}}).
\end{split}
\end{equation*}
Using the estimate (\ref{113}), (\ref{114}), (\ref{115}), (\ref{177}), (\ref{116}), and (\ref{117}), we have
\begin{equation}\label{119}
\int_{0}^{t}\|\nabla^2 c\|_{L^4}^2\, ds\leq Cc_{0}\exp\{C\Phi^2(\mathbf{v},m, t)t^{1/4}\}
\end{equation}
and
\begin{equation}\label{120}
\sup_{0\leq s\leq t}\|\nabla^2 c\|_{L^4}^2\leq Cc_{0}\Phi^2(\mathbf{v}, m, t)\exp\{C\Phi^2(\mathbf{v}, m, t)t^{1/4}\}.
\end{equation}
The estimate (\ref{113}), (\ref{114}), (\ref{115}), (\ref{177}), and (\ref{119}) yield (\ref{111}). The estimate (\ref{112}) is given by (\ref{116}), (\ref{117}), (\ref{178}), and (\ref{120}).
\end{proof}

\begin{lemma}\label{d}
Let $n$ be a smooth solution of the equation (\ref{3}). Then
\begin{equation}\label{128}
\sup_{0\leq s\leq t}(\|n\|_{H^1}^2+\|n_{t}\|_{L^2}^2)+\int _{0}^{t}\|\nabla n_{t}\|_{L^2}^2\, ds\leq Cc_{0}\exp\big\{Cc_{0}^2\exp\{C\Phi^3(\mathbf{v}, m, t)t^{1/4}\}\big\},
\end{equation}
\begin{equation}\label{129}
\int_{0}^{t}\|\nabla^2 n\|_{L^2\cap L^4}^2\, ds\leq Cc_{0}^2\exp\{Cc_{0}^2\exp\big\{C\Phi^3(\mathbf{v}, m, t)t^{1/4}\}\big\},
\end{equation}
and
\begin{equation}\label{130}
\sup_{0\leq s\leq t}\|\nabla^2 n\|_{L^2}^2\leq Cc_{0}^2\Phi^2(\mathbf{v}, m, t)\exp\big\{Cc_{0}^2\exp\{C\Phi^3(\mathbf{v}, m, t)t^{1/4}\}\big\}.
\end{equation}
\end{lemma}

\begin{proof}
To prove the estimate (\ref{128}), we can use the same method to get (\ref{111}). We remain to deal with the term $-\int\nabla\cdot(n\nabla c)n\, dx, -\int\nabla\cdot(n\nabla c)n_{t}\, dx, -\int\nabla\cdot(n\nabla c)_{t}n_{t}\, dx$. A straightforward computation shows that
\begin{equation*}
\begin{split}
-\int\nabla\cdot(n\nabla c)n\, dx\leq&C\|\Delta c\|_{L^2}\|n\|_{L^2}\|\nabla n\|_{L^2}\leq\varepsilon\|\nabla n\|_{L^2}^2+C_{\varepsilon}\|\Delta c\|_{L^2}^2\|n\|_{L^2},\\
-\int\nabla\cdot(n\nabla c)n_{t}\, dx\leq&C(\|\nabla n\|_{L^2}\|\nabla c\|_{L^{\infty}}\|n_{t}\|_{L^2}+\|n\|_{L^4}\|\Delta c\|_{L^4}\|n_{t}\|_{L^2})\\
\leq&\varepsilon\|n_{t}\|_{L^2}^2+C_{\varepsilon}(\|\nabla n\|_{L^2}^2\|\nabla c\|_{L^{\infty}}^2+\|n\|_{L^4}^2\|\Delta c\|_{L^4}^2),\\
-\int\nabla\cdot(n\nabla c)_{t}n_{t}\, dx\leq&C(\|\Delta c\|_{L^2}\|n_{t}\|_{L^2}\|\nabla n_{t}\|_{L^2}+\|n\|_{L^4}\|\nabla c_{t}\|_{L^4}\|\nabla n_{t}\|_{L^2})\\
\leq&\varepsilon\|\nabla n_{t}\|_{L^2}^2+C_{\varepsilon}(\|\Delta c\|_{L^2}^2\|n_{t}\|_{L^2}^2+\|n\|_{L^4}^2\|c_{t}\|_{L^2}^{\frac{1}{2}}\|\nabla^2 c_{t}\|_{L^2}^{\frac{3}{2}}).
\end{split}
\end{equation*}
The above estimates and Lemma \ref{b} give (\ref{128}). Finally, for any $p\in[2, 4]$, the $L^p$ theory asserts
\begin{equation}\label{181}
\begin{split}
&\|\nabla^2 n\|_{L^p}^2\\
\leq&C(\|n_{t}\|_{L^p}^2+\|\mathbf{v}\cdot\nabla n\|_{L^p}^2+\|n\dive\mathbf{v}\|_{L^p}^2+\|\nabla n\cdot\nabla c\|_{L^p}^2+\|n\Delta c\|_{L^p}^2)\\
\leq&C\big(\|n_{t}\|_{L^2}^{\frac{4}{p}}\|\nabla n_{t}\|_{L^2}^{2-\frac{4}{p}}+(\|\mathbf{v}\|_{L^{\infty}}^2+\|\nabla c\|_{L^{\infty}}^2)\|\nabla n\|_{L^p}^2+\|n\|_{L^{\frac{4p}{4-p}}}^2(\|\dive\mathbf{v}\|_{L^4}^2+\|\Delta c\|_{L^4}^2)\big).
\end{split}
\end{equation}
Choosing $p=2$ in (\ref{181}), together with (\ref{128}) and Lemma \ref{b}, we obtain
\begin{equation}\label{125}
\int_{0}^{t}\|\nabla^2 n\|_{L^2}^2\, ds\leq Cc_{0}^2\exp\big\{Cc_{0}^2\exp\{C\Phi^3(\mathbf{v},m, t)t^{1/4}\}\big\}
\end{equation}
and
\begin{equation}\label{126}
\sup_{0\leq s\leq t}\|\nabla^2 n\|_{L^2}^2\leq Cc_{0}^2\Phi^2(\mathbf{v}, m, t)\exp\big\{Cc_{0}^2\exp\{C\Phi^3(\mathbf{v},m , t)t^{1/4}\}\big\}.
\end{equation}
Choosing $p=4$ in (\ref{181}), together with (\ref{128}), (\ref{125}), and Lemma \ref{b} gives
\begin{equation}\label{127}
\int_{0}^{t}\|\nabla^2 n\|_{L^4}^2\, ds\leq Cc_{0}^2\exp\big\{Cc_{0}^2\exp\{C\Phi^3(\mathbf{v}, m, t)t^{1/4}\}\big\}.
\end{equation}
The inequalities (\ref{129}) and (\ref{130}) follow directly from (\ref{125})-(\ref{127}).
\end{proof}

\begin{lemma}\label{g}
Let $\mathbf{u}$ be a smooth solution of the equation (\ref{4}). Then
\begin{equation}\label{131}
\sup_{0\leq s\leq t}(\|\sqrt{h}\mathbf{u}\|_{L^2}^2+\|\nabla\mathbf{u}\|_{L^2}^2)+\int_{0}^{t}\|\sqrt{h}\dot{\mathbf{u}}\|_{L^2}^2\, ds\leq Cc_{0}\exp\big\{Cc_{0}^2\exp\{C\Phi^2(\mathbf{v}, m, t)t^{1/4}\}\big\}.
\end{equation}
\end{lemma}

\begin{proof}
Multiplying the equation $(\ref{4})_{1}$ by $\mathbf{u}$ and integrating over $\Omega$ to get
\begin{equation}\label{183}
\begin{split}
&\frac{1}{2}\frac{d}{dt}\int h\mathbf{u}^2\, dx+\mu\int |\nabla\mathbf{u}|^2\, dx+(\mu+\lambda)\int (\dive\mathbf{u})^2\, dx\\
=&-\int h^2\mathbf{u}\cdot\nabla n\, dx-\frac{1}{2}\int (1+n)\mathbf{u}\cdot\nabla h^2\, dx\\
\leq&C(\|\sqrt{h}\mathbf{u}\|_{L^2}\|\nabla n\|_{L^2}\|h\|_{L^{\infty}}^{\frac{3}{2}}+\|\sqrt{h}\mathbf{u}\|_{L^2}\|\nabla h\|_{L^2}\|h\|_{L^{\infty}}^{\frac{1}{2}}+\|\sqrt{h}\mathbf{u}\|_{L^2}\|\nabla h\|_{L^q}\|n\|_{L^p})\\
\leq&C(\|h\|_{L^{\infty}}^3+\|\nabla h\|_{L^2}^2+\|\nabla h\|_{L^q}^2)\|\sqrt{h}\mathbf{u}\|_{L^2}^2+C(\|\nabla n\|_{L^2}^2+\|n\|_{L^2}^{2-\frac{4}{p}}\|\nabla n\|_{L^2}^{\frac{4}{p}}+\|h\|_{L^{\infty}}),
\end{split}
\end{equation}
where $q=4/\alpha, 1/p+1/q=1/2$. Applying the Gronwall's inequality to (\ref{183}), together with Lemmas \ref{c} and \ref{d}, we have
\begin{equation}\label{133}
\sup_{0\leq s\leq t}\|\sqrt{h}\mathbf{u}\|_{L^2}^2+\int_{0}^{t}\|\nabla\mathbf{u}\|_{L^2}^2\, ds\leq Cc_{0}\exp\big\{Cc_{0}^2\exp\{C\Phi^2(\mathbf{v}, m, t)t^{1/4}\}\big\}.
\end{equation}
Multiplying the equation $(\ref{4})_{1}$ by $\dot{\mathbf{u}}$ and integrating over $\Omega$ to get
\begin{equation}\label{184}
\begin{split}
&\frac{\mu}{2}\frac{d}{dt}\int |\nabla\mathbf{u}|^2\, dx+\frac{(\mu+\lambda)}{2}\frac{d}{dt}\int(\dive\mathbf{u})^2\, dx+\int h\dot{\mathbf{u}}^2\, dx\\
\leq&C(\|\sqrt{h}\dot{\mathbf{u}}\|_{L^2}\|\nabla n\|_{L^2}\|h\|_{L^{\infty}}^{\frac{3}{2}}+\|\sqrt{h}\dot{\mathbf{u}}\|_{L^2}\|\nabla h\|_{L^2}\|h\|_{L^{\infty}}^{\frac{1}{2}}\\
&+\|\sqrt{h}\dot{\mathbf{u}}\|_{L^2}\|\nabla h\|_{L^q}\|n\|_{L^p}\|h\|_{L^{\infty}}^{\frac{1}{2}}+\|\dive\mathbf{v}\|_{L^{\infty}}\|\nabla\mathbf{u}\|_{L^2}^2)\\
\leq&\varepsilon\|\sqrt{h}\dot{\mathbf{u}}\|_{L^2}^2+C_{\varepsilon}\|h\|_{L^{\infty}}(\|\nabla n\|_{L^2}^2\|h\|_{L^{\infty}}^2+\|\nabla h\|_{L^2}^2+\|\nabla h\|_{L^q}^2\|n\|_{L^p}^2)\\
&+C\|\dive\mathbf{v}\|_{L^{\infty}}\|\nabla\mathbf{u}\|_{L^2}^2,
\end{split}
\end{equation}
where $q=4/\alpha, 1/p+1/q=1/2$. Using the Gronwall's inequality to (\ref{184}), together with Lemmas \ref{c} and \ref{d}, we have
\begin{equation}\label{185}
\sup_{0\leq s\leq t}\|\nabla\mathbf{u}\|_{L^2}^2+\int_{0}^{t}\|\sqrt{h}\dot{\mathbf{u}}\|_{L^2}^2\, ds\leq Cc_{0}\exp\big\{Cc_{0}^2\exp\{C\Phi^2(\mathbf{v}, m, t)t^{1/4}\}\big\}.
\end{equation}
Combining the estimates (\ref{133}) and (\ref{185}), we complete the proof of (\ref{131}).
\end{proof}

\begin{lemma}\label{h}
Let $(h, \mathbf{u})$ be a smooth solution to the equation (\ref{1}) and (\ref{4}). Then
\begin{equation}\label{B300}
\sup_{0\leq s\leq t}\|\sqrt{h}\dot{\mathbf{u}}\|_{L^2}^2+\int_{0}^{t}\|\nabla\dot{\mathbf{u}}\|_{L^2}^2\, ds\leq Cc_{0}\exp\big\{Cc_{0}^2\exp\{C\Phi^3(\mathbf{v}, m, t)t^{1/4}\}\big\}
\end{equation}
and
\begin{equation}\label{B301}
\sup_{0\leq s\leq t}\|\nabla^2\mathbf{u}\|_{L^2}^2+\int_{0}^{t}\|\nabla\mathbf{u}_{t}\|_{L^2}^2\, ds\leq Cc_{0}^2\exp\big\{Cc_{0}^2\exp\{C\Phi^3(\mathbf{v}, m, t)t^{1/4}\}\big\}.
\end{equation}
\end{lemma}

\begin{proof}
Taking the operator $\partial_{t}+\partial_{i}(\mathbf{v}_{i}\cdot)$ on both sides of equation $(\ref{4})_{1}$, we have
\begin{equation}\label{B211}
\begin{split}
&h_{t}\dot{\mathbf{u}}+h\dot{\mathbf{u}}_{t}+h_{t}^2\nabla n+h^2\nabla n_{t}+\frac{1}{2}n_{t}\nabla h^2+\frac{1}{2}(1+n)\nabla h_{t}^2+\dive(h\dot{\mathbf{u}}\otimes\mathbf{v})\\
&+\dive(h^2\nabla n\otimes\mathbf{v})+\frac{1}{2}\dive\big((1+n)\nabla h^2\otimes\mathbf{v}\big)\\
=&\mu\Delta\mathbf{u}_{t}+(\mu+\lambda)\nabla(\dive\mathbf{u})_{t}+\mu\dive(\Delta\mathbf{u}\otimes\mathbf{v})+(\mu+\lambda)\dive\big(\nabla(\dive\mathbf{u})\otimes\mathbf{v}\big).
\end{split}
\end{equation}
Multiplying the equation (\ref{B211}) by $\dot{\mathbf{u}}$ and integrating over $\Omega$ to get
\begin{equation*}
\begin{split}
&\frac{1}{2}\frac{d}{dt}\int h\dot{\mathbf{u}}^2\, dx\\
=&-\int h_{t}^2\dot{\mathbf{u}}\cdot\nabla n\, dx-\int h^2\dot{\mathbf{u}}\cdot\nabla n_{t}\, dx-\frac{1}{2}\int n_{t}\dot{\mathbf{u}}\cdot\nabla h^2\, dx-\frac{1}{2}\int (1+n)\dot{\mathbf{u}}\cdot\nabla h_{t}^2\, dx\\
&-\int\dive(h^2\nabla n\otimes\mathbf{v})\dot{\mathbf{u}}\, dx-\frac{1}{2}\int\dive\big((1+n)\nabla h^2\otimes\mathbf{v}\big)\dot{\mathbf{u}}\, dx+\mu\int\Delta\mathbf{u}_{t}\dot{\mathbf{u}}\, dx\\
&+(\mu+\lambda)\int\nabla(\dive\mathbf{u})_{t}\dot{\mathbf{u}}\, dx+\mu\int\dive(\Delta\mathbf{u}\otimes\mathbf{v})\dot{\mathbf{u}}\, dx+(\mu+\lambda)\int\dive\big(\nabla(\dive\mathbf{u})\otimes\mathbf{v}\big)\dot{\mathbf{u}}\, dx\\
:=&\sum_{i=1}^{10}E_{i}.
\end{split}
\end{equation*}
In the following, we will estimate $E_{i}$ for each $i\ (i=1,\cdots,10)$. First we consider $E_{1}+E_{5}$. By H\"older's inequality and equation $(\ref{1})_{1}$, we have
\begin{equation*}
\begin{split}
E_{1}+E_{5}=&2\int h\dive(h\mathbf{v})\dot{\mathbf{u}}\cdot\nabla n\, dx+\int h^2\mathbf{v}\cdot\nabla n\dive\dot{\mathbf{u}}\, dx\\
=&\int h^2\dive\mathbf{v}\dot{\mathbf{u}}\cdot\nabla n\, dx-\int\mathbf{v} h^2\dot{\mathbf{u}}\cdot\nabla^2 n\, dx\\
\leq&C(\|\sqrt{h}\dot{\mathbf{u}}\|_{L^2}\|\nabla n\|_{L^2}\|h\|_{L^{\infty}}^{\frac{3}{2}}\|\dive\mathbf{v}\|_{L^{\infty}}+\|\sqrt{h}\dot{\mathbf{u}}\|_{L^2}\|\nabla^2 n\|_{L^2}\|h^{3/2}\mathbf{v}\|_{L^{\infty}}).
\end{split}
\end{equation*}
Then we estimate $E_{2}+E_{3}$. Using the integrating by parts and H\"older's inequality, we obtain
\begin{equation*}
\begin{split}
E_{2}+E_{3}=&\frac{1}{2}\int n_{t}\dot{\mathbf{u}}\cdot\nabla h^2\, dx+\int h^2\dive\dot{\mathbf{u}}n_{t}\, dx\\
\leq&C(\|\sqrt{h}\dot{\mathbf{u}}\|_{L^2}\|\nabla h\|_{L^q}\|n_{t}\|_{L^p}\|h\|_{L^{\infty}}^{\frac{1}{2}}+\|\dive\dot{\mathbf{u}}\|_{L^2}\|n_{t}\|_{L^2}\|h\|_{L^{\infty}}^2),
\end{split}
\end{equation*}
where $q=4/\alpha, 1/p+1/q=1/2$. Then we consider $E_{7}+E_{9}$. A straightforward computation shows that
\begin{equation*}
\begin{split}
E_{7}+E_{9}=&-\frac{\mu}{2}\int|\nabla\dot{\mathbf{u}}|^2\, dx+\mu\int\nabla(\mathbf{v}\cdot\nabla\mathbf{u})\nabla\dot{\mathbf{u}}\, dx-\mu\int\Delta\mathbf{u}\mathbf{v}\dive\dot{\mathbf{u}}\, dx\\
=&-\frac{\mu}{2}\int |\nabla\dot{\mathbf{u}}|^2\, dx+\mu\int\nabla\mathbf{v}\nabla\mathbf{u}\nabla\dot{\mathbf{u}}\, dx\\
\leq&-\frac{\mu}{2}\int |\nabla\dot{\mathbf{u}}|^2\, dx+C\|\nabla\mathbf{v}\|_{L^{\infty}}\|\nabla\mathbf{u}\|_{L^2}\|\nabla\dot{\mathbf{u}}\|_{L^2}.
\end{split}
\end{equation*}
Through the similar calculation as $E_{7}+E_{9}$, we can get the estimate of $E_{8}+E_{10}$. That is
\begin{equation*}
E_{8}+E_{10}\leq-\frac{(\mu+\lambda)}{2}\int(\dive\dot{\mathbf{u}})^2\, dx+C\|\nabla\mathbf{v}\|_{L^{\infty}}\|\nabla\mathbf{u}\|_{L^2}\|\nabla\dot{\mathbf{u}}\|_{L^2}.
\end{equation*}
Finally, we consider $E_{4}+E_{6}$. A straightforward computation shows that
\begin{equation*}
\begin{split}
E_{4}+E_{6}=&\int \dot{\mathbf{u}}\cdot\nabla n hh_{t}\, dx-\int (1+n)\dive\dot{\mathbf{u}}h^2\dive\mathbf{v}\, dx\\
=&-\int\dot{\mathbf{u}}\cdot\nabla nh\mathbf{v}\cdot\nabla h\, dx-\int\dot{\mathbf{u}}\cdot\nabla n h^2\dive\mathbf{v}\, dx-\int(1+n)\dive\dot{\mathbf{u}}h^2\dive\mathbf{v}\, dx\\
\leq&C(\|\sqrt{h}\dot{\mathbf{u}}\|_{L^2}\|\nabla h\|_{L^q}\|\nabla n\|_{L^p}\|\sqrt{h}\mathbf{v}\|_{L^{\infty}}+\|\sqrt{h}\dot{\mathbf{u}}\|_{L^2}\|\nabla n\|_{L^2}\|\dive\mathbf{v}h^{3/2}\|_{L^{\infty}}\\
&+\|\dive\dot{\mathbf{u}}\|_{L^2}\|\dive\mathbf{v}\|_{L^2}\|h^2(1+n)\|_{L^{\infty}}),
\end{split}
\end{equation*}
where $q=4/\alpha, 1/p+1/q=1/2$. Combining all these estimates about $E_{i}\ (i=1,\cdots,10)$, we have
\begin{equation}\label{186}
\begin{split}
&\frac{d}{dt}\|\sqrt{h}\dot{\mathbf{u}}\|_{L^2}^2+\mu\|\nabla\dot{\mathbf{u}}\|_{L^2}^2+(\mu+\lambda)\|\dive\dot{\mathbf{u}}\|_{L^2}^2\\
\leq&C(\|h\|_{L^{\infty}}^3\|\dive\mathbf{v}\|_{L^{\infty}}^2+\|h\|_{L^{\infty}}^3\|\mathbf{v}\|_{L^{\infty}}^2+\|h\|_{L^{\infty}}+\|h\|_{L^{\infty}}\|\mathbf{v}\|_{L^{\infty}}^2)\|\sqrt{h}\dot{\mathbf{u}}\|_{L^2}^2\\
&+C(\|\nabla n\|_{L^2}^2+\|\nabla^2 n\|_{L^2}^2+\|\nabla h\|_{L^q}^2\|n_{t}\|_{L^p}^2+\|\nabla h\|_{L^q}^2\|\nabla n\|_{L^p}^2)\\
&+\varepsilon\|\nabla\dot{\mathbf{u}}\|_{L^2}^2+C_{\varepsilon}(\|\dive\mathbf{v}\|_{L^2}^2\|h^2(1+n)\|_{L^{\infty}}^2+\|n_{t}\|_{L^2}^2\|h\|_{L^{\infty}}^4+\|\nabla\mathbf{u}\|_{L^2}^2\|\nabla\mathbf{v}\|_{L^{\infty}}^2).
\end{split}
\end{equation}
Choosing $\varepsilon=\mu/2$ and applying Gronwall's inequality to (\ref{186}), together with Lemmas \ref{c} and \ref{d}, we can get (\ref{B300}). According to the $L^p$ theory, Lemma \ref{c}, \ref{d}, equation (\ref{4}), and (\ref{B300}), we have
\begin{equation}\label{B303}
\begin{split}
\|\nabla^2\mathbf{u}\|_{L^2}^2\leq&C(\|h\|_{L^{\infty}}\|\sqrt{h}\dot{\mathbf{u}}\|_{L^2}^2+\|h\|_{L^{\infty}}^4\|\nabla n\|_{L^2}^2+\|(1+n)h\|_{L^{\infty}}^2\|\nabla h\|_{L^2}^2)\\
\leq&Cc_{0}^2\exp\big\{Cc_{0}^2\exp\{C\Phi^3(\mathbf{v}, m, t)t^{1/4}\}\big\}.
\end{split}
\end{equation}
Through the definition of material derivation, together with the estimate (\ref{B300}), (\ref{B303}), and Lemma \ref{g}, we have
\begin{equation*}
\begin{split}
\int_{0}^{t}\|\nabla\mathbf{u}_{t}\|_{L^2}^2\, ds\leq&\int_{0}^{t}\|\nabla\dot{\mathbf{u}}\|_{L^2}^2\, ds+\int_{0}^{t}\|\nabla\mathbf{v}\|_{L^{\infty}}^2\|\nabla\mathbf{u}\|_{L^2}^2\, ds+\int_{0}^{t}\|\mathbf{v}\|_{L^{\infty}}^2\|\nabla^2\mathbf{u}\|_{L^2}^2\, ds\\
\leq&Cc_{0}^2\exp\big\{Cc_{0}^2\exp\{C\Phi^3(\mathbf{v}, m, t)t^{1/4}\}\big\}.
\end{split}
\end{equation*}
The above estimate combined with (\ref{B303}) yields (\ref{B301}).
\end{proof}

\begin{lemma}\label{i}
Let $h$ be as in Lemma \ref{c}. Then
\begin{equation}\label{136}
\||x|^{\alpha/2}h^2\|_{L^2}^2\leq Cc_{0}\exp\{C\Phi(\mathbf{v}, m, t)t^{1/2}\}.
\end{equation}
\end{lemma}

\begin{proof}
Multiplying the equation (\ref{5}) by $|x|^{\alpha}h^2$ and integrating over $\Omega$, using Lemma \ref{f} to get
\begin{equation}\label{187}
\begin{split}
\frac{1}{2}\frac{d}{dt}\int |x|^{\alpha}h^4\, dx=&-\alpha\int |x|^{\alpha-2}x_{i}\cdot\mathbf{v}h^4\, dx-\int |x|^{\alpha}\dive\mathbf{v}h^4\, dx\\
\leq&C(\||x|^{\alpha/2-1}\mathbf{v}\|_{L^2}\||x|^{\alpha/2}h^2\|_{L^2}+\|\dive\mathbf{v}\|_{L^{\infty}}\||x|^{\alpha/2}h^2\|_{L^2}^2)\\
\leq&C(1+\|\dive\mathbf{v}\|_{L^{\infty}})\||x|^{\alpha/2}h^2\|_{L^2}^2+C\||x|^{\alpha/2}\nabla\mathbf{v}\|_{L^2}^2.
\end{split}
\end{equation}
Using the Gronwall's inequality to (\ref{187}), we have
\begin{equation*}
\sup_{0\leq s\leq t}\||x|^{\alpha/2}h^2\|_{L^2}^2\leq Cc_{0}\exp\{C\Phi(\mathbf{v},m, t)t^{1/2}\},
\end{equation*}
which yields the inequality (\ref{136}).
\end{proof}

\begin{lemma}\label{e}
Let $c$ be as in Lemma \ref{b}. Then
\begin{equation}\label{z101}
\sup_{0\leq s\leq t}\||x|^{\alpha/2}c\|_{L^2}^2+\int_{0}^{t}\||x|^{\alpha/2}\nabla c\|_{L^2}^2\, ds\leq Cc_{0}\exp\{C\Phi^2(\mathbf{v}, m, t)t^{1/2}\}.
\end{equation}
\end{lemma}

\begin{proof}
Multiplying the equation $(\ref{2})_{1}$ by $|x|^{\alpha}c$ and integrating over $\Omega$, together with Lemma \ref{f} to get
\begin{equation}\label{188}
\begin{split}
&\frac{1}{2}\frac{d}{dt}\int |x|^{\alpha}c^2\, dx+\int |x|^{\alpha}|\nabla c|^2\, dx\\
=&-\int |x|^{\alpha}c\dive(c\mathbf{v})\, dx-\alpha\int|x|^{\alpha-2}x_{i}\cdot\nabla c c\, dx-\int |x|^{\alpha}mc^2\, dx\\
\leq&C(\||x|^{\alpha/2-1}\mathbf{v}\|_{L^2}\||x|^{\alpha/2}c\|_{L^2}\|c\|_{L^{\infty}}+\|\dive\mathbf{v}\|_{L^{\infty}}\||x|^{\alpha/2}c\|_{L^2}^2\\
&+\alpha\||x|^{\alpha/2-1}c\|_{L^2}\||x|^{\alpha/2}\nabla c\|_{L^2}+\|m\|_{L^{\infty}}\||x|^{\alpha/2}c\|_{L^2}^2)\\
\leq&C(\|\dive\mathbf{v}\|_{L^{\infty}}+\|m\|_{L^{\infty}}+\||x|^{\alpha/2}\nabla\mathbf{v}\|_{L^2}^2)\||x|^{\alpha/2}c\|_{L^2}^2\\
&+C\|c\|_{L^{\infty}}^2+\frac{\alpha^2}{2}\||x|^{\alpha/2}\nabla c\|_{L^2}^2.
\end{split}
\end{equation}
Using the Gronwall's inequality to (\ref{188}), together with Lemma \ref{b} and (\ref{200}) gives (\ref{z101}).
\end{proof}

\begin{lemma}\label{D1}
Let $n$ be as in Lemma \ref{d}. Then
\begin{equation}\label{190}
\begin{split}
&\sup_{0\leq s\leq t}\||x|^{\alpha/2}n\|_{L^2}^2+\int_{0}^{t}\||x|^{\alpha/2}\nabla n\|_{L^2}^2\, ds\\
\leq&Cc_{0}^2\Phi(\mathbf{v}, m, t)\exp\big\{Cc_{0}^{2}\exp\{C\Phi^3(\mathbf{v}, m, t)t^{1/4}\}\big\}.
\end{split}
\end{equation}
\end{lemma}

\begin{proof}
The proof of (\ref{190}) is similar to (\ref{z101}). It remains to deal with the term $-\int |x|^{\alpha}n\nabla\cdot(n\nabla c)\, dx$. A straightforward computation shows that
\begin{equation}\label{B207}
\begin{split}
&-\int |x|^{\alpha}n\nabla\cdot(n\nabla c)\, dx\\
\leq&\alpha\||x|^{\alpha/2-1}n\|_{L^2}\||x|^{\alpha/2}\nabla c\|_{L^2}\|n\|_{L^{\infty}}+\||x|^{\alpha/2}\nabla n\|_{L^2}\||x|^{\alpha/2}\nabla c\|_{L^2}\|n\|_{L^{\infty}}\\
\leq&(\frac{\alpha^4}{4}+\varepsilon)\||x|^{\alpha/2}\nabla n\|_{L^2}^2+C_{\varepsilon}\||x|^{\alpha/2}\nabla c\|_{L^2}^2\|n\|_{L^{\infty}}^2.
\end{split}
\end{equation}
Through the estimate (\ref{B207}) and Lemma \ref{e}, we obtain (\ref{190}).
\end{proof}

\begin{lemma}\label{A1}
Let $c$ be as in Lemma \ref{b}. Then
\begin{equation}\label{z103}
\begin{split}
&\sup_{0\leq s\leq t}(\||x|^{\alpha/2}\dot{c}\|_{L^2}^2+\||x|^{\alpha/2}\nabla c\|_{L^2}^2)+\int_{0}^{t}\||x|^{\alpha/2}\nabla\dot{c}\|_{L^2}^2\, ds\\
\leq&Cc_{0}\exp\big\{Cc_{0}^2\exp\{C\Phi^2(\mathbf{v}, m, t)t^{1/4}\}\big\}.
\end{split}
\end{equation}
\end{lemma}

\begin{proof}
Multiplying the equation $(\ref{2})_{1}$ by $|x|^{\alpha}\dot{c}$ and integrating over $\Omega$, together with Lemma \ref{f} to get
\begin{equation}\label{191}
\begin{split}
&\frac{1}{2}\frac{d}{dt}\int |x|^{\alpha}|\nabla c|^2\, dx+\int |x|^{\alpha}\dot{c}^2\, dx\\
=&-\int |x|^{\alpha}c\dive\mathbf{v}\dot{c}\, dx-\alpha\int |x|^{\alpha-2}x_{i}\cdot\nabla c\dot{c}\, dx+\frac{\alpha}{2}\int |x|^{\alpha-2}x_{i}\cdot\mathbf{v}|\nabla c|^2\, dx\\
&+\frac{1}{2}\int |x|^{\alpha}\dive\mathbf{v}|\nabla c|^2\, dx-\int |x|^{\alpha}mc\dot{c}\, dx\\
\leq&\||x|^{\alpha/2}\dot{c}\|_{L^2}\||x|^{\alpha/2}c\|_{L^2}\|\dive\mathbf{v}\|_{L^{\infty}}+\alpha\||x|^{\alpha/2-1}\dot{c}\|_{L^2}\||x|^{\alpha/2}\nabla c\|_{L^2}\\
&+\||x|^{\alpha/2-1}\mathbf{v}\|_{L^2}\||x|^{\alpha/2}\nabla c\|_{L^2}\|\nabla c\|_{L^{\infty}}+\|\dive\mathbf{v}\|_{L^{\infty}}\||x|^{\alpha/2}\nabla c\|_{L^2}^2\\
&+\||x|^{\alpha/2}c\|_{L^2}\||x|^{\alpha/2}\dot{c}\|_{L^2}\|m\|_{L^{\infty}}\\
\leq&\varepsilon\||x|^{\alpha/2}\dot{c}\|_{L^2}^2+C_{\varepsilon}\||x|^{\alpha/2}c\|_{L^2}^2(\|\dive\mathbf{v}\|_{L^{\infty}}^2+\|m\|_{L^{\infty}}^2)+(\alpha^4/4)\||x|^{\alpha/2}\nabla\dot{c}\|_{L^2}^2\\
&+C(\||x|^{\alpha/2}\nabla\mathbf{v}\|_{L^2}^2+\|\dive\mathbf{v}\|_{L^{\infty}}+1)\||x|^{\alpha/2}\nabla c\|_{L^2}^2+C\|\nabla c\|_{L^{\infty}}^2.
\end{split}
\end{equation}
Using the operator $\partial_{t}+\partial_{i}(\mathbf{v}_{i}\cdot)$ to equation (\ref{2}), we get
\begin{equation}\label{z102}
\begin{split}
&\dot{c}_{t}+(c\dive\mathbf{v})_{t}+\dive(\dot{c}\otimes\mathbf{v})+\dive(c\dive\mathbf{v}\otimes\mathbf{v})\\
=&\Delta c_{t}-(mc)_{t}+\dive(\Delta c\otimes\mathbf{v})-\dive(mc\otimes\mathbf{v}).
\end{split}
\end{equation}
Multiplying the equation (\ref{z102}) by $|x|^{\alpha}\dot{c}$ and integrating over $\Omega$ to get
\begin{equation*}
\begin{split}
&\frac{1}{2}\frac{d}{dt}\int |x|^{\alpha}\dot{c}^2\, dx+\int |x|^{\alpha}|\nabla\dot{c}|^2\, dx\\
=&-\int |x|^{\alpha}\dot{c}(c\dive\mathbf{v})_{t}\, dx-\int |x|^{\alpha}\dot{c}\dive(\dot{c}\otimes\mathbf{v})\, dx-\int |x|^{\alpha}\dot{c}\dive(c\dive\mathbf{v}\otimes\mathbf{v})\, dx\\
&-\alpha\int |x|^{\alpha-2}x_{i}\cdot\nabla\dot{c}\dot{c}\, dx-\int |x|^{\alpha}\dot{c}\Delta(\mathbf{v}\cdot\nabla c)\, dx-\int |x|^{\alpha}\dot{c}(mc)_{t}\, dx\\
&+\int |x|^{\alpha}\dot{c}\dive(\Delta c\otimes\mathbf{v})\, dx-\int |x|^{\alpha}\dot{c}\dive(mc\otimes\mathbf{v})\, dx\\
:=&\sum_{i=1}^{8}F_{i}.
\end{split}
\end{equation*}
In the following, we estimate $F_{i}\ (i=1, \cdots, 8)$. First, we consider $F_{1}$. By H\"older's inequality, Gagliardo-Nirenberg's inequality and Young's inequality, we have
\begin{equation*}
\begin{split}
F_{1}\leq&\||x|^{\alpha/2}\dot{c}\|_{L^{4}}\||x|^{\alpha/2}\dive\mathbf{v}\|_{L^2}\|c_{t}\|_{L^4}+\||x|^{\alpha/2}\dot{c}\|_{L^4}\||x|^{\alpha/2}c\|_{L^4}\|\dive\mathbf{v}_{t}\|_{L^2}\\
\leq&(\varepsilon/5)\||x|^{\alpha/2}\nabla\dot{c}\|_{L^2}^2+C_{\varepsilon}(\|\nabla c_{t}\|_{L^2}+\|\dive\mathbf{v}_{t}\|_{L^2})\||x|^{\alpha/2}\dot{c}\|_{L^2}^2\\
&+C_{\varepsilon}\||x|^{\alpha/2}\dive\mathbf{v}\|_{L^2}^2\|c_{t}\|_{L^2}\|\nabla c_{t}\|_{L^2}^{\frac{1}{2}}+C\||x|^{\alpha/2}c\|_{L^2}\||x|^{\alpha/2}\nabla c\|_{L^2}\|\dive\mathbf{v}_{t}\|_{L^2}^{\frac{3}{2}}.
\end{split}
\end{equation*}
Then we consider $F_{2}$. A straightforward computation shows that
\begin{equation*}
\begin{split}
F_{2}\leq&\alpha\||x|^{\alpha/2-1}\dot{c}\|_{L^2}\||x|^{\alpha/2}\dot{c}\|_{L^2}\|\mathbf{v}\|_{L^{\infty}}+\||x|^{\alpha/2}\nabla\dot{c}\|_{L^2}\||x|^{\alpha/2}\dot{c}\|_{L^2}\|\mathbf{v}\|_{L^{\infty}}\\
\leq&(\varepsilon/5)\||x|^{\alpha/2}\nabla\dot{c}\|_{L^2}^2+C_{\varepsilon}\||x|^{\alpha/2}\dot{c}\|_{L^2}^2\|\mathbf{v}\|_{L^{\infty}}^2.
\end{split}
\end{equation*}
The nonlinear terms $F_{3}$ and $F_{4}$ can be calculated as follow
\begin{equation*}
\begin{split}
F_{3}\leq&\||x|^{\alpha/2-1}\dot{c}\|_{L^2}\||x|^{\alpha/2}\dive\mathbf{v}\|_{L^2}\|\mathbf{v}c\|_{L^{\infty}}+\||x|^{\alpha/2}\dive\mathbf{v}\|_{L^2}\||x|^{\alpha/2}\nabla\dot{c}\|_{L^2}\|c\mathbf{v}\|_{L^{\infty}}\\
\leq&(\varepsilon/5)\||x|^{\alpha/2}\nabla\dot{c}\|_{L^2}^2+C_{\varepsilon}\||x|^{\alpha/2}\dive\mathbf{v}\|_{L^2}^2\|c\mathbf{v}\|_{L^{\infty}}^2
\end{split}
\end{equation*}
and
\begin{equation*}
F_{4}\leq\alpha\||x|^{\alpha/2-1} \dot{c}\|_{L^2}\||x|^{\alpha/2}\nabla\dot{c}\|_{L^2}\leq\frac{\alpha^2}{2}\||x|^{\alpha/2}\nabla\dot{c}\|_{L^2}^2.
\end{equation*}
Using the H\"older's inequality, Gagliardo-Nirenberg's inequality and Young's inequality, the nonlinear term $F_{6}$ can be estimated as follow
\begin{equation*}
\begin{split}
F_{6}=&-\int |x|^{\alpha}\dot{c}m_{t}c\, dx-\int |x|^{\alpha}\dot{c}^2 m\, dx+\int |x|^{\alpha}\dot{c}m\mathbf{v}\cdot\nabla c\, dx\\
\leq&\||x|^{\alpha/2}\dot{c}\|_{L^2}\||x|^{\alpha/2}c\|_{L^4}\|m_{t}\|_{L^4}+\|m\|_{L^{\infty}}\||x|^{\alpha/2}\dot{c}\|_{L^2}^2+\||x|^{\alpha/2}\dot{c}\|_{L^2}\||x|^{\alpha/2}\nabla c\|_{L^2}\|m\mathbf{v}\|_{L^{\infty}}\\
\leq&C(1+\|m\|_{L^{\infty}}^2+\|\nabla m_{t}\|_{L^2})\||x|^{\alpha/2}\dot{c}\|_{L^2}^2+C\||x|^{\alpha/2}c\|_{L^2}^2\|m_{t}\|_{L^2}\\
&+C(\||x|^{\alpha/2}\nabla c\|_{L^2}^2\|\nabla m_{t}\|_{L^2}+\||x|^{\alpha/2}\nabla c\|_{L^2}^2\|\mathbf{v}\|_{L^{\infty}}^2).
\end{split}
\end{equation*}
For the nonlinear term $F_{8}$, a straightforward computation shows that
\begin{equation*}
F_{8}\leq\||x|^{\alpha/2-1}\dot{c}\|_{L^2}\||x|^{\alpha/2}c\|_{L^{2}}\|m\mathbf{v}\|_{L^{\infty}}\leq(\varepsilon/5)\||x|^{\alpha/2}\nabla\dot{c}\|_{L^2}^2+C_{\varepsilon}\||x|^{\alpha/2}c\|_{L^2}^2\|m\mathbf{v}\|_{L^{\infty}}^2.
\end{equation*}
The nonlinear term $F_{5}+F_{7}$ can be estimated as
\begin{equation*}
\begin{split}
F_{5}+F_{7}\leq&\||x|^{\alpha/2}\nabla\dot{c}\|_{L^2}\||x|^{\alpha/2}\nabla c\|_{L^2}\|\nabla\mathbf{v}\|_{L^{\infty}}\\
\leq&(\varepsilon/5)\||x|^{\alpha/2}\nabla\dot{c}\|_{L^2}^2+C_{\varepsilon}\||x|^{\alpha/2}\nabla c\|_{L^2}^2\|\nabla\mathbf{v}\|_{L^{\infty}}^2.
\end{split}
\end{equation*}
Combining all these estimates about $F_{i}\ (i=1,\cdots, 8)$ to get
\begin{equation}\label{140}
\begin{split}
&\frac{d}{dt}\||x|^{\alpha/2}\dot{c}\|_{L^2}^2+\||x|^{\alpha/2}|\nabla\dot{c}|\|_{L^2}^2\\
\leq&C_{\varepsilon}(\|\nabla c_{t}\|_{L^2}+\|\dive\mathbf{v}_{t}\|_{L^2}+\|\mathbf{v}\|_{L^{\infty}}^2+\|m\|_{L^{\infty}}^2+\|\nabla m_{t}\|_{L^2}+1)\||x|^{\alpha/2}\dot{c}\|_{L^2}^2\\
&+(\varepsilon+\frac{\alpha^2}{2})\||x|^{\alpha/2}\nabla\dot{c}\|_{L^2}^2+\||x|^{\alpha/2}\dive\mathbf{v}\|_{L^2}^2\|c_{t}\|_{L^2}\|\nabla c_{t}\|_{L^2}^{\frac{1}{2}}+\||x|^{\alpha/2}c\|_{L^2}^2\|\dive\mathbf{v}_{t}\|_{L^2}^{\frac{3}{2}}\\
&+\||x|^{\alpha/2}\dive\mathbf{v}\|_{L^2}^2\|c\mathbf{v}\|_{L^{\infty}}^2+\||x|^{\alpha/2}c\|_{L^2}^2\|m_{t}\|_{L^2}^2+\||x|^{\alpha/2}c\|_{L^2}^2\|m\mathbf{v}\|_{L^{\infty}}^2\\
&+(\|\dive\mathbf{v}_{t}\|_{L^2}^{\frac{3}{2}}+\|\nabla m_{t}\|_{L^2}^2+\|\mathbf{v}\|_{L^{\infty}}^2+\|\nabla\mathbf{v}\|_{L^{\infty}}^2)\||x|^{\alpha/2}\nabla c\|_{L^2}^2.
\end{split}
\end{equation}
Adding the inequalities (\ref{191}) and (\ref{140}), choosing $\varepsilon=1/2$, together with the fact
\begin{equation*}
\|\nabla c\|_{L^{\infty}}^2\leq\|c\|_{L^4}^{\frac{1}{2}}\|\nabla^2 c\|_{L^4}^{\frac{3}{2}}\leq\|c\|_{L^2}^{\frac{1}{4}}\|\nabla c\|_{L^2}^{\frac{1}{4}}\|\nabla^2 c\|_{L^4}^{\frac{3}{2}},
\end{equation*}
we have
\begin{equation}\label{192}
\begin{split}
&\frac{d}{dt}\big(\||x|^{\alpha/2}\dot{c}\|_{L^2}^2+\||x|^{\alpha/2}\nabla c\|_{L^2}^2\big)+(\||x|^{\alpha/2}\nabla\dot{c}\|_{L^2}^2+\||x|^{\alpha/2}\dot{c}\|_{L^2}^2)\\
\leq&C(\|\nabla c_{t}\|_{L^2}+\|\dive\mathbf{v}_{t}\|_{L^2}^{\frac{3}{2}}+\|\mathbf{v}\|_{W^{1, \infty}}^2+\|m\|_{L^{\infty}}^2+\|m_{t}\|_{H^1}\\
&+\||x|^{\alpha/2}\nabla\mathbf{v}\|_{L^2}^2+1)(\||x|^{\alpha/2}\dot{c}\|_{L^2}^2+\||x|^{\alpha/2}\nabla c\|_{L^2}^2)\\
&+C(\|\dive\mathbf{v}_{t}\|_{L^2}^{\frac{3}{2}}+\|m_{t}\|_{L^2}+\|m\mathbf{v}\|_{L^{\infty}}^2+\|\dive\mathbf{v}\|_{L^{\infty}}^2)\||x|^{\alpha/2} c\|_{L^2}^2\\
&+C(\|c_{t}\|_{L^2}\|\nabla c_{t}\|_{L^2}^{\frac{1}{2}}+\|c\mathbf{v}\|_{L^{\infty}}^2)\||x|^{\alpha/2}\dive\mathbf{v}\|_{L^2}^2+C\|\nabla c\|_{L^{\infty}}^2.
\end{split}
\end{equation}
Applying Gronwall's inequality to (\ref{192}), together with Lemmas \ref{b} and \ref{e} yields (\ref{z103}).
\end{proof}

\begin{lemma}\label{H}
Let $n$ be as in Lemma \ref{d}. Then
\begin{equation}\label{z106}
\begin{split}
&\sup_{0\leq s\leq t}(\||x|^{\alpha/2}\dot{n}\|_{L^2}^2+\||x|^{\alpha/2}\nabla n\|_{L^2}^2)+\int_{0}^{t}\||x|^{\alpha/2}\nabla\dot{n}\|_{L^2}^2\, ds\\
\leq&Cc_{0}^2\exp\big\{Cc_{0}^2\exp\{C\Phi^3(\mathbf{v}, m, t)t^{1/5}\}\big\}.
\end{split}
\end{equation}
\end{lemma}

\begin{proof}
Multiplying the equation $(\ref{3})_{1}$ by $|x|^{\alpha}\dot{n}$ and integrating over $\Omega$ to get
\begin{equation}\label{193}
\begin{split}
&\frac{1}{2}\frac{d}{dt}\int |x|^{\alpha}|\nabla n|^2\, dx+\int |x|^{\alpha}\dot{n}^2\, dx\\
=&-\int |x|^{\alpha}\dot{n} n\dive\mathbf{v}\, dx-\alpha\int |x|^{\alpha-2}x_{i}\cdot\nabla n\dot{n}\, dx+\int |x|^{\alpha}\mathbf{v}\cdot\nabla n\Delta n\, dx\\
&-\int |x|^{\alpha}\dot{n}\nabla\cdot(n\nabla c)\, dx.
\end{split}
\end{equation}
The terms $-\int |x|^{\alpha}\dot{n} n\dive\mathbf{v}\, dx, -\alpha\int |x|^{\alpha-2}x_{i}\cdot\nabla n\dot{n}\, dx$ and $\int |x|^{\alpha}\mathbf{v}\cdot\nabla n\Delta n\, dx$ can be estimated as (\ref{191}) in Lemma \ref{A1}. We only give the estimate of the nonlinear term $-\int |x|^{\alpha}\dot{n}\nabla\cdot(n\nabla c)\, dx$. By H\"older's inequality and Cauchy's inequality, we have
\begin{equation}\label{B209}
\begin{split}
&-\int |x|^{\alpha}\dot{n}\nabla\cdot(n\nabla c)\, dx\\
\leq&\||x|^{\alpha/2}\dot{n}\|_{L^2}\||x|^{\alpha/2}\nabla n\|_{L^2}\|\nabla c\|_{L^{\infty}}+\||x|^{\alpha/2}\dot{n}\|_{L^2}\||x|^{\alpha/2}n\|_{L^4}\|\Delta c\|_{L^4}\\
\leq&\frac{\varepsilon}{2}\||x|^{\alpha/2}\dot{n}\|_{L^2}^2+C_{\varepsilon}(\|\nabla c\|_{L^{\infty}}^2+\|\Delta c\|_{L^4}^2)\||x|^{\alpha/2}\nabla n\|_{L^2}^2+C\||x|^{\alpha/2}n\|_{L^2}^2\|\Delta c\|_{L^4}^2.
\end{split}
\end{equation}
Using the operator $\partial_{t}+\partial_{i}(\mathbf{v}_{i}\cdot)$ on equation $(\ref{3})_{1}$, multiplying the resulting equation by $|x|^{\alpha}\dot{n}$, integrating over $\Omega$ to get
\begin{equation}\label{B210}
\begin{split}
&\frac{1}{2}\frac{d}{dt}\int|x|^{\alpha}\dot{n}^2\, dx+\int |x|^{\alpha}|\nabla\dot{n}|^2\, dx\\
=&-\int |x|^{\alpha}\dot{n}(n\dive\mathbf{v})_{t}\, dx-\int |x|^{\alpha}\dot{n}\dive(\dot{n}\otimes\mathbf{v})\, dx-\int |x|^{\alpha}\dot{n}\dive(n\dive\mathbf{v}\otimes\mathbf{v})\, dx\\
&-\alpha\int |x|^{\alpha-2}x_{i}\cdot\nabla\dot{n}\dot{n}\, dx-\int |x|^{\alpha}\Delta(\mathbf{v}\cdot\nabla n)\dot{n}\, dx-\int |x|^{\alpha}\dot{n}\nabla\cdot(n\nabla c)_{t}\, dx\\
&+\int |x|^{\alpha}\dot{n}\dive(\Delta n\otimes\mathbf{v})\, dx-\int |x|^{\alpha}\dot{n}\dive(\nabla\cdot(n\nabla c)\otimes\mathbf{v})\, dx\\
:=&\sum_{i=1}^8G_{i}.
\end{split}
\end{equation}
The estimates of $G_{1}$ to $G_{4}$, $G_{5}+G_{7}$ are similar as $F_{1}$ to $F_{4}$ and $F_{5}+F_{7}$ in Lemma \ref{A1}. Here we omit. For the nonlinear term $G_{6}$, by H\"older's inequality, Gagliardo-Nirenberg's inequality and Young's inequality, we have
\begin{equation*}
\begin{split}
G_{6}=&\alpha\int |x|^{\alpha-2}x_{i}\cdot\nabla c\dot{n}^2\, dx-\alpha\int |x|^{\alpha-2}x_{i}\cdot\nabla c\dot{n}\mathbf{v}\cdot\nabla n\, dx+\alpha\int |x|^{\alpha-2}x_{i}\cdot\nabla c_{t}n\dot{n}\, dx\\
&+\int |x|^{\alpha}\nabla\dot{n}\cdot\nabla c\dot{n}\, dx-\int |x|^{\alpha}\nabla\dot{n}\cdot\nabla c\mathbf{v}\cdot\nabla n\, dx+\int |x|^{\alpha}\nabla\dot{n}\cdot\nabla c_{t}n\, dx\\
\leq&\||x|^{\alpha/2-1}\dot{n}\|_{L^2}\||x|^{\alpha/2}\dot{n}\|_{L^2}\|\nabla c\|_{L^{\infty}}+\||x|^{\alpha/2-1}\dot{n}\|_{L^2}\||x|^{\alpha/2}\nabla n\|_{L^2}\|\mathbf{v}\cdot\nabla c\|_{L^{\infty}}\\
&+\alpha\||x|^{\alpha/2-1}\dot{n}\|_{L^2}\||x|^{\alpha/2}n\|_{L^4}\|\nabla c_{t}\|_{L^4}+\||x|^{\alpha/2}\nabla\dot{n}\|_{L^2}\||x|^{\alpha/2}\dot{n}\|_{L^2}\|\nabla c\|_{L^{\infty}}\\
&+\||x|^{\alpha/2}\nabla\dot{n}\|_{L^2}\||x|^{\alpha/2}\nabla n\|_{L^2}\|\mathbf{v}\cdot\nabla c\|_{L^{\infty}}+\||x|^{\alpha/2}\nabla\dot{n}\|_{L^2}\||x|^{\alpha/2}n\|_{L^4}\|\nabla c_{t}\|_{L^4}\\
\leq&(\frac{\varepsilon}{5}+\frac{3}{4}\alpha^2)\||x|^{\alpha/2}\nabla\dot{n}\|_{L^2}^2+C_{\varepsilon}\||x|^{\alpha/2}\dot{n}\|_{L^2}^2\|\nabla c\|_{L^{\infty}}^2+C\||x|^{\alpha/2}\nabla n\|_{L^2}^2(\|\mathbf{v}\cdot\nabla c\|_{L^{\infty}}^2+\|\nabla c_{t}\|_{L^4}^2).
\end{split}
\end{equation*}
For the nonlinear term $G_{8}$, a straightforward computation shows that
\begin{equation*}
\begin{split}
G_{8}=&\alpha\int |x|^{\alpha-2}x_{i}\cdot\nabla n\mathbf{v}\cdot\nabla c\dot{n}\, dx+\alpha\int |x|^{\alpha-2}x_{i}\cdot\mathbf{v}\dot{n}n\Delta c\, dx+\int |x|^{\alpha}\nabla\dot{n}\nabla n\nabla c\mathbf{v}\, dx\\
&+\int |x|^{\alpha}\nabla\dot{n}n\Delta c\mathbf{v}\, dx\\
\leq&(\alpha^2/2+1)\||x|^{\alpha/2}\nabla\dot{n}\|_{L^2}(\||x|^{\alpha/2}\nabla n\|_{L^2}\|\mathbf{v}\cdot\nabla c\|_{L^{\infty}}+\||x|^{\alpha/2}n\|_{L^4}\|\Delta c\|_{L^4}\|\mathbf{v}\|_{L^{\infty}})\\
\leq&(\alpha^4/4+\varepsilon/5)\||x|^{\alpha/2}\nabla\dot{n}\|_{L^2}^2+C_{\varepsilon}\||x|^{\alpha/2}\nabla n\|_{L^2}^2(\|\mathbf{v}\cdot\nabla c\|_{L^{\infty}}^2+\|\Delta c\|_{L^4}^2\|\mathbf{v}\|_{L^{\infty}}^2).
\end{split}
\end{equation*}
Putting all these estimates of $G_{i}$ into (\ref{B210}), combining the estimate of (\ref{193}), then choosing $\varepsilon=1/2$, we have
\begin{equation}\label{194}
\begin{split}
&\frac{d}{dt}\big(\||x|^{\alpha/2}\dot{n}\|_{L^2}^2+\||x|^{\alpha/2}\nabla n\|_{L^2}^2\big)+(\||x|^{\alpha/2}\nabla\dot{n}\|_{L^2}^2+\||x|^{\alpha/2}\dot{n}\|_{L^2}^2)\\
\leq&C\big(1+\|\dive\mathbf{v}_{t}\|_{L^2}^{\frac{3}{2}}+\|\mathbf{v}\|_{W^{1, \infty}}^2+(\|\Delta c\|_{L^4}^2+\|\nabla c\|_{L^{\infty}}^2)(1+\|\mathbf{v}\|_{L^{\infty}}^2)+\|\nabla c_{t}\|_{L^4}^2\\
&+\||x|^{\alpha/2}\nabla\mathbf{v}\|_{L^2}^2\big)(\||x|^{\alpha/2}\dot{n}\|_{L^2}^2+\||x|^{\alpha/2}\nabla n\|_{L^2}^2)\\
&+C(\|\dive\mathbf{v}_{t}\|_{L^2}^{\frac{3}{2}}+\|\Delta c\|_{L^4}^2+\|\dive\mathbf{v}\|_{L^{\infty}}^2)\||x|^{\alpha/2}n\|_{L^2}^2+C\||x|^{\alpha/2}\dive\mathbf{v}\|_{L^2}^2\|n\mathbf{v}\|_{L^{\infty}}^2\\
&+C(\||x|^{\alpha/2}\nabla\dot{c}\|_{L^2}^2+\|\nabla n\|_{L^{\infty}}^2+\|n_{t}\|_{L^2}\|\nabla n_{t}\|_{L^2}\||x|^{\alpha/2}\dive\mathbf{v}\|_{L^2}^2).
\end{split}
\end{equation}
Applying Gronwall's inequality to (\ref{194}), together with Lemmas \ref{b}-\ref{d}, \ref{D1}-\ref{A1} yields (\ref{z106}).
\end{proof}

\begin{lemma}\label{j}
Let $(h, \mathbf{u})$ be as in Lemma \ref{h}. For $q=4/\alpha$, we have
\begin{equation}\label{B304}
\begin{split}
&\sup_{0\leq s\leq t}(\||x|^{\alpha/2}\nabla\mathbf{u}\|_{L^2}^2+\||x|^{\alpha/2}\sqrt{h}\dot{\mathbf{u}}\|_{L^2}^2)+\int_{0}^{t}\||x|^{\alpha/2}\nabla\dot{\mathbf{u}}\|_{L^2}^2\, ds\\
\leq&Cc_{0}^2\exp\big\{Cc_{0}^2\exp\{C\Phi^3(\mathbf{v}, m, t)t^{1/5}\}\big\}
\end{split}
\end{equation}
and
\begin{equation}\label{B305}
\int_{0}^{t}(\|\mathbf{u}_{t}\|_{L^q}^2+\|\nabla^2\mathbf{u}\|_{L^q}^2)\, ds\leq Cc_{0}^2\exp\big\{Cc_{0}^2\exp\{C\Phi^3(\mathbf{v}, m, t)t^{1/5}\}\big\}.
\end{equation}
\end{lemma}

\begin{proof}
Multiplying the equation $(\ref{4})_{1}$ by $|x|^{\alpha}\dot{\mathbf{u}}$ and integrating over $\Omega$ to get
\begin{equation}\label{B212}
\begin{split}
&\int |x|^{\alpha}h\dot{\mathbf{u}}^2\, dx\\
=&-\frac{1}{2}\int |x|^{\alpha}\dot{\mathbf{u}}\cdot\nabla\big((1+n)h^2\big)\, dx-\frac{1}{2}\int |x|^{\alpha}h^2\dot{\mathbf{u}}\cdot\nabla n\, dx+\mu\int |x|^{\alpha}\Delta\mathbf{u}\cdot\dot{\mathbf{u}}\, dx\\
&+(\mu+\lambda)\int |x|^{\alpha}\nabla(\dive\mathbf{u})\cdot\dot{\mathbf{u}}\, dx\\
:=&\sum_{i=1}^{4}H_{i}.
\end{split}
\end{equation}
In the following, we estimate $H_{i}\ (i=1, \cdots, 4)$. First, we estimate $H_{1}$. By integrating by parts and Lemma \ref{f}, we have
\begin{equation*}
\begin{split}
H_{1}\leq&\alpha\||x|^{\alpha/2-1}\dot{\mathbf{u}}\|_{L^2}\||x|^{\alpha/2}h^2\|_{L^2}(1+\|n\|_{L^{\infty}})+\||x|^{\alpha/2}\dive\dot{\mathbf{u}}\|_{L^2}\||x|^{\alpha/2}h^2\|_{L^2}(1+\|n\|_{L^{\infty}})\\
\leq&\frac{\alpha^4}{4}\||x|^{\alpha/2}\nabla\dot{\mathbf{u}}\|_{L^2}^2+C\||x|^{\alpha/2}h^2\|_{L^2}^2(1+\|n\|_{L^{2}}\|\nabla^2 n\|_{L^2}).
\end{split}
\end{equation*}
A straightforward computation shows that $H_{2}$ satisfies
\begin{equation*}
H_{2}\leq\||x|^{\alpha/2}\sqrt{h}\dot{\mathbf{u}}\|_{L^2}\||x|^{\alpha/2}\nabla n\|_{L^2}\|h^{\frac{3}{2}}\|_{L^{\infty}}\leq\varepsilon\||x|^{\alpha/2}\sqrt{h}\dot{\mathbf{u}}\|_{L^2}^2+C_{\varepsilon}\||x|^{\alpha/2}\nabla n\|_{L^2}^2\|h\|_{L^{\infty}}^3.
\end{equation*}
For the nonlinear term $H_{3}$, using Lemma \ref{f}, we have
\begin{equation*}
\begin{split}
H_{3}\leq&-\frac{\mu}{2}\frac{d}{dt}\int |x|^{\alpha}|\nabla\mathbf{u}|^2\, dx+\mu\frac{\alpha^2}{2}\||x|^{\alpha/2}\nabla\dot{\mathbf{u}}\|_{L^2}\||x|^{\alpha/2}\nabla\mathbf{u}\|_{L^2}+\|\nabla\mathbf{v}\|_{L^{\infty}}\||x|^{\alpha/2}\nabla\mathbf{u}\|_{L^2}^2\\
&+\mu\frac{\alpha^2}{2}\||x|^{\alpha/2}\nabla\mathbf{v}\|_{L^2}\||x|^{\alpha/2}\nabla\mathbf{u}\|_{L^2}\|\nabla\mathbf{u}\|_{L^{\infty}}.
\end{split}
\end{equation*}
Similar as the calculation of $H_{3}$, we have
\begin{equation*}
\begin{split}
H_{4}\leq&-\frac{(\mu+\lambda)}{2}\frac{d}{dt}\int|x|^{\alpha}(\dive\mathbf{u})^2\, dx+(\mu+\lambda)\frac{\alpha^2}{2}\||x|^{\alpha/2}\nabla\dot{\mathbf{u}}\|_{L^2}\||x|^{\alpha/2}\nabla\mathbf{u}\|_{L^2}\\
&+\|\nabla\mathbf{v}\|_{L^{\infty}}\||x|^{\alpha/2}\nabla\mathbf{u}\|_{L^2}^2+(\mu+\lambda)\frac{\alpha^2}{2}\||x|^{\alpha/2}\nabla\mathbf{v}\|_{L^2}\||x|^{\alpha/2}\nabla\mathbf{u}\|_{L^2}\|\nabla\mathbf{u}\|_{L^{\infty}}.
\end{split}
\end{equation*}
In order to further estimate $H_{3}$ and $H_{4}$, we should first estimate $\|\nabla\mathbf{u}\|_{L^{\infty}}^2$. Through the Gagliardo-Nirenberg's inequality, we have
\begin{equation}\label{195}
\|\nabla\mathbf{u}\|_{L^{\infty}}^2\leq C\|\mathbf{u}\|_{L^q}^{1-\frac{2}{q}}\|\nabla^2\mathbf{u}\|_{L^q}^{1+\frac{2}{q}},
\end{equation}
where $q=4/\alpha$. Using the $L^p$ elliptic estimate to equation $(\ref{4})_{1}$ to get
\begin{equation*}
\begin{split}
\|\nabla^2\mathbf{u}\|_{L^q}^2\leq&C(\|h\dot{\mathbf{u}}\|_{L^q}^2+\|h^2\nabla n\|_{L^q}^2+\|h\nabla h\|_{L^q}^2+\|n h\nabla h\|_{L^q}^2)\\
\leq&C(\|h\|_{L^{\infty}}^2\|\dot{\mathbf{u}}\|_{L^q}^2+\|h\|_{L^{\infty}}^4\|\nabla n\|_{L^q}^2+\|h\|_{L^{\infty}}^2\|\nabla h\|_{L^q}^2+\|n h\|_{L^{\infty}}^2\|\nabla h\|_{L^q}^2),
\end{split}
\end{equation*}
which together with Lemma \ref{f} and (\ref{195}) to get
\begin{equation}\label{196}
\begin{split}
\|\nabla\mathbf{u}\|_{L^{\infty}}^2\leq&\||x|^{\alpha/2}\nabla\mathbf{u}\|_{L^2}^{1-\frac{2}{q}}\|h\|_{L^{\infty}}^{1+\frac{2}{q}}\||x|^{\alpha/2}\nabla\dot{\mathbf{u}}\|_{L^2}^{1+\frac{2}{q}}+\||x|^{\alpha/2}\nabla\mathbf{u}\|_{L^2}^{1-\frac{2}{q}}\|h\|_{L^{\infty}}^{2+\frac{4}{q}}\|\nabla n\|_{L^q}^{1+\frac{2}{q}}\\
&+\||x|^{\alpha/2}\nabla\mathbf{u}\|_{L^2}^{1-\frac{2}{q}}\|h\|_{L^{\infty}}^{1+\frac{2}{q}}\|\nabla h\|_{L^q}^{1+\frac{2}{q}}+\||x|^{\alpha/2}\nabla\mathbf{u}\|_{L^2}^{1-\frac{2}{q}}\|n h\|_{L^{\infty}}^{1+\frac{2}{q}}\|\nabla h\|_{L^q}^{1+\frac{2}{q}}\\
\leq&\varepsilon\||x|^{\alpha/2}\nabla\dot{\mathbf{u}}\|_{L^2}^2+C_{\varepsilon}\||x|^{\alpha/2}\nabla\mathbf{u}\|_{L^2}^2(1+\|h\|_{L^{\infty}}^{\frac{4(q+2)}{(q-2)}})\\
&+\|\nabla n\|_{L^q}^2+\|h\|_{L^{\infty}}^2\|\nabla h\|_{L^q}^2+\|n h\|_{L^{\infty}}^2\|\nabla h\|_{L^q}^2.
\end{split}
\end{equation}
Putting the estimate (\ref{196}) and all the estimates of $H_{i}\ (i=1,\cdots,4)$ into (\ref{B212}), we have
\begin{equation}\label{197}
\begin{split}
&\frac{\mu}{2}\frac{d}{dt}\||x|^{\alpha/2}\nabla\mathbf{u}\|_{L^2}^2+\frac{(\mu+\lambda)}{2}\frac{d}{dt}\||x|^{\alpha/2}\dive\mathbf{u}\|_{L^2}^2+\frac{1}{2}\||x|^{\alpha/2}\sqrt{h}\dot{\mathbf{u}}\|_{L^2}^2\\
\leq&2(2\mu+\lambda)\alpha^4\||x|^{\alpha/2}\nabla\dot{\mathbf{u}}\|_{L^2}^2+C\||x|^{\alpha/2}h^2\|_{L^2}^2(1+\|n\|_{L^2}\|\nabla^2 n\|_{L^2})\\
&+C(1+\|\nabla\mathbf{v}\|_{L^{\infty}}+\||x|^{\alpha/2}\nabla\mathbf{v}\|_{L^2}^2+\|h\|_{L^{\infty}}^{\frac{4(q+2)}{q-2}})\||x|^{\alpha/2}\nabla\mathbf{u}\|_{L^2}^2\\
&+C(\|\nabla n\|_{L^q}^2+\|h\|_{L^{\infty}}^2\|\nabla h\|_{L^q}^2+\|nh\|_{L^{\infty}}^2\|\nabla h\|_{L^q}^2+\|h\|_{L^{\infty}}^3\||x|^{\alpha/2}\nabla n\|_{L^2}^2).
\end{split}
\end{equation}
Multiplying the equation (\ref{B211}) by $|x|^{\alpha}\dot{\mathbf{u}}$, integrating over $\Omega$ to get
\begin{equation*}
\begin{split}
&\frac{1}{2}\frac{d}{dt}\int |x|^{\alpha}h\dot{\mathbf{u}}^2\, dx+\mu\int |x|^{\alpha}|\nabla\dot{\mathbf{u}}|^2\, dx+(\mu+\lambda)\int |x|^{\alpha}(\dive\dot{\mathbf{u}})^2\, dx\\
=&-\int |x|^{\alpha}h_{t}^2\dot{\mathbf{u}}\cdot\nabla n\, dx-\int |x|^{\alpha}\dot{\mathbf{u}}\cdot\nabla n_{t}h^2\, dx-\frac{1}{2}\int |x|^{\alpha}\dot{\mathbf{u}}\cdot\nabla h^2 n_{t}\, dx-\frac{1}{2}\int |x|^{\alpha}\dot{\mathbf{u}}\cdot\nabla h_{t}^2(1+n)\, dx\\
&-\int |x|^{\alpha}\dot{\mathbf{u}}\dive(h^2\nabla n\otimes\mathbf{v})\, dx-\frac{1}{2}\int |x|^{\alpha}\dot{\mathbf{u}}\dive\big((1+n)\nabla h^2\otimes\mathbf{v}\big)\, dx-\mu\int |x|^{\alpha}\dot{\mathbf{u}}\Delta(\mathbf{v}\cdot\nabla\mathbf{u})\, dx\\
&-(\mu+\lambda)\int |x|^{\alpha}\dot{\mathbf{u}}\nabla(\dive(\mathbf{v}\cdot\nabla\mathbf{u}))\, dx+\mu\int |x|^{\alpha}\dot{\mathbf{u}}\dive(\Delta\mathbf{u}\otimes\mathbf{v})\, dx\\
&+(\mu+\lambda)\int |x|^{\alpha}\dot{\mathbf{u}}\dive\big(\nabla(\dive\mathbf{u})\otimes\mathbf{v}\big)\, dx-\alpha\mu\int |x|^{\alpha-2}x_{i}\cdot\dot{\mathbf{u}}\nabla\dot{\mathbf{u}}\, dx\\
&-\alpha(\mu+\lambda)\int |x|^{\alpha-2}x_{i}\cdot\dot{\mathbf{u}}\dive\dot{\mathbf{u}}\, dx-\alpha^2\int|x|^{\alpha-2}x_{i}\cdot\mathbf{v}h|\dot{\mathbf{u}}|^2\, dx\\
:=&\sum_{i=1}^{13}I_{i}.
\end{split}
\end{equation*}
In the following, we establish the estimates of $I_{i}\ (i=1,\cdots,13)$. First we consider $I_{1}+I_{2}+I_{5}$. By integrating by parts, equation (\ref{5}), H\"older's inequality, Gagliardo-Nirenberg's inequality and Young's inequality, we have
\begin{equation*}
\begin{split}
&I_{1}+I_{2}+I_{5}\\
\leq&\||x|^{\alpha/2}\sqrt{h}\dot{\mathbf{u}}\|_{L^2}\||x|^{\alpha/2}\nabla n\|_{L^2}\|\dive\mathbf{v}\|_{L^{\infty}}\|h\|_{L^{\infty}}^{\frac{3}{2}}+\||x|^{\alpha/2}\nabla\dot{n}\|_{L^2}\||x|^{\alpha/2}\sqrt{h}\dot{\mathbf{u}}\|_{L^2}\|h\|_{L^{\infty}}^{\frac{3}{2}}\\
\leq&C\||x|^{\alpha/2}\sqrt{h}\dot{\mathbf{u}}\|_{L^2}^2(1+\|\dive\mathbf{v}\|_{L^{\infty}})+C(\|\dive\mathbf{v}\|_{L^{\infty}}\||x|^{\alpha/2}\nabla n\|_{L^2}^2\|h\|_{L^{\infty}}^3+\||x|^{\alpha/2}\nabla\dot{n}\|_{L^2}^2\|h\|_{L^{\infty}}^3).
\end{split}
\end{equation*}
Then we consider $I_{3}+I_{4}+I_{6}$. Through Lemma \ref{f}, equation (\ref{5}), integration by parts, H\"older's inequality and Young's inequality, we have
\begin{equation*}
\begin{split}
&2(I_{3}+I_{4}+I_{6})\\
=&\int |x|^{\alpha}\dot{\mathbf{u}}\cdot\nabla n_{t} h^2\, dx+\int n_{t}h^2\dive(|x|^{\alpha}\dot{\mathbf{u}})\, dx+\int|x|^{\alpha}\dot{\mathbf{u}}\cdot\nabla n h_{t}^2\, dx\\
&+\int(1+n)h_{t}^2\dive(|x|^{\alpha}\dot{\mathbf{u}})\, dx+\int(1+n)\mathbf{v}\cdot\nabla h^2\dive(|x|^{\alpha}\dot{\mathbf{u}})\, dx\\
=&\int|x|^{\alpha}\dot{\mathbf{u}}\cdot\nabla\dot{n}h^2\, dx+\int |x|^{\alpha}h^2\dot{n}\dive\dot{\mathbf{u}}\, dx+\alpha\int |x|^{\alpha-2}x_{i}\cdot\dot{\mathbf{u}}h^2\dot{n}\, dx-2\int |x|^{\alpha}h^2\dot{\mathbf{u}}\cdot\nabla n\dive\mathbf{v}\, dx\\
&-2\int |x|^{\alpha}(1+n)h^2\dive\mathbf{v}\dive\dot{\mathbf{u}}\, dx-2\alpha\int |x|^{\alpha-2}x_{i}\cdot\dot{\mathbf{u}}(1+n)h^2\dive\mathbf{v}\, dx\\
\leq&\||x|^{\alpha/2}\sqrt{h}\dot{\mathbf{u}}\|_{L^2}\||x|^{\alpha/2}\nabla\dot{n}\|_{L^2}\|h\|_{L^{\infty}}^{\frac{3}{2}}+\||x|^{\alpha/2}\nabla\dot{\mathbf{u}}\|_{L^2}\||x|^{\alpha/2}\dot{n}\|_{L^{2}}\|h\|_{L^{\infty}}^2\\
&+\||x|^{\alpha/2}\sqrt{h}\dot{\mathbf{u}}\|_{L^2}\||x|^{\alpha/2-1}\dot{n}\|_{L^2}\|h\|_{L^{\infty}}^{\frac{3}{2}}+\||x|^{\alpha/2}\sqrt{h}\dot{\mathbf{u}}\|_{L^2}\||x|^{\alpha/2}\nabla\mathbf{v}\|_{L^{2}}\|\nabla n\|_{L^{\infty}}\|h\|_{L^{\infty}}^{\frac{3}{2}}\\
&+\||x|^{\alpha/2}\nabla\dot{\mathbf{u}}\|_{L^2}\||x|^{\alpha/2}\dive\mathbf{v}\|_{L^2}\|h^2(1+n)\|_{L^{\infty}}+\||x|^{\alpha/2}\dive\mathbf{v}\|_{L^2}\||x|^{\alpha/2-1}\dot{\mathbf{u}}\|_{L^2}\|h^2(1+n)\|_{L^{\infty}}\\
\leq&\varepsilon\||x|^{\alpha/2}\nabla\dot{\mathbf{u}}\|_{L^2}^2+C_{\varepsilon}(\||x|^{\alpha/2}\dot{n}\|_{L^2}^2\|h\|_{L^{\infty}}^4+\||x|^{\alpha/2}\dive\mathbf{v}\|_{L^2}^2\|h^2(1+n)\|_{L^{\infty}}^2)\\
&+C\big(\||x|^{\alpha/2}\nabla\dot{n}\|_{L^2}^2\|h\|_{L^{\infty}}^3+\|\nabla n\|_{L^{\infty}}^2\|h\|_{L^{\infty}}^3+\||x|^{\alpha/2}\sqrt{h}\dot{\mathbf{u}}\|_{L^2}^2(1+\||x|^{\alpha/2}\nabla\mathbf{v}\|_{L^2}^2)\big).
\end{split}
\end{equation*}
By integrating by parts, it is easy to get
\begin{equation*}
I_{7}+I_{9}\leq\||x|^{\alpha/2}\nabla\dot{\mathbf{u}}\|_{L^2}\||x|^{\alpha/2}\nabla\mathbf{u}\|_{L^{2}}\|\nabla\mathbf{v}\|_{L^{\infty}}\leq\varepsilon\||x|^{\alpha/2}\nabla\dot{\mathbf{u}}\|_{L^2}^2+C_{\varepsilon}\||x|^{\alpha/2}\nabla\mathbf{u}\|_{L^2}^2\|\nabla\mathbf{v}\|_{L^2}^2
\end{equation*}
and
\begin{equation*}
I_{8}+I_{10}\leq\varepsilon\||x|^{\alpha/2}\dive\dot{\mathbf{u}}\|_{L^2}^2+C_{\varepsilon}\||x|^{\alpha/2}\nabla\mathbf{u}\|_{L^2}^2\|\nabla\mathbf{v}\|_{L^2}^2.
\end{equation*}
For the nonlinear terms $I_{11}$ to $I_{13}$, using Lemma \ref{f} to get
\begin{equation*}
\begin{split}
I_{11}\leq&\alpha\mu\||x|^{\alpha/2-1}\dot{\mathbf{u}}\|_{L^2}\||x|^{\alpha/2}\nabla\dot{\mathbf{u}}\|_{L^2}\leq\mu\frac{\alpha^2}{2}\||x|^{\alpha/2}\nabla\dot{\mathbf{u}}\|_{L^2}^2,\\
I_{12}\leq&\alpha(\mu+\lambda)\||x|^{\alpha/2-1}\dot{\mathbf{u}}\|_{L^2}\||x|^{\alpha/2}\dive\dot{\mathbf{u}}\|_{L^2}\leq (\mu+\lambda)\frac{\alpha^2}{2}\||x|^{\alpha/2}\nabla\dot{\mathbf{u}}\|_{L^2}^2,\\
I_{13}\leq&\alpha\||x|^{\alpha/2}\sqrt{h}\dot{\mathbf{u}}\|_{L^2}\||x|^{\alpha/2-1}\mathbf{v}\|_{L^2}\|h\|_{L^{\infty}}^{\frac{3}{2}}\leq\||x|^{\alpha/2}\sqrt{h}\dot{\mathbf{u}}\|_{L^2}^2+\frac{\alpha^3}{4}\||x|^{\alpha/2}\nabla\mathbf{v}\|_{L^2}^2\|h\|_{L^{\infty}}^3.
\end{split}
\end{equation*}
Combining all these estimates of $I_{i}\, (i=1\cdots, 13)$, we have
\begin{equation}\label{198}
\begin{split}
&\frac{d}{dt}\||x|^{\alpha/2}\sqrt{h}\dot{\mathbf{u}}\|_{L^2}^2+\mu\||x|^{\alpha/2}\nabla\dot{\mathbf{u}}\|_{L^2}^2+(\mu+\lambda)\||x|^{\alpha/2}\dive\dot{\mathbf{u}}\|_{L^2}^2\\
\leq&C(1+\|\dive\mathbf{v}\|_{L^{\infty}}+\||x|^{\alpha/2}\nabla\mathbf{v}\|_{L^2}^2)\||x|^{\alpha/2}\sqrt{h}\dot{\mathbf{u}}\|_{L^2}^2+\big(3\varepsilon+(2\mu+\lambda)\frac{\alpha^2}{2}\big)\||x|^{\alpha/2}\nabla\dot{\mathbf{u}}\|_{L^2}^2\\
&+C_{\varepsilon}(\||x|^{\alpha/2}\dot{n}\|_{L^2}^2\|h\|_{L^{\infty}}^4+\||x|^{\alpha/2}\dive\mathbf{v}\|_{L^2}^2\|h^2(1+n)\|_{L^{\infty}}^2+\||x|^{\alpha/2}\nabla\mathbf{u}\|_{L^2}^2\|\nabla\mathbf{v}\|_{L^2}^2)\\
&+C(\|\dive\mathbf{v}\|_{L^{\infty}}\||x|^{\alpha/2}\nabla n\|_{L^2}^2+\||x|^{\alpha/2}\nabla\dot{n}\|_{L^2}^2+\|\nabla n\|_{L^{\infty}}^2+\||x|^{\alpha/2}\dive\mathbf{v}\|_{L^2}^2)\|h\|_{L^{\infty}}^3.
\end{split}
\end{equation}
Adding (\ref{197}) and (\ref{198}), we obtain
\begin{equation}\label{199}
\begin{split}
&\frac{d}{dt}\big(\||x|^{\alpha/2}\nabla\mathbf{u}\|_{L^2}^2+\||x|^{\alpha/2}\dive\mathbf{u}\|_{L^2}^2+\||x|^{\alpha/2}\sqrt{h}\dot{\mathbf{u}}\|_{L^2}^2\big)\\
&+\big(\||x|^{\alpha/2}\sqrt{h}\dot{\mathbf{u}}\|_{L^2}^2+\||x|^{\alpha/2}\nabla\dot{\mathbf{u}}\|_{L^2}^2+\||x|^{\alpha/2}\dive\dot{\mathbf{u}}\|_{L^2}^2\big)\\
\leq&C(1+\|\nabla\mathbf{v}\|_{L^2\cap L^{\infty}}^2+\|h\|_{L^{\infty}}^{4(q+2)/(q-2)}+\||x|^{\alpha/2}\nabla\mathbf{v}\|_{L^2}^2)(\||x|^{\alpha/2}\nabla\mathbf{u}\|_{L^2}^2+\||x|^{\alpha/2}\sqrt{h}\dot{\mathbf{u}}\|_{L^2}^2)\\
&+C(\|\dive\mathbf{v}\|_{L^{\infty}}\||x|^{\alpha/2}\nabla n\|_{L^2}^2+\|\nabla n\|_{L^{\infty}}^2+\||x|^{\alpha/2}\nabla\dot{n}\|_{L^2}^2+\||x|^{\alpha/2}\nabla n\|_{L^{2}}^2)\|h\|_{L^{\infty}}^3\\
&+C\big(\||x|^{\alpha/2}\dot{n}\|_{L^2}^2\|h\|_{L^{\infty}}^4+\||x|^{\alpha/2}\dive\mathbf{v}\|_{L^2}^2(\|h^2(1+n)\|_{L^{\infty}}^2+\|h\|_{L^{\infty}}^3)+\|\nabla n\|_{L^q}^2\\
&+\||x|^{\alpha/2}h^2\|_{L^2}^2(1+\|n\|_{L^2}\|\nabla^2 n\|_{L^2})+\|h\|_{L^{\infty}}^2\|\nabla h\|_{L^q}^2(1+\|n\|_{L^{\infty}}^2)\big).
\end{split}
\end{equation}
Using the Gronwall's inequality to (\ref{199}), together with Lemma \ref{c}, \ref{d}, \ref{i}, and \ref{H} gives (\ref{B304}). According to the definition of the material derivative and elliptic estimate, we have
\begin{equation*}
\int_{0}^{t}\|\mathbf{u}_{t}\|_{L^q}^2\, ds\leq \int_{0}^{t}\|\dot{\mathbf{u}}\|_{L^q}^2\, ds+\int_{0}^{t}\|\mathbf{v}\cdot\nabla\mathbf{u}\|_{L^q}^2\, ds
\end{equation*}
and
\begin{equation*}
\begin{split}
\int_{0}^{t} \|\nabla^2\mathbf{u}\|_{L^q}^2 \, ds\leq &C\int_{0}^{t}(\|h\dot{\mathbf{u}}\|_{L^q}^2+\|h^2\nabla n\|_{L^q}^2+\|(1+n)h\nabla h\|_{L^q}^2)\, ds\\
\leq&C\int_{0}^{t} (\|h\|_{L^{\infty}}^2\|\dot{\mathbf{u}}\|_{L^q}^2+\|h\|_{L^{\infty}}^4\|\nabla n\|_{L^q}^2+\|(1+n)h\|_{L^{\infty}}^2\|\nabla h\|_{L^q}^2)\, ds.
\end{split}
\end{equation*}
The estimate (\ref{B304}), Lemma \ref{f}, \ref{c}, \ref{d}-\ref{h} yield (\ref{B305}).
\end{proof}

\begin{lemma}\label{k}
Assume that $(h, c, n, \mathbf{u})$ is the solution of equation (\ref{1})-(\ref{4}). Then there exists a time $T^*\in (0, 1]$ depending only on $c_{0}, \mu, \lambda, \alpha$ such that
\begin{equation*}
\Phi(\mathbf{u}, T^*)\leq M,
\end{equation*}
provided $\Phi(\mathbf{v}, T^*)\leq M$ with some given $M=M(\mu, \lambda, c_{0})>1$.
\end{lemma}

\begin{proof}
Through Lemmas \ref{d}-\ref{h} and Lemma \ref{j}, we have
\begin{equation*}
\Phi(\mathbf{u}, n, t)\leq Cc_{0}^2\exp\{Cc_{0}^2\exp\{C\Phi^3(\mathbf{v}, m, t)t^{1/5}\}\},
\end{equation*}
which yields that
\begin{equation*}
\Phi(\mathbf{u}, n, T^*)\leq M
\end{equation*}
by choosing $M=Cc_{0}^2\exp\{Cc_{0}^2\exp\{C\}\}$ and
\begin{equation*}
T^*=\min\{M^{-15}, 1\}.
\end{equation*}
\end{proof}

\subsection{The local existence of strong solution}

In the following, based on the uniform estimates obtained in the previous subsection of solutions to the linearized problem (\ref{1})-(\ref{4}), we use the Schauder fixed point theory to establish the unique local strong solution of the problem (\ref{A})-(\ref{A2}).

\begin{lemma}\label{x}
Let $\Omega\subset\mathbb{R}^2$ be a bounded domain and the far field state $\widetilde{h}=0$. Suppose that the initial data $(n_{0}\geq 0, c_{0}\geq 0, h_{0}\geq\delta, \mathbf{u}_{0})$ satisfy (\ref{A3})-(\ref{304}), where $\delta>0$. Then for $T^*$ as in Lemma \ref{k}, there exists a unique strong solution $(n, c, h, \mathbf{u})$ to the problem (\ref{A}) on $\Omega\times[0, T^*]$ with 
\begin{equation}\label{z600}
(n, c, h, \mathbf{u})|_{t=0}=(n_{0}, c_{0}, h_{0}, \mathbf{u}_{0}),\ (n, c, \mathbf{u})|_{\partial\Omega}=(0,0, \mathbf{0}),
\end{equation}
satisfying
\begin{equation}\label{201}
\begin{cases}
n\in C(0, T^*; H_{0}^2),\ n_{t}\in L^{\infty}(0, T^*; L^2)\cap L^{2}(0, T^*; H^1),\\
c\in C(0, T^*; W^{2, 4}),\ c_{t}\in L^{\infty}(0, T^*; H^1)\cap L^2(0, T^*; H^2),\\
(h, h^2)\in C(0, T^*; W^{1, p}),\ (h, h^2)_{t}\in L^{\infty}(0, T^*; L^p),\\
\mathbf{u}\in C(0, T^*; L^q\cap D^1\cap D^2)\cap L^2(0, T^*; D^{2, q}),\\
\mathbf{u}_{t}\in L^2(0, T^*; L^q\cap D^1),\ \sqrt{h}\mathbf{u}_{t}\in L^{\infty}(0, T^*; L^2),
\end{cases}
\end{equation}
for $q=4/\alpha$ and any $p\in[2, q]$. Moreover, the following estimate holds:
\begin{equation}\label{202}
\begin{split}
&\sup_{0\leq t\leq T^*}(\|n\|_{H^2}^2+\|n_{t}\|_{L^2}^2+\|c\|_{H^2}^2+\|\nabla^2 c\|_{L^4}^2+\|c_{t}\|_{H^1}^2+\|\sqrt{h}\mathbf{u}\|_{L^2}^2+\|\nabla^2\mathbf{u}\|_{L^2}^2)\\
&+\sup_{0\leq t\leq T^*}(\|(h, h^2)\|_{W^{1, p}}^2+\|(h, h^2)_{t}\|_{L^p}^2)+\int_{0}^{T^*}(\|\nabla^2 n\|_{L^4}^2+\|\nabla n_{t}\|_{L^2}^2+\|\nabla^2 c_{t}\|_{L^2}^2)\, dt\\
&+\sup_{0\leq t\leq T^*}\int_{\Omega}|x|^{\alpha}(n^2+c^2+|\nabla n|^2+|\nabla c|^2+|\dot{n}|^2+|\dot{c}|^2)\, dx\\
&+\sup_{0\leq t\leq T^*}\int_{\Omega}(1+|x|^{\alpha})(|\nabla\mathbf{u}|^2+h^4+h|\dot{\mathbf{u}}|^2+h|\mathbf{u}|^2)\, dx+\int_{0}^{T^*}\int_{\Omega}(1+|x|^{\alpha})|\nabla\dot{\mathbf{u}}|^2\, dx\, dt\\
&+\int_{0}^{T^*}\|\nabla^2\mathbf{u}\|_{L^q}^2\, dt+\int_{0}^{T^*}\int_{\Omega}|x|^{\alpha}(|\nabla\dot{n}|^2+|\nabla\dot{c}|^2)\, dxdt\\
\leq&C(\mu, \lambda, c_{0}).
\end{split}
\end{equation}
\end{lemma}

\begin{proof}
Setting
\begin{equation*}
\mathcal{B}:=L^2(0, T^*; H_{0}^1)\times L^2(0, T^*; H_{0}^1)
\end{equation*}
and
\begin{equation*}
\begin{split}
\mathfrak{R}:=\big\{&(m, \mathbf{v})| (m, \mathbf{v})\in L^{\infty}(0, T^*; H_{0}^1)\cap L^2(0, T^*; D^2)\times L^{\infty}(0, T^*; H_{0}^1\cap D^2)\cap L^2(0, T^*; D^{2, q}),\\
&(m_{t}, \mathbf{v}_{t})\in L^{\infty}(0, T^*; L^2)\cap L^2(0, T^*; H_{0}^1)\times L^2(0, T^*; H_{0}^1), \Phi(m, \mathbf{v}, T^*)< M\big\}.
\end{split}
\end{equation*}
By Lemma \ref{G1}, it is easy to get that $\mathfrak{R}$ is a convex and compact subset of the Banach space $\mathcal{B}$. For any $(m, \mathbf{v})\in\mathfrak{R}$, through Lemma \ref{G2}, there exists a unique solution $h=h(\mathbf{v})$ of the linearized equation (\ref{1}) on $\Omega\times[0, T^*]$. Besides, the linearized equation (\ref{2}) admits a unique solution $c=c(m, \mathbf{v})$ on $\Omega\times [0, T^*]$ as well. Moreover, a unique $n=\mathcal{T}_{1}(\mathbf{v}, c(m, \mathbf{v}))$ solves the linearized equation (\ref{3}) on $\Omega\times [0, T^*]$ and thus the linearized equation (\ref{4}) has a unique solution $\mathbf{u}=\mathcal{T}_{2}\big(\mathbf{v}, h(\mathbf{v}), \mathcal{T}_{1}(\mathbf{v}, c(m, \mathbf{v}))\big)$ on $\Omega\times [0, T^*]$. Therefore, we write $(n, \mathbf{u}):=\mathcal{T}(m, \mathbf{v})=(\mathcal{T}_{1}, \mathcal{T}_{2})$ with $\mathcal{T}$ mapping from $\mathfrak{R}$ to $\mathfrak{R}$. Next, we will show $\mathcal{T}$ is a continuous operator in $\mathcal{B}$. First, using the Lemma \ref{c}, we have
\begin{equation}\label{C5}
\sup_{0\leq t\leq T^*}(\|h(\mathbf{v})\|_{W^{1, q}}+\|h(\mathbf{v})_{t}\|_{L^q})\leq C.
\end{equation}
Let $\{(m_{k}, \mathbf{v}_{k})\}_{k=1}^{\infty}\in\mathfrak{R}$ and $(m_{k}, \mathbf{v}_{k})\to (m, \mathbf{v})$ in $\mathcal{B}$ as $k\to\infty$. Thus
\begin{equation*}
m_{k}\to m\ in\ L^2(0, T^*; H_{0}^1)\ as\ k\to+\infty
\end{equation*}
and
\begin{equation*}
\mathbf{v}_{k}\to\mathbf{v}\ in\ L^2(0, T^*; H_{0}^1)\ as\ k\to+\infty,
\end{equation*}
which together with $m_{k},\ \mathbf{v}_{k}\in\mathfrak{R}$ implies
\begin{equation}\label{z500}
m_{k}\rightharpoonup m\quad\omega^*\ in\ L^{\infty}(0, T^*; H_{0}^1)\cap L^2(0, T^*; D^{2})\ as\ k\to+\infty
\end{equation}
and
\begin{equation}\label{6}
\mathbf{v}_{k}\rightharpoonup\mathbf{v}\quad\omega^*\ in\ L^{\infty}(0, T^*; H_{0}^1\cap D^2)\cap L^2(0, T^*; D^{2, q})\ as\ k\to+\infty.
\end{equation}
Thus, it follows from Lemma \ref{G1} and (\ref{C5}) that up to a subsequence,
\begin{equation}\label{z400}
h(\mathbf{v}_{k_{j}})\to h\ in\ C(\overline{\Omega}\times [0, T^*])\ as\ k_{j}\to+\infty.
\end{equation}
Taking the limits in (\ref{1}) where $h, \mathbf{v}$ are replaced by $h(\mathbf{v}_{k_{j}})$ and $\mathbf{v}_{k_{j}}$ respectively, we obtain that $h$ is a weak solution to (\ref{1}). Then by the uniqueness of the weak solutions due to (\ref{6}) and (\ref{z400}), we have $h=h(\mathbf{v})$, which again together with (\ref{C5}) and (\ref{z400}) implies
\begin{equation*}
h(\mathbf{v}_{k})\to h\ in\ C(\overline{\Omega}\times [0, T^*])\ as\ k\to+\infty
\end{equation*}
and
\begin{equation*}
h(\mathbf{v}_{k})\rightharpoonup h\quad\omega^*\ in\ L^{\infty}(0, T^*; W^{1, q})\ as\ k\to+\infty.
\end{equation*}
By the virtue of Lemmas \ref{G1} and \ref{b}, we get that up to a subsequence, 
\begin{equation}\label{9}
c(m_{k_{j}}, \mathbf{v}_{k_{j}})\to c\ in\ C(0, T^*; H_{0}^1)\cap L^2(0, T^*; H_{0}^1)\ as\ k_{j}\to+\infty.
\end{equation}
Taking the limits in equation (\ref{2}), where $m, c, \mathbf{v}$ are replaced by $m_{k_{j}}, c(m_{k_{j}}, \mathbf{v}_{k_{j}}), \mathbf{v}_{k_{j}}$, we obtain that $c$ is a weak solution of (\ref{2}). By the uniqueness of the weak solution, due to (\ref{z500}), (\ref{6}), and (\ref{9}), we have $c=c(m, \mathbf{v})$, which again together with (\ref{9}) and Lemma \ref{b} obtains
\begin{equation}\label{11}
c(m_{k}, \mathbf{v}_{k})\to c\ in\ C(0, T^*; H_{0}^1)\cap L^2(0, T^*; H_{0}^1)\ as\ k\to+\infty
\end{equation}
and
\begin{equation}\label{12}
c(m_{k}, \mathbf{v}_{k})\rightharpoonup c\quad\omega^*\ in\ L^{\infty}(0, T^*; H_{0}^2)\cap L^2(0, T^*; W^{2, 4})\ as\ k\to+\infty.
\end{equation}
After that, let $n_{k}=\mathcal{T}_{1}(\mathbf{v}_{k}, c(m_{k}, \mathbf{v}_{k}))$. By Lemmas \ref{G1} and \ref{d}, we have that up to a subsequence,
\begin{equation}\label{14}
n_{k_{j}}\to n\ in\ C(0, T^*; H_{0}^1)\cap L^2(0, T^*; H_{0}^1)\ as\ k_{j}\to+\infty.
\end{equation}
Taking the limits in equation (\ref{3}), where $n, c, \mathbf{v}$ are replaced by $n_{k_{j}}, c_{k_{j}}, \mathbf{v}_{k_{j}}$, we have $n$ is a weak solution to (\ref{3}). By the uniqueness of the weak solution, due to (\ref{6}) and (\ref{12}), we have $n=\mathcal{T}_{1}(\mathbf{v}, c)$, which together with (\ref{14}) and Lemma \ref{d} again to get
\begin{equation*}
\mathcal{T}_{1}(\mathbf{v}_{k}, c(m_{k}, \mathbf{v}_{k}))\to n\ in\ C(0, T^*; H_{0}^1)\cap L^2(0, T^*; H_{0}^1)\ as\ k\to+\infty
\end{equation*}
and
\begin{equation*}
\mathcal{T}_{1}(\mathbf{v}_{k}, c(m_{k}, \mathbf{v}_{k}))\rightharpoonup n\ in\ L^{\infty}(0, T^*; H_{0}^2)\cap L^2(0, T^*; W^{2, 4})\ as\ k\to+\infty.
\end{equation*}
At last, let $\mathbf{u}_{k}=\mathcal{T}_{2}\big(\mathbf{v}_{k}, h(\mathbf{v}_{k}), \mathcal{T}_{1}(\mathbf{v}_{k}, c(m_{k}, \mathbf{v}_{k}))\big)$. By Lemma \ref{G1}, \ref{g}, \ref{h}, and \ref{j}, up to a subsequence, we have
\begin{equation*}
\mathbf{u}_{k_{j}}\to\mathbf{u}\ in\ L^{2}(0, T^*; H_{0}^1)\ as\ k_{j}\to+\infty
\end{equation*}
and
\begin{equation*}
\mathbf{u}_{k_{j}}\rightharpoonup\mathbf{u}\quad\omega^*\ in\ L^{\infty}(0, T^*; H_{0}^1\cap D^2)\cap L^2(0, T^*; D^{2, q})\ as\ k\to+\infty.
\end{equation*}
Taking the limits in equation (\ref{4}), where $h,\mathbf{u}, n, \mathbf{v}$ are replaced by $h(\mathbf{v}_{k_{j}}), \mathbf{u}_{k_{j}}, n_{k_{j}}$, and $\mathbf{v}_{k_{j}}$, we obtain that $\mathbf{u}=\mathcal{T}_{2}(\mathbf{v}, h(\mathbf{v}), n)=\mathcal{T}_{2}\big(\mathbf{v}, h(\mathbf{v}), \mathcal{T}_{1}(\mathbf{v}, c(\mathbf{v}, m))\big)$ is a weak solution to (\ref{4}) with
\begin{equation*}
\mathbf{u}_{k}\to\mathbf{u}\ in\ L^2(0, T^*; H_{0}^1)\ as\ k\to+\infty.
\end{equation*}
Therefore $(n, \mathbf{u})=\mathcal{T}(m, \mathbf{v})$ is a continuous operator in $\mathcal{B}=\big(L^2(0, T^*; H_{0}^1)\big)^2$. By Schauder fixed point theorem, there exists $(n, \mathbf{u})\in\mathfrak{R}$ such that
\begin{equation*}
\big(\mathcal{T}_{2}(\mathbf{u}, h(\mathbf{u}), \mathcal{T}_{1}(\mathbf{u}, c(\mathbf{u}, n))), \mathcal{T}_{1}(\mathbf{u}, c(\mathbf{u}, n))\big)=(\mathbf{u}, n),
\end{equation*}
which together with Lemmas \ref{c}-\ref{j} implies (\ref{202}). Then the uniqueness of the solution $(n, \mathbf{u})$ and (\ref{201}) follows from (\ref{202}), Sobolev embedding theory and Lemma \ref{f}.
\end{proof}

Now, we can prove the local existence of unique strong solution to the Cauchy problem (\ref{A})-(\ref{A2}) on $\mathbb{R}^2\times[0, T^*]$.

\begin{proposition}\label{l}
For $\widetilde{h}=0$ and $\Omega=\mathbb{R}^2$, assume that $(n_{0}\geq 0, c_{0}\geq 0, h_{0}\geq0, \mathbf{u}_{0})$ satisfies (\ref{A3})-(\ref{304}). Then for $T^*$ as in Lemma \ref{k}, there exists a unique strong solution $(n, c, h, \mathbf{u})$ to the problem (\ref{A})-(\ref{A2}) on $\mathbb{R}^2\times[0, T^*]$ satisfying (\ref{201}) and (\ref{202}).
\end{proposition}

\begin{proof}
Define $B_{R}=\{x||x|<R\}$ and
\begin{equation*}
\psi^{R}(x)=\psi(x/R),\ g^{R}(x)=\psi^R(x)g(x),\ h_{0}^{R}=h_{0}+R^{-2},
\end{equation*}
for $(x, t)\in\mathbb{R}^2\times[0, T^*]$, where $0\leq\psi\in C_{0}^{\infty}(B_{1})$ is a smooth cut-off function such that $\psi=1$ in $B_{1/2}$. Let $\mathbf{u}_{0}^{R}\in H_{0}^1(B_{R})\cap H^3(B_{R})$ be the unique solution to the elliptic boundary value problem
\begin{equation*}
L\mathbf{u}_{0}^{R}=F_{0}^{R}\ in\ \ B_{R}\ and\ \mathbf{u}_{0}^{R}|_{\partial B_{R}}=\mathbf{0},
\end{equation*}
where
\begin{equation*}
F_{0}^{R}=-(h_{0}^{R})^2\nabla n_{0}^{R}-\frac{1}{2}(1+n_{0}^R)\nabla (h_{0}^{R})^2+h_{0}^{R}g^{R}.
\end{equation*}
By defining $\mathbf{u}_{0}^{R}$ to be zero outside $B_{R}$, we can extend $\mathbf{u}_{0}^{R}$ to $\mathbb{R}^2$, i.e.
\begin{equation*}
\mathbf{u}_{0}^{R}\to \mathbf{u}_{0}\ in\ D^{1}(B_{R})\ as\ R\to\infty.
\end{equation*}
The proof is similar as in \cite{c2006}, so we omit here. By Lemma \ref{x}, there exists some $0<T^*\leq 1$ such that for each $R>1$, the initial-boundary-value problem (\ref{A}) and (\ref{z600}) with $\Omega=B_{R}$ and $(n_{0}, c_{0}, h_{0}, \mathbf{u}_{0})=(n_{0}^{R}, c_{0}^{R}, h_{0}^{R}, \mathbf{u}_{0}^{R})$ for $n_{0}^{R}, c_{0}^{R}, h_{0}^{R}, \mathbf{u}_{0}^{R}$ defined as above, has a unique strong solution $(n, c, h, \mathbf{u})$ satisfying (\ref{201}) and (\ref{202}). We denote $(n, c, h, \mathbf{u})=(n^R, c^R, h^R, \mathbf{u}^R)$. By defining zero outside $B_{R}$, we extend $(n^R, c^R, h^R, \mathbf{u}^R)$ to $\mathbb{R}^2$. Then let $R\to\infty$, a direct computation shows that there exists a unique strong solution $(n, c, h, \mathbf{u})$ to problem (\ref{A})-(\ref{A2}) in $\mathbb{R}^2\times [0, T^*]$ which satisfies (\ref{201})-(\ref{202}). Please refer to \cite{D2012} and \cite{L2012} for details.
\end{proof}

\section{Regularity analysis}

In this section, we establish the higher order estimates of the unique strong solution $(n, c, h, \mathbf{u})$ obtained in Proposition \ref{l} for $\widetilde{h}=0$ and finish the proof of Theorem \ref{B}.

\begin{lemma}\label{n}
The following estimate holds
\begin{equation}\label{z107}
\sup_{0\leq t\leq T^*}(\|\nabla^2 h\|_{L^2}^2+\|\nabla^2 h^2\|_{L^2}^2)+\int_{0}^{T^*}\|\nabla^3\mathbf{u}\|_{L^2}^2\, dt \leq C.
\end{equation}
\end{lemma}

\begin{proof}
Adding the equations $(\ref{A})_{3}$ and (\ref{5}), using the operator $\nabla^2$ on the resulting equation shows that
\begin{equation}\label{148}
\begin{split}
&\frac{d}{dt}\big(\|\nabla^2 h\|_{L^2}^2+\|\nabla^2 h^2\|_{L^2}^2\big)\\
\leq&C(1+\|\dive\mathbf{u}\|_{L^{\infty}})(\|\nabla^2 h\|_{L^2}^2+\|\nabla^2 h^2\|_{L^2}^2)+C\|\nabla^3\mathbf{u}\|_{L^2}^2\\
\leq&C(1+\|\dive\mathbf{u}\|_{L^{\infty}})(\|\nabla^2 h\|_{L^2}^2+\|\nabla^2 h^2\|_{L^2}^2)+C(1+\||x|^{\alpha/2}\nabla\dot{\mathbf{u}}\|_{L^2}^2+\|\nabla\dot{\mathbf{u}}\|_{L^2}^2),
\end{split}
\end{equation}
where we have used the following elliptic estimate
\begin{equation}\label{147}
\begin{split}
\|\nabla^2\mathbf{u}\|_{H^1}^2\leq&C(\|h\dot{\mathbf{u}}\|_{H^1}^2+\|h^2\nabla n\|_{H^1}^2+\|(1+n)\nabla h^2\|_{H^1}^2)\\
\leq&C(1+\|\dot{\mathbf{u}}\cdot\nabla h\|_{L^2}^2+\|h\nabla\dot{\mathbf{u}}\|_{L^2}^2+\|\nabla h^2\nabla n\|_{L^2}^2+\|h^2\nabla^2 n\|_{L^2}^2)\\
&+C(\|(1+n)|\nabla h|^2\|_{L^2}^2+\|(1+n)h\nabla^2 h\|_{L^2}^2)\\
\leq&C(1+\|\nabla h\|_{L^p}^2\|\dot{\mathbf{u}}\|_{L^q}^2+\|h\|_{L^{\infty}}^2\|\nabla\dot{\mathbf{u}}\|_{L^2}^2+\|\nabla h^2\|_{L^4}^2\|\nabla n\|_{L^4}^2)\\
&+C(\|h\|_{L^{\infty}}^4\|\nabla^2 n\|_{L^2}^2+(1+\|n\|_{L^{\infty}}^2)\|\nabla h\|_{L^4}^4+\|(1+n)h\|_{L^{\infty}}^2\|\nabla^2 h\|_{L^2}^2)\\
\leq&C(1+\||x|^{\alpha/2}\nabla\dot{\mathbf{u}}\|_{L^2}^2+\|\nabla\dot{\mathbf{u}}\|_{L^2}^2+\|\nabla^2 h^2\|_{L^2}^2).
\end{split}
\end{equation}
Using the Gronwall's inequality to (\ref{148}), we have
\begin{equation}\label{149}
\sup_{0\leq t\leq T^*}(\|\nabla^2 h\|_{L^2}^2+\|\nabla^2 h^2\|_{L^2}^2)\leq C.
\end{equation}
Integrating the estimate (\ref{147}) with respect to $t$, together with (\ref{149}) yields 
\begin{equation}\label{150}
\int_{0}^{T^*}\|\nabla^3\mathbf{u}\|_{L^2}^2\, dt\leq C.
\end{equation}
The estimate (\ref{z107}) can be given by (\ref{149}) and (\ref{150}).
\end{proof}

\begin{lemma}\label{s}
The following estimate holds
\begin{equation}\label{z108}
\sup_{0\leq t\leq T^*}\|\nabla n_{t}\|_{L^2}^2+\int_{0}^{T^*}\|n_{tt}\|_{L^2}^2\, dt\leq C.
\end{equation}
\end{lemma}

\begin{proof}
Setting $\mathbf{v}=\mathbf{u}$, then the method to deal with (\ref{z108}) is similar as the way to get the estimate (\ref{117}). We need to establish additionally the estimate of $-\int \nabla\cdot(n\nabla c)_{t}n_{tt}\, dx$, i.e.
\begin{equation}\label{C102}
\begin{split}
-\int \nabla\cdot(n\nabla c)_{t}n_{tt}\, dx\leq&C(\|\nabla n_{t}\|_{L^2}\|\nabla c\|_{L^{\infty}}\|n_{tt}\|_{L^2}+\|\nabla n\|_{L^4}\|\nabla c_{t}\|_{L^4}\|n_{tt}\|_{L^2})\\
&+C(\|n_{t}\|_{L^4}\|\Delta c\|_{L^4}\|n_{tt}\|_{L^2}+\|n\|_{L^{\infty}}\|\Delta c_{t}\|_{L^2}\|n_{tt}\|_{L^2})\\
\leq&\varepsilon\|n_{tt}\|_{L^2}^2+C_{\varepsilon}\|\nabla c\|_{L^{\infty}}^2\|\nabla n_{t}\|_{L^2}^2\\
&+C_{\varepsilon}(\|\nabla n\|_{L^4}^2\|\nabla c_{t}\|_{L^4}^2+\|n_{t}\|_{L^4}^2\|\Delta c\|_{L^4}^2+\|n\|_{L^{\infty}}^2\|\Delta c_{t}\|_{L^2}^2).
\end{split}
\end{equation}
Then using the Gronwall's inequality, we can get (\ref{z108}).
\end{proof}

\begin{lemma}\label{o}
The following estimate holds
\begin{equation}\label{z109}
\sup_{0\leq t\leq T^*}(\|\nabla^3 c\|_{L^2}^2+\|\nabla^3 n\|_{L^2}^2)+\int_{0}^{T^*}\|\nabla^4 c\|_{L^2}^2\, dt\leq C.
\end{equation}
\end{lemma}

\begin{proof}
Using the $L^p$ theory to equation $(\ref{A})_{2}$, it is easy to get
\begin{equation}\label{151}
\|\nabla^2 c\|_{H^1}^2\leq C(\|c_{t}\|_{H^1}^2+\|\dive(c\mathbf{u})\|_{H^1}^2+\|nc\|_{H^1}^2)\leq C
\end{equation}
and
\begin{equation}\label{z108}
\begin{split}
&\int_{0}^{T^*}\|\nabla^2 c\|_{H^2}^2\, dt\\
\leq&C\int_{0}^{T^*}(\|c_{t}\|_{H^2}^2+\|\dive(c\mathbf{u})\|_{H^2}^2+\|nc\|_{H^2}^2)\, dt\\
\leq&C\int_{0}^{T^*}(1+\|\nabla^2 c_{t}\|_{L^2}^2+\|\mathbf{u}\|_{L^{\infty}}^2\|\nabla^3 c\|_{L^2}^2+\|\nabla\mathbf{u}\|_{L^4}^2\|\nabla^2 c\|_{L^4}^2+\|c\|_{L^{\infty}}^2\|\nabla^3\mathbf{u}\|_{L^2}^2)\, dt\\
\leq&C.
\end{split}
\end{equation}
Based on the estimate (\ref{151}) and Lemma \ref{s}, using the $L^p$ theory to equation $(\ref{A})_{1}$ to obtain
\begin{equation}\label{152}
\begin{split}
\|\nabla^2 n\|_{H^1}^2\leq&C(\|n_{t}\|_{H^1}^2+\|\dive(n\mathbf{u})\|_{H^1}^2+\|\nabla\cdot(n\nabla c)\|_{H^1}^2)\\
\leq&C(1+\|\nabla n_{t}\|_{L^2}^2+\|\nabla n\|_{L^4}^2\|\nabla^2 c\|_{L^4}^2+\|n\|_{L^{\infty}}^2\|\nabla^3 c\|_{L^2}^2)\\
\leq&C.
\end{split}
\end{equation}
The estimate (\ref{z109}) can be given by (\ref{151})-(\ref{152}).
\end{proof}

\begin{lemma}\label{p}
The following estimate holds
\begin{equation}\label{207}
\int_{0}^{T^*}(\|\nabla^2 n_{t}\|_{L^2}^2+\|\nabla^4 n\|_{L^2}^2)\, dt\leq C.
\end{equation}
\end{lemma}

\begin{proof}
Setting $\mathbf{v}=\mathbf{u}$ in (\ref{178}), then the estimate of $\int_{0}^{T^*}\|\nabla^2 n_{t}\|_{L^2}^2\, dt$ is the same as (\ref{178}). We need to check $\int_{0}^{T^*}\|\nabla\cdot(n\nabla c)_{t}\|_{L^2}^2\, dt$. Through Lemma \ref{o}, we have
\begin{equation*}
\begin{split}
&\int_{0}^{T^*}\|\nabla\cdot(n\nabla c)_{t}\|_{L^2}^2\, dt\\
\leq&C\int_{0}^{T^*}(\|\nabla n_{t}\|_{L^2}^2\|\nabla c\|_{L^{\infty}}^2+\|n_{t}\|_{L^4}^2\|\Delta c\|_{L^4}^2+\|\nabla n\|_{L^{\infty}}^2\|\nabla c_{t}\|_{L^2}^2+\|n\|_{L^{\infty}}^2\|\Delta c_{t}\|_{L^2}^2)\, dt\\
\leq&C.
\end{split}
\end{equation*}
Combining the above estimate and (\ref{178}), we get $\int_{0}^{T^*}\|\nabla^2 n_{t}\|_{L^2}^2\, dt\leq C$. Using the $L^p$ theory to equation $(\ref{A})_{1}$, together with Lemmas \ref{n} and \ref{o} to get
\begin{equation}\label{B400}
\begin{split}
&\int_{0}^{T^*}\|\nabla^2 n\|_{H^2}^2\, dt\\
\leq&C\int_{0}^{T^*}(\|n_{t}\|_{H^2}^2+\|\dive(n\mathbf{u})\|_{H^2}^2+\|\nabla\cdot(n\nabla c)\|_{H^2}^2)\, dt\\
\leq&C\int_{0}^{T^*}(1+\|\nabla^2 n_{t}\|_{L^2}^2+\|\mathbf{u}\|_{L^{\infty}}^2\|\nabla^3 n\|_{L^2}^2+\|\nabla\mathbf{u}\|_{L^4}^2\|\nabla^2 n\|_{L^4}^2+\|\nabla n\|_{L^{\infty}}^2\|\nabla^2\mathbf{u}\|_{L^2}^2)\, dt\\
&+C\int_{0}^{T^*}(\|n\|_{L^{\infty}}^2\|\nabla^3\mathbf{u}\|_{L^2}^2+\|\nabla^3 n\|_{L^2}^2\|\nabla c\|_{L^{\infty}}^2+\|\nabla^2 n\|_{L^4}^2\|\nabla^2 c\|_{L^4}^2)\, dt\\
&+C\int_{0}^{T^*}(\|\nabla n\|_{L^{\infty}}^2\|\nabla^3 c\|_{L^2}^2+\|n\|_{L^{\infty}}^2\|\nabla^4 c\|_{L^2}^2)\, dt\\
\leq&C.
\end{split}
\end{equation}
\end{proof}

\begin{lemma}\label{q}
The following estimate holds
\begin{equation}\label{210}
\sup_{0\leq t\leq T^*}\|\nabla h_{t}\|_{L^2}^2+\int_{0}^{T^*}\|h_{tt}\|_{L^2}^2\, dt\leq C.
\end{equation}
\end{lemma}

\begin{proof}
Taking the operator $\nabla$ on equation $(\ref{A})_{3}$, together with Lemma \ref{n} to get
\begin{equation}\label{156}
\|\nabla h_{t}\|_{L^2}^2\leq C(\|\mathbf{u}\|_{L^{\infty}}^2\|\nabla^2 h\|_{L^2}^2+\|\nabla\mathbf{u}\|_{L^4}^2\|\nabla h\|_{L^4}^2+\|h\|_{L^{\infty}}^2\|\nabla^2\mathbf{u}\|_{L^2}^2)\leq C.
\end{equation}
Taking the derivative with respective to $t$ on equation $(\ref{A})_{3}$, together with (\ref{156}) and Lemma \ref{n}, we have
\begin{equation}\label{157}
\begin{split}
&\int_{0}^{T^*}\|h_{tt}\|_{L^2}^2\, dt\\
\leq&C\int_{0}^{T^*}(\|\mathbf{u}\|_{L^{\infty}}^2\|\nabla h_{t}\|_{L^2}^2+\|h_{t}\|_{L^2}^2\|\dive\mathbf{u}\|_{L^{\infty}}^2+\|\mathbf{u}\|_{L^q}^2\|\nabla h\|_{L^p}^2+\|h\|_{L^{\infty}}^2\|\dive\mathbf{u}_{t}\|_{L^2}^2)\, dt\\
\leq&C,
\end{split}
\end{equation}
where $q=4/\alpha, 1/q+1/p=1/2$. The estimates (\ref{156}) and (\ref{157}) yield (\ref{210}).
\end{proof}

\begin{lemma}\label{r}
The following estimate holds
\begin{equation}\label{z211}
\sup_{0\leq t\leq T^*}\|\nabla\mathbf{u}_{t}\|_{L^2}^2+\int_{0}^{T^*}(\|\sqrt{h}\mathbf{u}_{tt}\|_{L^2}^2+\|\nabla^2\mathbf{u}_{t}\|_{L^2}^2)\, dt\leq C.
\end{equation}
\end{lemma}

\begin{proof}
Taking the equation $(\ref{A})_{4}$ with respect to $t$, multiplying the resulting equation by $\mathbf{u}_{tt}$ and integrating over $\mathbb{R}^2$ to get
\begin{equation*}
\begin{split}
&\frac{\mu}{2}\frac{d}{dt}\int |\nabla\mathbf{u}_{t}|^2\, dx+\frac{(\mu+\lambda)}{2}{\frac{d}{dt}}\int(\dive\mathbf{u}_{t})^2\, dx+\int h\mathbf{u}_{tt}^2\, dx\\
=&\frac{d}{dt}\big(-\frac{1}{2}\int h_{t}|\mathbf{u}_{t}|^2\, dx-\int h_{t}\mathbf{u}\cdot\nabla\mathbf{u}\cdot\mathbf{u}_{t}\, dx+\frac{1}{2}\int\big((1+n)h^2\big)_{t}\dive\mathbf{u}_{t}\, dx\big)\\
&+\frac{1}{2}\int h_{tt}|\mathbf{u}_{t}|^2\, dx+\int (h_{t}\mathbf{u}\cdot\nabla\mathbf{u})_{t}\cdot\mathbf{u}_{t}\, dx-\int h\mathbf{u}_{t}\cdot\nabla\mathbf{u}\cdot\mathbf{u}_{tt}\, dx-\int h\mathbf{u}\cdot\nabla\mathbf{u}_{t}\cdot\mathbf{u}_{tt}\, dx\\
&-\frac{1}{2}\int\big((1+n)h^2\big)_{tt}\dive\mathbf{u}_{t}\, dx-\frac{1}{2}\int (h^2\nabla n)_{t}\cdot\mathbf{u}_{tt}\, dx\\
:=&\frac{d}{dt}J_{1}+\sum_{i=2}^{7}J_{i}.
\end{split}
\end{equation*}
In the following, we estimate $J_{i}\ (i=1,\cdots, 7)$ respectively. First we consider $J_{1}$. Through the equation $(\ref{A})_{3}$ and integrating by parts, we have
\begin{equation*}
\begin{split}
J_{1}\leq&\|h\|_{L^{\infty}}^{\frac{1}{2}}\|\mathbf{u}\|_{L^{\infty}}\|\sqrt{h}\mathbf{u}_{t}\|_{L^2}\|\nabla\mathbf{u}_{t}\|_{L^2}+\|\sqrt{h}\mathbf{u}\|_{L^{\infty}}\|\sqrt{h}\mathbf{u}_{t}\|_{L^2}\|\nabla\mathbf{u}\|_{L^4}^2\\
&+\|h\|_{L^{\infty}}^{\frac{1}{2}}\|\sqrt{h}\mathbf{u}_{t}\|_{L^2}\|\nabla^2\mathbf{u}\|_{L^2}\|\mathbf{u}\|_{L^{\infty}}^2+\|h\mathbf{u}^2\|_{L^{\infty}}\|\nabla\mathbf{u}\|_{L^2}\|\nabla\mathbf{u}_{t}\|_{L^2}\\
&+\|n_{t}\|_{L^2}\|h\|_{L^{\infty}}^2\|\nabla\mathbf{u}_{t}\|_{L^2}+\|(1+n)h\|_{L^{\infty}}\|h_{t}\|_{L^2}\|\dive\mathbf{u}_{t}\|_{L^2}\\
\leq&\eta\|\nabla\mathbf{u}_{t}\|_{L^2}^2+C_{\eta}.
\end{split}
\end{equation*}
The estimates of $J_{2}$ to $J_{5}$ can refer to \cite{L2012}. Here we only give the result. i.e.
\begin{equation*}
J_{2}+J_{3}+J_{4}+J_{5}\leq\varepsilon\|\sqrt{h}\mathbf{u}_{tt}\|_{L^2}^2+C_{\varepsilon}(\|\mathbf{u}_{t}\|_{L^q}^2+\|\nabla\mathbf{u}_{t}\|_{L^2}^4+\|h_{tt}\|_{L^2}^2+1).
\end{equation*}
Then we estimate $J_{6}$. By Lemmas \ref{s} and \ref{q}, we have
\begin{equation*}
\begin{split}
J_{6}\leq&C(\|n_{tt}\|_{L^2}\|h\|_{L^{\infty}}^2\|\dive\mathbf{u}_{t}\|_{L^2}+\|n_{t}\|_{L^4}\|h\|_{L^{\infty}}\|h_{t}\|_{L^4}\|\dive\mathbf{u}_{t}\|_{L^2}\\
&+\|(1+n)\|_{L^{\infty}}\|h_{t}\|_{L^4}^2\|\dive\mathbf{u}_{t}\|_{L^2}+\|(1+n)h\|_{L^{\infty}}\|h_{tt}\|_{L^2}\|\dive\mathbf{u}_{t}\|_{L^2})\\
\leq&C(1+\|\dive\mathbf{u}_{t}\|_{L^2}^2+\|n_{tt}\|_{L^2}^2+\|h_{tt}\|_{L^2}^2).
\end{split}
\end{equation*}
Finally, we estimate $J_{7}$. Using Lemmas \ref{s} and \ref{q} shows that
\begin{equation*}
\begin{split}
J_{7}\leq&C(\|\sqrt{h}\mathbf{u}_{tt}\|_{L^2}\|h_{t}\|_{L^4}\|\nabla n\|_{L^4}\|\sqrt{h}\|_{L^{\infty}}+\|\sqrt{h}\mathbf{u}_{tt}\|_{L^2}\|\nabla n_{t}\|_{L^2}\|\sqrt{h}\|_{L^{\infty}}^3)\\
\leq&\varepsilon\|\sqrt{h}\mathbf{u}_{tt}\|_{L^2}^2+C_{\varepsilon}.
\end{split}
\end{equation*}
Combining all these estimates of $J_{i}\ (i=1,\cdots, 7)$, choosing $\eta=\mu/4$ and $\varepsilon=1/4$, we have
\begin{equation}\label{211}
\begin{split}
&\frac{d}{dt}\int |\nabla\mathbf{u}_{t}|^2\, dx+\frac{d}{dt}\int(\dive\mathbf{u}_{t})^2\, dx+\int h\mathbf{u}_{tt}^2\, dx\\
\leq&C(\|\nabla\mathbf{u}_{t}\|_{L^2}^4+\|\mathbf{u}_{t}\|_{L^q}^2+\|h_{tt}\|_{L^2}^2+\|n_{tt}\|_{L^2}^2+1).
\end{split}
\end{equation}
Taking into account the compatibility condition (\ref{A3}), we can define
\begin{equation*}
\dot{\mathbf{u}}(x, t=0)=g,
\end{equation*}
which implies
\begin{equation*}
\begin{split}
\|\nabla\mathbf{u}_{t}(\cdot, 0)\|_{L^2}^2=&\|\nabla\dot{\mathbf{u}}(\cdot, 0)-\nabla(\mathbf{u}_{0}\cdot\nabla\mathbf{u}_{0})\|_{L^2}^2\\
\leq&\|\nabla g\|_{L^2}^2+C(\|\nabla\mathbf{u}_{0}\|_{L^2}^2+\|\mathbf{u}_{0}\|_{L^{\infty}}^2)\|\nabla^2\mathbf{u}_{0}\|_{L^2}^2\\
\leq&\|\nabla g\|_{L^2}^2+C.
\end{split}
\end{equation*}
Using the Gronwall's inequality to inequality (\ref{211}), together with Lemmas \ref{s} and \ref{q}, we have
\begin{equation}\label{158}
\sup_{0\leq t\leq T^*}\|\nabla\mathbf{u}_{t}\|_{L^2}^2+\int_{0}^{T^*}\|\sqrt{h}\mathbf{u}_{tt}\|_{L^2}^2\, dt\leq C.
\end{equation}
Differentiating the equation $(\ref{A})_{4}$ with respect to $t$, using the $L^p$ estimate to the resulting equation, together with Lemma \ref{n}, \ref{q}, and (\ref{158}) to get
\begin{equation*}
\begin{split}
&\int_{0}^{T^*}\|\nabla^2\mathbf{u}_{t}\|_{L^2}^2\, dt\\
\leq&C\int_{0}^{T^*}(\|h_{t}\|_{L^p}^2\|\mathbf{u}_{t}\|_{L^q}^2+\|h\|_{L^{\infty}}\|\sqrt{h}\mathbf{u}_{tt}\|_{L^2}^2+\|h_{t}\|_{L^2}^2\|\mathbf{u}\|_{L^{\infty}}^2\|\nabla\mathbf{u}\|_{L^2}^2+\|h\|_{L^{\infty}}^2\|\mathbf{u}_{t}\|_{L^q}^2\|\nabla\mathbf{u}\|_{L^p}^2)\, dt\\
&+C\int_{0}^{T^*}(\|h\|_{L^{\infty}}^2\|\mathbf{u}\|_{L^{\infty}}^2\|\nabla\mathbf{u}_{t}\|_{L^2}^2+\|h\|_{L^{\infty}}^2\|h_{t}\|_{L^4}^2\|\nabla n\|_{L^4}^2+\|h\|_{L^{\infty}}^4\|\nabla n_{t}\|_{L^2}^2)\, dt\\
&+C\int_{0}^{T^*}(\|n_{t}\|_{L^4}^2\|h\|_{L^{\infty}}^2\|\nabla h\|_{L^4}^2+\|(1+n)\|_{L^{\infty}}^2\|h_{t}\|_{L^2}^2\|\nabla h\|_{L^2}^2+\|(1+n)h\|_{L^{\infty}}^2\|\nabla h_{t}\|_{L^2}^2)\, dt\\
\leq&C,
\end{split}
\end{equation*}
where $q=4/\alpha, 1/q+1/p=1/2$. The above estimate combined with (\ref{158}) gives (\ref{z211}).
\end{proof}

\begin{lemma}\label{t}
The following estimate holds
\begin{equation}\label{219}
\sup_{0\leq t\leq T^*}(\|\nabla^3 h\|_{L^2}^2+\|\nabla^2 h_{t}\|_{L^2}^2+\|\nabla^3\mathbf{u}\|_{L^2}^2)+\int_{0}^{T^*}\|\nabla^4\mathbf{u}\|_{L^2}^2\, dt\leq C.
\end{equation}
\end{lemma}

\begin{proof}
Taking the operator $\nabla^3$ on equation $(\ref{A})_{3}$, using Lemma \ref{n}, it is easy to get
\begin{equation}\label{214}
\begin{split}
\frac{d}{dt}\|\nabla^3 h\|_{L^2}^2\leq&C(\|\nabla^3\mathbf{u}\|_{L^2}^2+\|\nabla\mathbf{u}\|_{L^{\infty}}+1)\|\nabla^3 h\|_{L^2}^2+C(\|\nabla^3\mathbf{u}\|_{L^2}^2+\|\nabla^4\mathbf{u}\|_{L^2}^2+1)\\
\leq&C(\|\nabla^3\mathbf{u}\|_{L^2}^2+\|\nabla\mathbf{u}\|_{L^{\infty}}+\|\mathbf{u}_{t}\|_{L^q}^2+1)\|\nabla^3 h\|_{L^2}^2\\
&+C(\|\nabla^3\mathbf{u}\|_{L^2}^2+\|\mathbf{u}_{t}\|_{L^q}^2+\|\nabla^2\mathbf{u}_{t}\|_{L^2}^2+1),
\end{split}
\end{equation}
where we have used the following fact that
\begin{equation}\label{215}
\begin{split}
&\|\nabla^4\mathbf{u}\|_{L^2}^2\\
\leq&C(\|\mathbf{u}_{t}\|_{L^q}^2\|\nabla^2 h\|_{L^p}^2+\|\nabla h\|_{L^4}^2\|\nabla\mathbf{u}_{t}\|_{L^4}^2+\|h\|_{L^{\infty}}^2\|\nabla^2\mathbf{u}_{t}\|_{L^2}^2+\|\nabla^2 h\|_{L^2}^2\|\mathbf{u}\|_{L^{\infty}}^2\|\nabla\mathbf{u}\|_{L^{\infty}}^2)\\
&+C(\|\nabla h\|_{L^4}^2\|\nabla\mathbf{u}\|_{L^8}^4+\|\nabla h\|_{L^4}^2\|\mathbf{u}\|_{L^{\infty}}^2\|\nabla^2\mathbf{u}\|_{L^4}^2+\|h\|_{L^{\infty}}^2\|\nabla\mathbf{u}\|_{L^{\infty}}^2\|\nabla^2\mathbf{u}\|_{L^2}^2)\\
&+C(\|h\mathbf{u}\|_{L^{\infty}}^2\|\nabla^3\mathbf{u}\|_{L^2}^2+\|\nabla^2 h^2\|_{L^2}^2\|\nabla n\|_{L^{\infty}}^2+\|\nabla h^2\|_{L^4}^2\|\nabla^2 n\|_{L^4}^2)\\
&+C(\|h\|_{L^{\infty}}^4\|\nabla^3 n\|_{L^2}^2+\|\nabla^2 n\|_{L^4}^2\|h\|_{L^{\infty}}^2\|\nabla h\|_{L^4}^2+\|\nabla n\|_{L^{\infty}}^2\|\nabla h\|_{L^4}^4)\\
&+C(\|(1+n)h\|_{L^{\infty}}^2\|\nabla^2 h\|_{L^2}^2)\\
\leq&C(\|\mathbf{u}_{t}\|_{L^q}^2+\|\nabla^2\mathbf{u}_{t}\|_{L^2}^2+\|\mathbf{u}_{t}\|_{L^q}^2\|\nabla^3 h\|_{L^2}^2+\|\nabla^3\mathbf{u}\|_{L^2}^2+1).
\end{split}
\end{equation}
Using the Gronwall's inequality to (\ref{214}), together with Lemmas \ref{n} and \ref{r} to get
\begin{equation}\label{161}
\sup_{0\leq t\leq T^*}\|\nabla^3 h\|_{L^2}^2\leq C.
\end{equation}
Furthermore, integrating (\ref{215}) with respect to $t$, combining with Lemma \ref{n}, \ref{r}, and (\ref{161}) to get
\begin{equation}\label{218}
\int_{0}^{T^*}\|\nabla^4\mathbf{u}\|_{L^2}^2\, dt\leq C.
\end{equation}
Differentiating the equation $(\ref{A})_{4}$ with respect to $t$, multiplying the resulting equation by $\nabla^4\mathbf{u}$ and integrating over $\mathbb{R}^2$ to get
\begin{equation}\label{220}
\begin{split}
&\frac{\mu}{2}\frac{d}{dt}\int |\nabla^3\mathbf{u}|^2\, dx+\frac{(\mu+\lambda)}{2}\frac{d}{dt}\int |\nabla^2(\dive\mathbf{u})|^2\, dx\\
\leq&C(\|\nabla^4\mathbf{u}\|_{L^2}^2+\|h_{t}\|_{L^p}^2\|\mathbf{u}_{t}\|_{L^q}^2+\|h\|_{L^{\infty}}\|\sqrt{h}\mathbf{u}_{tt}\|_{L^2}^2+\|h_{t}\|_{L^4}^2\|\mathbf{u}\|_{L^{\infty}}^2\|\nabla\mathbf{u}\|_{L^4}^2)\\
&+C(\|h\|_{L^{\infty}}\|\sqrt{h}\mathbf{u}_{t}\|_{L^2}^2\|\nabla\mathbf{u}\|_{L^{\infty}}^2+\|h\mathbf{u}\|_{L^{\infty}}^2\|\nabla\mathbf{u}_{t}\|_{L^2}^2+\|h\|_{L^{\infty}}^2\|h_{t}\|_{L^4}^2\|\nabla n\|_{L^4}^2)\\
&+C(\|h\|_{L^{\infty}}^4\|\nabla n_{t}\|_{L^2}^2+\|n_{t}\|_{L^4}^2\|h\|_{L^{\infty}}^2\|\nabla h\|_{L^4}^2+\|(1+n)\|_{L^{\infty}}^2\|h_{t}\|_{L^4}^2\|\nabla h\|_{L^4}^2)\\
&+C\|(1+n)h\|_{L^{\infty}}^2\|\nabla h_{t}\|_{L^2}^2,
\end{split}
\end{equation}
 where $q=4/\alpha, 1/q+1/p=1/2$. According to (\ref{158}), (\ref{218}), Lemma \ref{n}, \ref{s}, and \ref{q}, we have
\begin{equation}\label{217}
\sup_{0\leq t\leq T^*}\|\nabla^3\mathbf{u}\|_{L^2}^2\leq C.
\end{equation}
Taking the operator $\nabla^2$ on equation $(\ref{A})_{3}$, combining with (\ref{161}), (\ref{217}), and Lemma \ref{n} to get
\begin{equation}\label{162}
\begin{split}
&\sup_{0\leq t\leq T^*}\|\nabla^2 h_{t}\|_{L^2}^2\\
\leq&C(\|\nabla^2 h\|_{L^4}^2\|\nabla\mathbf{u}\|_{L^4}^2+\|\nabla h\|_{L^{\infty}}^2\|\nabla^2\mathbf{u}\|_{L^2}^2+\|h\|_{L^{\infty}}^2\|\nabla^3\mathbf{u}\|_{L^2}^2+\|\mathbf{u}\|_{L^{\infty}}^2\|\nabla^3 h\|_{L^2}^2)\\
\leq&C,
\end{split}
\end{equation}
which together with (\ref{161}), (\ref{218}), and (\ref{217}) yield the estimate (\ref{219}).
\end{proof}

\begin{lemma}\label{u}
The following estimate holds
\begin{equation}\label{z201}
\sup_{0\leq t\leq T^*} t(\|\sqrt{h}\mathbf{u}_{tt}\|_{L^2}^2+\|\mathbf{u}_{t}\|_{L^{8/\alpha}\cap D^2}^2+\|\nabla^4\mathbf{u}\|_{L^2}^2)+\int_{0}^{T^*}t\|\nabla\mathbf{u}_{tt}\|_{L^2}^2\, dt\leq C.
\end{equation}
\end{lemma}

\begin{proof}
Differentiating the equation $(\ref{A})_{4}$ with respect to $t$ twice to get 
\begin{equation}\label{163}
\begin{split}
&h\mathbf{u}_{ttt}+h\mathbf{u}\cdot\nabla\mathbf{u}_{tt}-\mu\Delta\mathbf{u}_{tt}-(\mu+\lambda)\nabla(\dive\mathbf{u})_{tt}\\
=&2\dive(h\mathbf{u})\mathbf{u}_{tt}+\dive(h\mathbf{u})_{t}\mathbf{u}_{t}-(h_{tt}\mathbf{u}+2h_{t}\mathbf{u}_{t})\cdot\nabla\mathbf{u}-2(h\mathbf{u})_{t}\cdot\nabla\mathbf{u}_{t}-h\mathbf{u}_{tt}\cdot\nabla\mathbf{u}\\
&-2h_{t}^2\nabla n-2hh_{tt}\nabla n-4hh_{t}\nabla n_{t}-h^2\nabla n_{tt}-n_{tt}h\nabla h-2n_{t}h_{t}\nabla h-2n_{t}h\nabla h_{t}\\
&-(1+n)h_{tt}\nabla h-2(1+n)h_{t}\nabla h_{t}-(1+n)h\nabla h_{tt}.
\end{split}
\end{equation}
Multiplying the equation (\ref{163}) by $\mathbf{u}_{tt}$ and integrating over $\mathbb{R}^2$ to get
\begin{equation*}
\begin{split}
&\frac{1}{2}\frac{d}{dt}\int h\mathbf{u}_{tt}^2\, dx+\mu\int|\nabla\mathbf{u}_{tt}|^2\, dx+(\mu+\lambda)\int(\dive\mathbf{u}_{tt})^2\, dx\\
=&2\int\dive(h\mathbf{u})\mathbf{u}_{tt}^2\, dx+\int\dive(h\mathbf{u})_{t}\mathbf{u}_{t}\mathbf{u}_{tt}\, dx-\int h_{tt}\mathbf{u}\cdot\nabla\mathbf{u}\mathbf{u}_{tt}\, dx-2\int h_{t}\mathbf{u}_{t}\cdot\nabla\mathbf{u}\mathbf{u}_{tt}\, dx\\
&-2\int(h\mathbf{u})_{t}\cdot\nabla\mathbf{u}_{t}\mathbf{u}_{tt}\, dx-\int h\mathbf{u}_{tt}^2\cdot\nabla\mathbf{u}\, dx-2\int h_{t}^2\mathbf{u}_{tt}\cdot\nabla n\, dx-2\int hh_{tt}\mathbf{u}_{tt}\cdot\nabla n\, dx\\
&-4\int hh_{t}\mathbf{u}_{tt}\cdot\nabla n_{t}\, dx-\int h^2\mathbf{u}_{tt}\cdot\nabla n_{tt}\, dx-\int n_{tt}h\mathbf{u}_{tt}\cdot\nabla h\, dx-2\int n_{t}h_{t}\mathbf{u}_{tt}\cdot\nabla h\, dx\\
&-2\int n_{t}h\mathbf{u}_{tt}\cdot\nabla h_{t}\, dx-\int (1+n)h_{tt}\mathbf{u}_{tt}\cdot\nabla h\, dx-2\int(1+n)h_{t}\mathbf{u}_{tt}\cdot\nabla h_{t}\, dx\\
&-\int (1+n)h\mathbf{u}_{tt}\cdot\nabla h_{tt}\, dx\\
:=&\sum_{i=1}^{16}K_{i}.
\end{split}
\end{equation*}
The estimates of $K_{1}$ to $K_{6}$ can refer to Lemma 4.6 in \cite{L2012}. Here we only list the result. We have
\begin{equation*}
\begin{split}
&K_{1}+K_{2}+K_{3}+K_{4}+K_{5}+K_{6}\\
\leq&\varepsilon\|\nabla\mathbf{u}_{tt}\|_{L^2}^2+C_{\varepsilon}\|\sqrt{h}\mathbf{u}_{tt}\|_{L^2}^2+C_{\varepsilon}(\|\mathbf{u}_{t}\|_{L^q}^2+\|\nabla\mathbf{u}_{t}\|_{H^1}^2+1).
\end{split}
\end{equation*}
Using Lemma \ref{s}, \ref{o}, and \ref{q}, a straightforward computation shows that
\begin{equation*}
K_{7}\leq\|\sqrt{h}\mathbf{u}_{tt}\|_{L^2}\|\sqrt{h}\|_{L^{\infty}}\|h_{t}\|_{L^4}\|\nabla n\|_{L^4}\leq C(\|\sqrt{h}\mathbf{u}_{tt}\|_{L^2}^2+1),
\end{equation*}
\begin{equation*}
K_{8}\leq\|\sqrt{h}\mathbf{u}_{tt}\|_{L^2}\|\sqrt{h}\|_{L^{\infty}}\|h_{tt}\|_{L^2}\|\nabla n\|_{L^{\infty}}\leq C(\|\sqrt{h}\mathbf{u}_{tt}\|_{L^2}^2+\|h_{tt}\|_{L^2}^2),
\end{equation*}
and
\begin{equation*}
K_{9}\leq\|\sqrt{h}\mathbf{u}_{tt}\|_{L^2}\|h_{t}\|_{L^4}\|\nabla n_{t}\|_{L^4}\|\sqrt{h}\|_{L^{\infty}}\leq C(\|\sqrt{h}\mathbf{u}_{tt}\|_{L^2}^2+\|\nabla^2 n_{t}\|_{L^2}^2).
\end{equation*}
Then we estimate $K_{10}+K_{11}$. Through the integrating by parts and Lemma \ref{t}, we have
\begin{equation*}
\begin{split}
K_{10}+K_{11}=&\int h\mathbf{u}_{tt}\cdot\nabla h n_{tt}\, dx+\int h^2\dive\mathbf{u}_{tt} n_{tt}\, dx\\
\leq&\|\sqrt{h}\mathbf{u}_{tt}\|_{L^2}\|\sqrt{h}\|_{L^{\infty}}\|\nabla h\|_{L^{\infty}}\|n_{tt}\|_{L^2}+\|h\|_{L^{\infty}}^2\|\dive\mathbf{u}_{tt}\|_{L^2}\|n_{tt}\|_{L^2}\\
\leq&\varepsilon\|\dive\mathbf{u}_{tt}\|_{L^2}^2+C_{\varepsilon}(\|\sqrt{h}\mathbf{u}_{tt}\|_{L^2}^2+\|n_{tt}\|_{L^2}^2).
\end{split}
\end{equation*}
Then we estimate $K_{12}$. By Lemma \ref{s}, \ref{q}, and \ref{t}, we have
\begin{equation*}
\begin{split}
K_{12}=&2\int\mathbf{u}_{tt}\cdot\nabla n_{t} h_{t}h\, dx+2\int \mathbf{u}_{tt}\cdot\nabla h_{t}n_{t}h\, dx+2\int n_{t}h_{t}\dive\mathbf{u}_{tt}h\, dx\\
\leq&C(\|\sqrt{h}\mathbf{u}_{tt}\|_{L^2}\|h_{t}\|_{L^4}\|\nabla n_{t}\|_{L^4}\|\sqrt{h}\|_{L^{\infty}}+\|\sqrt{h}\mathbf{u}_{tt}\|_{L^2}\|\sqrt{h}\|_{L^{\infty}}\|n_{t}\|_{L^4}\|\nabla h_{t}\|_{L^4}\\
&+\|\dive\mathbf{u}_{tt}\|_{L^2}\|n_{t}\|_{L^4}\|h_{t}\|_{L^4}\|h\|_{L^{\infty}})\\
\leq&\varepsilon\|\dive\mathbf{u}_{tt}\|_{L^2}^2+C_{\varepsilon}(\|\sqrt{h}\mathbf{u}_{tt}\|_{L^2}^2+\|\nabla^2 n_{t}\|_{L^2}^2+1)
\end{split}
\end{equation*}
and
\begin{equation*}
K_{13}\leq\|\sqrt{h}\mathbf{u}_{tt}\|_{L^2}\|n_{t}\|_{L^4}\|\nabla h_{t}\|_{L^4}\|\sqrt{h}\|_{L^{\infty}}\leq C(\|\sqrt{h}\mathbf{u}_{tt}\|_{L^2}^2+1).
\end{equation*}
Then we estimate $K_{14}$. Using Lemmas \ref{o} and \ref{t}, we have
\begin{equation*}
\begin{split}
K_{14}\leq&\int h|\mathbf{u}||\nabla\mathbf{u}||\nabla n||\nabla h||\mathbf{u}_{tt}|\, dx+\int h\mathbf{u}^2|\nabla^2 n||\nabla h||\mathbf{u}_{tt}|\, dx+2\int h\mathbf{u}^2|\nabla n||\nabla^2 h||\mathbf{u}_{tt}|\, dx\\
&+\int h\mathbf{u}^2|\nabla n||\nabla h||\nabla\mathbf{u}_{tt}|\, dx+\int h|\mathbf{u}||\nabla\mathbf{u}|(1+n)|\nabla^2 h||\mathbf{u}_{tt}|\, dx+\int h\mathbf{u}^2(1+n)|\nabla^3 h||\mathbf{u}_{tt}|\, dx\\
&+\int h\mathbf{u}^2(1+n)|\nabla^2 h||\nabla\mathbf{u}_{tt}|\, dx+\int h_{t}|\mathbf{u}|(1+n)|\nabla h||\nabla\mathbf{u}_{tt}|\, dx+\int h|\mathbf{u}_{t}||\nabla n||\nabla h||\mathbf{u}_{tt}|\, dx\\
&+\int h|\mathbf{u}_{t}|(1+n)|\nabla^2 h||\mathbf{u}_{tt}|\, dx+\int h|\mathbf{u}_{t}|(1+n)|\nabla h||\nabla\mathbf{u}_{tt}|\, dx\\
\leq&\varepsilon\|\nabla\mathbf{u}_{tt}\|_{L^2}^2+C_{\varepsilon}(\|\sqrt{h}\mathbf{u}_{tt}\|_{L^2}^2+\|\mathbf{u}_{t}\|_{L^q}^2+1).
\end{split}
\end{equation*}
Then we estimate $K_{15}$. By Lemma \ref{o}, \ref{q}, and \ref{t}, we have
\begin{equation*}
\begin{split}
K_{15}\leq&\int h|\mathbf{u}||\nabla n||\nabla h_{t}||\mathbf{u}_{tt}|\, dx+\int h|\mathbf{u}|(1+n)|\nabla^2 h_{t}||\mathbf{u}_{tt}|\, dx+\int h|\mathbf{u}|(1+n)|\nabla h_{t}||\nabla\mathbf{u}_{tt}|\, dx\\
\leq&\|\sqrt{h}\mathbf{u}_{tt}\|_{L^2}(\|\nabla h_{t}\|_{L^2}\|\sqrt{h}\mathbf{u}\nabla n\|_{L^{\infty}}+\|\nabla^2 h_{t}\|_{L^2}\|\sqrt{h}\mathbf{u}(1+n)\|_{L^{\infty}})\\
&+\|\nabla\mathbf{u}_{tt}\|_{L^2}\|\nabla h_{t}\|_{L^2}\|h\mathbf{u}(1+n)\|_{L^{\infty}}\\
\leq&\varepsilon\|\nabla\mathbf{u}_{tt}\|_{L^2}^2+C_{\varepsilon}(\|\sqrt{h}\mathbf{u}_{tt}\|_{L^2}^2+1).
\end{split}
\end{equation*}
Finally, we estimate $K_{16}$. Using Lemma \ref{n}, \ref{o}, and \ref{t}, we have
\begin{equation*}
\begin{split}
K_{16}\leq&\int h|h_{tt}||\mathbf{u}_{tt}||\nabla n|\, dx+\int h|\mathbf{u}||\nabla\mathbf{u}||\nabla n||\nabla h||\mathbf{u}_{tt}|\, dx+\int h|\mathbf{u}|^2|\nabla^2 n||\nabla h||\mathbf{u}_{tt}|\, dx\\
&+\int h|\mathbf{u}|^2|\nabla n||\nabla^2 h||\mathbf{u}_{tt}|\, dx+\int h|\mathbf{u}|^2|\nabla n||\nabla h||\nabla\mathbf{u}_{tt}|\, dx\\
&+\int h|\mathbf{u}||\nabla\mathbf{u}|(1+n)|\nabla^2 h||\mathbf{u}_{tt}|\, dx+\int h|\mathbf{u}|^2|\nabla n||\nabla^2 h||\mathbf{u}_{tt}|\, dx\\
&+\int h|\mathbf{u}|^2(1+n)|\nabla^3 h||\mathbf{u}_{tt}|\, dx+\int h|\mathbf{u}|^2(1+n)|\nabla^2 h||\nabla\mathbf{u}_{tt}|\, dx\\
&+\int |h_{t}||\mathbf{u}|(1+n)|\nabla h||\nabla\mathbf{u}_{tt}|\, dx+\int h|\mathbf{u}_{t}||\nabla n||\nabla h||\mathbf{u}_{tt}|\, dx\\
&+\int h|\mathbf{u}_{t}|(1+n)|\nabla^2 h||\mathbf{u}_{tt}|\, dx+\int h|\mathbf{u}_{t}|(1+n)|\nabla h||\nabla\mathbf{u}_{tt}|\, dx+\int (1+n)h|h_{tt}||\nabla\mathbf{u}_{tt}|\, dx\\
\leq&\varepsilon\|\nabla\mathbf{u}_{tt}\|_{L^2}^2+C_{\varepsilon}(\|\sqrt{h}\mathbf{u}_{tt}\|_{L^2}^2+\|h_{tt}\|_{L^2}^2+\|\mathbf{u}_{t}\|_{L^q}^2+1).
\end{split}
\end{equation*}
Choosing $\varepsilon=1/6$, and then combining all these estimates about $K_{i}\ (i=1,\cdots, 16)$ to get
\begin{equation}\label{221}
\begin{split}
&\frac{d}{dt}\|\sqrt{h}\mathbf{u}_{tt}\|_{L^2}^2+\|\nabla\mathbf{u}_{tt}\|_{L^2}^2\\
\leq&C(\|\sqrt{h}\mathbf{u}_{tt}\|_{L^2}^2+\|\mathbf{u}_{t}\|_{L^q}^2+\|\nabla^2\mathbf{u}_{t}\|_{L^2}^2+\|\nabla^2 n_{t}\|_{L^2}^2+\|n_{tt}\|_{L^2}^2+\|h_{tt}\|_{L^2}^2+1).
\end{split}
\end{equation}
Multiplying (\ref{221}) by $t$, integrating the resulting equation with respect to $t$, using Lemma \ref{s}, \ref{p}, \ref{q}, and \ref{r} to get
\begin{equation}\label{164}
\sup_{0\leq t\leq T^*}t\|\sqrt{h}\mathbf{u}_{tt}\|_{L^2}^2+\int_{0}^{T^*} t\|\nabla\mathbf{u}_{tt}\|_{L^2}^2\, dt\leq C.
\end{equation}
Setting $\mathbf{u}=\mathbf{v}$ in equation (\ref{B211}), multiplying the resulting equation by $|x|^{\alpha/2}\dot{\mathbf{u}}_{t}$, and then integrating over $\mathbb{R}^2$ to get
\begin{equation*}
\begin{split}
&\frac{\mu}{2}\frac{d}{dt}\int |x|^{\alpha/2}|\nabla\dot{\mathbf{u}}|^2\, dx+\frac{(\mu+\lambda)}{2}\frac{d}{dt}\int |x|^{\alpha/2}(\dive\dot{\mathbf{u}})^2\, dx+\int |x|^{\alpha/2}h|\dot{\mathbf{u}}_{t}|^2\, dx\\
=&-\int |x|^{\alpha/2}h_{t}\dot{\mathbf{u}}\dot{\mathbf{u}}_{t}\, dx-\int |x|^{\alpha/2} h_{t}^2\dot{\mathbf{u}}_{t}\cdot\nabla n\, dx-\int |x|^{\alpha/2} h^2\dot{\mathbf{u}}_{t}\cdot\nabla n_{t}\, dx\\
&-\frac{1}{2}\int |x|^{\alpha/2}n_{t}\dot{\mathbf{u}}_{t}\cdot\nabla h^2\, dx-\frac{1}{2}\int |x|^{\alpha/2}(1+n)\dot{\mathbf{u}}_{t}\cdot\nabla h_{t}^2\, dx-\int |x|^{\alpha/2}\dot{\mathbf{u}}_{t}\dive(h\dot{\mathbf{u}}\otimes\mathbf{u})\, dx\\
&-\int |x|^{\alpha/2}\dot{\mathbf{u}}_{t}\dive(h^2\nabla n\otimes\mathbf{u})\, dx-\frac{1}{2}\int |x|^{\alpha/2}\dot{\mathbf{u}}_{t}\dive\big((1+n)\mathbf{u}\otimes\nabla h^2\big)\, dx\\
&-\mu\int |x|^{\alpha/2}\dot{\mathbf{u}}_{t}\Delta(\mathbf{u}\cdot\nabla\mathbf{u})\, dx-\mu\int\nabla(|x|^{\alpha/2})\dot{\mathbf{u}}_{t}\cdot\nabla\dot{\mathbf{u}}\, dx-(\mu+\lambda)\int\nabla(|x|^{\alpha/2})\dot{\mathbf{u}}_{t}\dive\dot{\mathbf{u}}\, dx\\
&-(\mu+\lambda)\int |x|^{\alpha/2}\nabla\big(\dive(\mathbf{u}\cdot\nabla\mathbf{u})\big)\dot{\mathbf{u}}_{t}\, dx+\mu\int |x|^{\alpha/2}\dive(\mathbf{u}\otimes\Delta\mathbf{u})\dot{\mathbf{u}}_{t}\, dx\\
&+(\mu+\lambda)\int |x|^{\alpha/2}\dot{\mathbf{u}}_{t}\dive\big(\nabla(\dive\mathbf{u})\otimes\mathbf{u}\big)\, dx\\
:=&\sum_{i=1}^{14}L_{i}.
\end{split}
\end{equation*}
The estimates of $L_{1}+L_{6}, L_{9}$ to $L_{14}$ can refer to Lemma 4.6 in \cite{L2012}. Here we only list the result. We have
\begin{equation*}
\begin{split}
&L_{1}+L_{6}+L_{9}+L_{10}+L_{11}+L_{12}+L_{13}\\
\leq&\frac{d}{dt}\Psi(t)+C(\|\nabla\dot{\mathbf{u}}\|_{L^2}^2+\||x|^{\alpha/2}\nabla\dot{\mathbf{u}}\|_{L^2}^2+\|\nabla\mathbf{u}_{tt}\|_{L^2}^2+\|\mathbf{u}_{t}\|_{L^{q}}^2+\|\nabla^2\mathbf{u}_{t}\|_{L^2}^2+1),
\end{split}
\end{equation*}
where
\begin{equation*}
|\Psi(t)|\leq 2(2\mu+\lambda)\alpha^2\||x|^{\alpha/4}\nabla\dot{\mathbf{u}}\|_{L^2}^2+C\||x|^{\alpha/2}\nabla\mathbf{u}\|_{L^2}^2+C\|\nabla\mathbf{u}\|_{L^2}^2.
\end{equation*}
Then we consider $L_{2}+L_{7}$. By Lemmas \ref{o} and \ref{t}, we have
\begin{equation*}
\begin{split}
L_{2}+L_{7}=&\int |x|^{\alpha/2}\mathbf{u}\cdot\nabla h^2\dot{\mathbf{u}}_{t}\cdot\nabla n\, dx+2\int |x|^{\alpha/2}h^2\dive\mathbf{u}\dot{\mathbf{u}}_{t}\cdot\nabla n\, dx\\
&-\int |x|^{\alpha/2}\dot{\mathbf{u}}_{t}\cdot\nabla n\dive(h^2\mathbf{u})\, dx-\int |x|^{\alpha/2}\dot{\mathbf{u}}_{t}h^2\mathbf{u}\Delta n\, dx\\
=&\int |x|^{\alpha/2}h^2\dive\mathbf{u}\dot{\mathbf{u}}_{t}\cdot\nabla n dx-\int |x|^{\alpha/2}\dot{\mathbf{u}}_{t}h^2\mathbf{u}\Delta n\, dx\\
\leq&\||x|^{\alpha/4}\sqrt{h}\dot{\mathbf{u}}_{t}\|_{L^2}(\||x|^{\alpha/4}h^{\frac{3}{2}}\|_{L^2}\|\dive\mathbf{u}\nabla n\|_{L^{\infty}}+\||x|^{\alpha/4}h^{\frac{3}{2}}\|_{L^2}\|\mathbf{u}\Delta n\|_{L^{\infty}})\\
\leq&\varepsilon\||x|^{\alpha/4}\sqrt{h}\dot{\mathbf{u}}_{t}\|_{L^2}^2+C_{\varepsilon}(\||x|^{\alpha/2}h^2\|_{L^2}^2\|h\|_{L^2}+\|\nabla^4 n\|_{L^2}^2).
\end{split}
\end{equation*}
Then we estimate $L_{3}$. By Lemma \ref{s} and the definition of material derivative, we have
\begin{equation*}
\begin{split}
L_{3}=&-\int |x|^{\alpha/2}h^2\dot{\mathbf{u}}_{t}\cdot\nabla\dot{n}\, dx+\int |x|^{\alpha/2}\dot{\mathbf{u}}_{t}\cdot\nabla(\mathbf{u}\cdot\nabla n)h^2\, dx\\
\leq&\||x|^{\alpha/4}\sqrt{h}\dot{\mathbf{u}}_{t}\|_{L^2}(\||x|^{\alpha/4}\nabla\dot{n}\|_{L^2}\|h\|_{L^{\infty}}^{\frac{3}{2}}+\||x|^{\alpha/4}\nabla\mathbf{u}\|_{L^2}\|h^{\frac{3}{2}}\nabla n\|_{L^{\infty}}+\||x|^{\alpha/4}h^{\frac{3}{2}}\|_{L^2}\|\mathbf{u}\Delta n\|_{L^{\infty}})\\
\leq&\varepsilon\||x|^{\alpha/4}\sqrt{h}\dot{\mathbf{u}}_{t}\|_{L^2}^2+C_{\varepsilon}(\||x|^{\alpha/2}\nabla\dot{n}\|_{L^2}^2+\||x|^{\alpha/2}\nabla\mathbf{u}\|_{L^2}^2+\|\nabla^4 n\|_{L^2}^2+1).
\end{split}
\end{equation*}
Then we estimate $L_{4}$. By Lemma \ref{t}, we have
\begin{equation*}
\begin{split}
L_{4}=&-\int |x|^{\alpha/2}\dot{n}\dot{\mathbf{u}}_{t}\cdot\nabla h h\, dx+\int |x|^{\alpha/2}\mathbf{u}\cdot\nabla nh\dot{\mathbf{u}}_{t}\cdot\nabla h\, dx\\
\leq&\||x|^{\alpha/4}\sqrt{h}\dot{\mathbf{u}}_{t}\|_{L^2}\||x|^{\alpha/4}\dot{n}\|_{L^2}\|\sqrt{h}\nabla h\|_{L^{\infty}}+\||x|^{\alpha/4}\sqrt{h}\dot{\mathbf{u}}_{t}\|_{L^2}\||x|^{\alpha/4}\nabla n\|_{L^2}\|\sqrt{h}\mathbf{u}\cdot\nabla h\|_{L^{\infty}}\\
\leq&\varepsilon\||x|^{\alpha/4}\sqrt{h}\dot{\mathbf{u}}_{t}\|_{L^2}^2+C_{\varepsilon}(\||x|^{\alpha/2}\dot{n}\|_{L^2}^2+\||x|^{\alpha/2}\nabla n\|_{L^2}^2+1).
\end{split}
\end{equation*}
Then we estimate $L_{5}+L_{8}$. By Lemma \ref{t}, we have
\begin{equation*}
\begin{split}
&L_{5}+L_{8}\\
=&\int |x|^{\alpha/2}(1+n)\dot{\mathbf{u}}_{t}\cdot\nabla h^2\dive\mathbf{u}\, dx+\int |x|^{\alpha/2}\dot{\mathbf{u}}_{t}\cdot\nabla^2\mathbf{u}(1+n)h^2\, dx-\frac{1}{2}\int |x|^{\alpha/2}\dot{\mathbf{u}}_{t}\cdot\nabla n\mathbf{u}\cdot\nabla h^2\, dx\\
\leq&\||x|^{\alpha/4}\sqrt{h}\dot{\mathbf{u}}_{t}\|_{L^2}(\||x|^{\alpha/4}\dive\mathbf{u}\|_{L^2}\|(1+n)\sqrt{h}\nabla h\|_{L^{\infty}}+\||x|^{\alpha/4}h^{\frac{3}{2}}\|_{L^2}\|(1+n)\nabla^2\mathbf{u}\|_{L^{\infty}})\\
&+\||x|^{\alpha/4}\sqrt{h}\dot{\mathbf{u}}_{t}\|_{L^2}\||x|^{\alpha/4}\nabla n\|_{L^2}\|\sqrt{h}\mathbf{u}\cdot\nabla h\|_{L^{\infty}}\\
\leq&\varepsilon\||x|^{\alpha/4}\sqrt{h}\dot{\mathbf{u}}_{t}\|_{L^2}^2+C_{\varepsilon}(\||x|^{\alpha/2}\dive\mathbf{u}\|_{L^2}^2+\|\nabla^4\mathbf{u}\|_{L^2}^2+\||x|^{\alpha/2}\nabla n\|_{L^2}^2+1).
\end{split}
\end{equation*}
Combining all these estimates about $L_{i}\ (i=1\cdots,14)$, together with (\ref{200}), we obtain
\begin{equation}\label{222}
\begin{split}
&\frac{d}{dt}\int |x|^{\alpha/2}|\nabla\dot{\mathbf{u}}|^2\, dx+\frac{d}{dt}\int |x|^{\alpha/2}(\dive\dot{\mathbf{u}})^2\, dx+\int |x|^{\alpha/2}h|\dot{\mathbf{u}}_{t}|^2\, dx\\
\leq&C(\||x|^{\alpha/2}\nabla\dot{\mathbf{u}}\|_{L^2}^2+\|\nabla\dot{\mathbf{u}}\|_{L^2}^2+\|\nabla^4 n\|_{L^2}^2+\|\nabla\mathbf{u}_{tt}\|_{L^2}^2+\|\nabla^4\mathbf{u}\|_{L^2}^2+\||x|^{\alpha/2}\nabla\dot{n}\|_{L^2}^2)\\
&+\||x|^{\alpha/2}\nabla\mathbf{u}\|_{L^2}^2+\||x|^{\alpha/2}\dot{n}\|_{L^2}^2+\||x|^{\alpha/2}\nabla n\|_{L^2}^2+\|\mathbf{u}_{t}\|_{L^q}^2+\|\nabla^2\mathbf{u}_{t}\|_{L^2}^2+1).
\end{split}
\end{equation}
Multiplying the inequality (\ref{222}) by $t$, integrating over $\mathbb{R}^2$, combining with Lemma \ref{p}, \ref{r}, \ref{t}, and (\ref{164}) to get
\begin{equation}\label{223}
\sup_{0\leq t\leq T^*}t\||x|^{\alpha/4}|\nabla\dot{\mathbf{u}}|\|_{L^2}^2+\int_{0}^{T^*}t\||x|^{\alpha/2}\sqrt{h}\dot{\mathbf{u}}_{t}\|_{L^2}^2\, dt\leq C.
\end{equation}
It follows from Lemma \ref{f} and (\ref{223}) that
\begin{equation*}
\sup_{0\leq t\leq T^*} t\|\dot{\mathbf{u}}\|_{L^{8/\alpha}}^2\leq C\sup_{0\leq t\leq T^*} t\||x|^{\alpha/4}\nabla\dot{\mathbf{u}}\|_{L^2}^2\leq C,
\end{equation*}
which yields that
\begin{equation}\label{z200}
\sup_{0\leq t\leq T^*} t\|\mathbf{u}_{t}\|_{L^{8/\alpha}}^2\leq C\sup_{0\leq t\leq T^*} t(\|\dot{\mathbf{u}}\|_{L^{8/\alpha}}^2+\|\mathbf{u}\cdot\nabla\mathbf{u}\|_{L^{8/\alpha}}^2)\leq C.
\end{equation}
Using the $L^p$ theory to equation $(\ref{A})_{4}$, we have
\begin{equation}\label{165}
\sup_{0\leq t\leq T^*}t\|\nabla^2\mathbf{u}_{t}\|_{L^2}^2\leq C,\ \sup_{0\leq t\leq T^*} t\|\nabla^4\mathbf{u}\|_{L^2}^2\leq C.
\end{equation}
The estimate (\ref{z201}) can be given by (\ref{164}), (\ref{z200}), and (\ref{165}).
\end{proof}

\begin{lemma}\label{v}
The following estimate holds
\begin{equation}\label{z201}
\sup_{0\leq t\leq T^*} t(\|\nabla^2 c_{t}\|_{L^2}^2+\|\nabla^4 c\|_{L^2}^2)\leq C.
\end{equation}
\end{lemma}

\begin{proof}
It is easy to get
\begin{equation*}
\begin{split}
\int_{0}^{T^*}\|\dot{c}_{t}\|_{L^2}^2\, dt\leq &\int_{0}^{T^*} \|c_{tt}\|_{L^2}^2\, dt+\int_{0}^{T^*}\|(\mathbf{u}\cdot\nabla c)_{t}\|_{L^2}^2\, dt\\
\leq&\int_{0}^{T^*}(\|c_{tt}\|_{L^2}^2+\|\mathbf{u}_{t}\|_{L^q}^2\|\nabla c\|_{L^p}^2+\|\mathbf{u}\|_{L^{\infty}}\|\nabla c_{t}\|_{L^2}^2)\, dt\\
\leq&C.
\end{split}
\end{equation*}
Differentiating the equation $(\ref{A})_{2}$ with respect to $t$ twice to get
\begin{equation}\label{225}
\dot{c}_{tt}+(c\dive\mathbf{u})_{tt}=\Delta c_{tt}-(nc)_{tt}.
\end{equation}
Multiplying the equation (\ref{225}) by $\dot{c}_{t}$ and integrating over $\mathbb{R}^2$ to get
\begin{equation*}
\begin{split}
&\frac{1}{2}\frac{d}{dt}\int \dot{c}_{t}^2\, dx+\int |\nabla c_{t}|^2\, dx\\
=&-\int (c\dive\mathbf{u})_{tt}\dot{c}_{t}\, dx+\int\nabla(\mathbf{u}\cdot\nabla c)_{t}\nabla\dot{c}_{t}\, dx-\int(nc)_{tt}\dot{c}_{t}\, dx\\
:=&\sum_{i=1}^{3}M_{i}.
\end{split}
\end{equation*}
In the following, we estimate $M_{i}\ (i=1, 2, 3)$ respectively. First, we consider $M_{1}$. By Lemma \ref{t}, it is easy to get
\begin{equation*}
\begin{split}
M_{1}=&-\int c_{tt}\dive\mathbf{u}\dot{c}_{t}\, dx-2\int c_{t}\dive\mathbf{u}_{t}\dot{c}_{t}\, dx-\int c\dive\mathbf{u}_{tt}\dot{c}_{t}\, dx\\
\leq&\|c_{tt}\|_{L^2}\|\dive\mathbf{u}\|_{L^{\infty}}\|\dot{c}_{t}\|_{L^2}+\|c_{t}\|_{L^4}\|\dive\mathbf{u}_{t}\|_{L^4}\|\dot{c}_{t}\|_{L^2}+\|c\|_{L^{\infty}}\|\dive\mathbf{u}_{tt}\|_{L^2}\|\dot{c}_{t}\|_{L^2}\\
\leq&C(\|\dot{c}_{t}\|_{L^2}^2+\|c_{tt}\|_{L^2}^2+\|\nabla^2\mathbf{u}_{t}\|_{L^2}^2+\|\dive\mathbf{u}_{tt}\|_{L^2}^2+1).
\end{split}
\end{equation*}
Then we consider $M_{2}$. Using Lemmas \ref{o} and \ref{r}, we have
\begin{equation*}
\begin{split}
M_{2}=&-\int\nabla\mathbf{u}_{t}\nabla c\nabla\dot{c}_{t}\, dx-\int\mathbf{u}_{t}\nabla^2 c\nabla\dot{c}_{t}\, dx-\int\nabla\mathbf{u}\nabla c_{t}\nabla\dot{c}_{t}\, dx-\int\mathbf{u}\nabla^2 c_{t}\nabla \dot{c}_{t}\, dx\\
\leq&\|\nabla\mathbf{u}_{t}\|_{L^2}\|\nabla c\|_{L^{\infty}}\|\nabla\dot{c}_{t}\|_{L^2}+\|\mathbf{u}_{t}\|_{L^q}\|\nabla^2 c\|_{L^p}\|\nabla\dot{c}_{t}\|_{L^2}\\
&+\|\nabla\mathbf{u}\|_{L^4}\|\nabla c_{t}\|_{L^4}\|\nabla\dot{c}_{t}\|_{L^2}+\|\mathbf{u}\|_{L^{\infty}}\|\nabla^2 c_{t}\|_{L^2}\|\nabla\dot{c}_{t}\|_{L^2}\\
\leq&\varepsilon\|\nabla\dot{c}_{t}\|_{L^2}^2+C_{\varepsilon}(\|\mathbf{u}_{t}\|_{L^q}^2+\|\nabla^2 c_{t}\|_{L^2}^2).
\end{split}
\end{equation*}
Finally, we consider the nonlinear term $M_{3}$. Through Lemma \ref{s}, we have
\begin{equation*}
\begin{split}
M_{3}=&-\int n_{tt}c\dot{c}_{t}\, dx-\int n_{t}c_{t}\dot{c}_{t}\, dx-\int nc_{tt}\dot{c}_{t}\, dx\\
\leq&\|n_{tt}\|_{L^2}\|c\|_{L^{\infty}}\|\dot{c}_{t}\|_{L^2}+\|n_{t}\|_{L^4}\|c_{t}\|_{L^4}\|\dot{c}_{t}\|_{L^2}+\|n\|_{L^{\infty}}\|c_{tt}\|_{L^2}\|\dot{c}_{t}\|_{L^2}\\
\leq&C(\|\dot{c}_{t}\|_{L^2}^2+\|n_{tt}\|_{L^2}^2+\|c_{tt}\|_{L^2}^2+1).
\end{split}
\end{equation*}
Combining all these estimates about $M_{i}$ and choosing $\varepsilon=1/2$, we have
\begin{equation}\label{226}
\begin{split}
&\frac{d}{dt}\|\dot{c}_{t}\|_{L^2}^2+\|\nabla\dot{c}_{t}\|_{L^2}^2\\
\leq&C(\|c_{tt}\|_{L^2}^2+\|\dot{c}_{t}\|_{L^2}^2+\|\nabla^2\mathbf{u}_{t}\|_{L^2}^2+\|\dive\mathbf{u}_{tt}\|_{L^2}^2+\|\mathbf{u}_{t}\|_{L^q}^2+\|\nabla^2 c_{t}\|_{L^2}^2+\|n_{tt}\|_{L^2}^2+1).
\end{split}
\end{equation}
Multiplying the inequality (\ref{226}) by $t$, then using the Gronwall's inequality, together with Lemma \ref{s}, \ref{r}, and \ref{u} to get
\begin{equation}\label{167}
\sup_{0\leq t\leq T^*} t\|\dot{c}_{t}\|_{L^2}^2+\int_{0}^{T^*} t\|\nabla\dot{c}_{t}\|_{L^2}^2\, dt\leq C.
\end{equation}
Through the elliptic estimate to equation $(\ref{A})_{2}$, together with (\ref{167}) and Lemma \ref{r}, we have
\begin{equation}\label{227}
\begin{split}
&\sup_{0\leq t\leq T^*}t\|\nabla^2 c_{t}\|_{L^2}^2\\
\leq&C\sup_{0\leq t\leq T^*} t(\|\dot{c}_{t}\|_{L^2}^2+\|c_{t}\dive\mathbf{u}\|_{L^2}^2+\|c\dive\mathbf{u}_{t}\|_{L^2}^2+\|n_{t}c\|_{L^2}^2+\|nc_{t}\|_{L^2}^2)\\
\leq&C
\end{split}
\end{equation}
and
\begin{equation}\label{168}
\sup_{0\leq t\leq T^*} t\|\nabla^4 c\|_{L^2}^2\leq\sup_{0\leq t\leq T^*} C t(\|\nabla^2 c_{t}\|_{L^2}^2+\|\nabla^3(c\mathbf{u})\|_{L^2}^2+\|\nabla^2(nc)\|_{L^2}^2)\leq C.
\end{equation}
Combining the estimates (\ref{227}) and (\ref{168}), we complete the proof of (\ref{z201}).
\end{proof}

\begin{lemma}\label{w}
The following estimate hold
\begin{equation}\label{z202}
\sup_{0\leq t\leq T^*} t(\|\nabla^2 n_{t}\|_{L^2}^2+\|\nabla^4 n\|_{L^2}^2)\leq C.
\end{equation}
\end{lemma}

\begin{proof}
According to Lemma \ref{s} and the definition of the material derivative, it is easy to get
\begin{equation*}
\int_{0}^{T^*}\|\dot{n}_{t}\|_{L^2}^2\, dt\leq C.
\end{equation*}
Differentiating the equation $(\ref{A})_{1}$ with respect to $t$ twice, then multiplying the resulting equation by $\dot{n}_{t}$ and integrating over $\mathbb{R}^2$ to get
\begin{equation}\label{z301}
\begin{split}
&\frac{1}{2}\frac{d}{dt}\int\dot{n}_{t}^2\, dx+\int |\nabla\dot{n}_{t}|^2\, dx\\
=&-\int (n\dive\mathbf{u})_{tt}\dot{n}_{t}\, dx-\int\Delta(\mathbf{u}\cdot\nabla n)_{t}\dot{n}_{t}+\int(n\cdot\nabla c)_{tt}\nabla\dot{n}_{t}\, dx.
\end{split}
\end{equation}
For the nonlinear term $\int (n\cdot\nabla c)_{tt}\nabla\dot{n}_{t}\, dx$, through Lemmas \ref{s} and \ref{o}, we have
\begin{equation*}
\begin{split}
&\int (n\cdot\nabla c)_{tt}\nabla\dot{n}_{t}\, dx\\
\leq&\|n_{tt}\|_{L^2}\|\nabla c\|_{L^{\infty}}\|\nabla\dot{n}_{t}\|_{L^2}+\|n_{t}\|_{L^4}\|\nabla c_{t}\|_{L^4}\|\nabla\dot{n}_{t}\|_{L^2}+\|n\|_{L^{\infty}}\|\nabla c_{tt}\|_{L^2}\|\nabla\dot{n}_{t}\|_{L^2}\\
\leq&\varepsilon\|\nabla\dot{n}_{t}\|_{L^2}^2+C_{\varepsilon}(\|n_{tt}\|_{L^2}^2+\|\nabla^2 c_{t}\|_{L^2}^2+\|\nabla c_{tt}\|_{L^2}^2+1)\\
\leq&\varepsilon\|\nabla\dot{n}_{t}\|_{L^2}^2+C_{\varepsilon}(\|n_{tt}\|_{L^2}^2+\|\nabla^2 c_{t}\|_{L^2}^2+\|\nabla\dot{c}_{t}\|_{L^2}^2+\|\mathbf{u}_{t}\|_{L^q}^2+1),
\end{split}
\end{equation*}
where we use the following fact
\begin{equation}\label{230}
\begin{split}
&\|\nabla c_{tt}\|_{L^2}^2\\
\leq&C(\|\nabla\dot{c}_{t}\|_{L^2}^2+\|\nabla\mathbf{u}_{t}\nabla c\|_{L^2}^2+\|\mathbf{u}_{t}\cdot\nabla^2 c\|_{L^2}^2+\|\nabla\mathbf{u}\nabla c_{t}\|_{L^2}^2+\|\mathbf{u}\cdot\nabla^2 c_{t}\|_{L^2}^2)\\
\leq&C(\|\nabla\dot{c}_{t}\|_{L^2}^2+\|\nabla\mathbf{u}_{t}\|_{L^2}^2\|\nabla c\|_{L^{\infty}}^2+\|\mathbf{u}_{t}\|_{L^q}^2\|\nabla^2 c\|_{L^p}^2+\|\nabla\mathbf{u}\|_{L^{\infty}}^2\|\nabla c_{t}\|_{L^2}^2+\|\mathbf{u}\|_{L^{\infty}}^2\|\nabla^2 c_{t}\|_{L^2}^2)\\
\leq&C(\|\nabla\dot{c}_{t}\|_{L^2}^2+\|\mathbf{u}_{t}\|_{L^q}^2+\|\nabla^2 c_{t}\|_{L^2}^2).
\end{split}
\end{equation}
The estimates of the other two nonlinear terms on the right side of (\ref{z301}) are the same as $M_{1}$ and $M_{2}$. Here we omit. (\ref{z301}) can be written as
\begin{equation}\label{229}
\begin{split}
\frac{d}{dt}\|\dot{n}_{t}\|_{L^2}^2+\|\nabla\dot{n}_{t}\|_{L^2}^2\leq&C(\|\dot{n}_{t}\|_{L^2}^2+\|n_{tt}\|_{L^2}^2+\|\nabla^2\mathbf{u}_{t}\|_{L^2}^2+\|\dive\mathbf{u}_{tt}\|_{L^2}^2+\|\mathbf{u}_{t}\|_{L^q}^2)\\
&+C(\|\nabla^2 n_{t}\|_{L^2}^2+\|\nabla^2 c_{t}\|_{L^2}^2+\|\nabla\dot{c}_{t}\|_{L^2}^2+1).
\end{split}
\end{equation}
Multiplying (\ref{229}) by $t$, together with Lemma \ref{s}, \ref{p}, \ref{r}, \ref{u}, and \ref{v}, we have
\begin{equation}\label{170}
\sup_{0\leq t\leq T^*} t\|\dot{n}_{t}\|_{L^2}^2+\int_{0}^{T^*} t\|\nabla\dot{n}_{t}\|_{L^2}^2\, dt\leq C.
\end{equation}
Through elliptic estimate, we have
\begin{equation}\label{171}
\sup_{0\leq t\leq T^*}t\|\nabla^2 n_{t}\|_{L^2}^2\leq C\sup_{0\leq t\leq T^*} t(\|\dot{n}_{t}\|_{L^2}^2+\|(n\dive\mathbf{u})_{t}\|_{L^2}^2+\|\nabla\cdot(n\nabla c)_{t}\|_{L^2}^2)\leq C
\end{equation}
and
\begin{equation}\label{172}
\sup_{0\leq t\leq T^*} t\|\nabla^4 n\|_{L^2}^2\leq C\sup_{0\leq t\leq T^*} t(\|\nabla^2 n_{t}\|_{L^2}^2+\|\nabla^3(n\mathbf{u})\|_{L^2}^2+\|\nabla^3(n\nabla c)\|_{L^2}^2)\leq C.
\end{equation}
Combining the estimates (\ref{171}) and (\ref{172}), we obtain (\ref{z202}).
\end{proof}

Now we prove our main result about Theorem \ref{B}.\\
{\textbf{Proof of Theorem \ref{B}}} If $\widetilde{h}=0$, Proposition \ref{l} and Lemmas \ref{n}-\ref{w} show Theorem \ref{B}. It remains to prove that $(n, c, h, \mathbf{u})$ becomes a classical solution for positive time, that is for any $\tau\in (0, T^*]$,
\begin{equation*}
n_{t},\ \nabla^2 n,\ c_{t},\ \nabla^2 c,\ h_{t},\ \nabla h,\ \mathbf{u}_{t},\ \nabla^2\mathbf{u}\in C(\overline{\mathbb{R}^2}\times[\tau, T^*]).
\end{equation*} 
It follows from Proposition \ref{l}, Lemma \ref{t}, and \ref{u} that for any $\tau\in (0, T^*]$,
\begin{equation}\label{233}
\sup_{\tau\leq t\leq T^*}\|\nabla^2\mathbf{u}\|_{H^2}\leq C(\tau),
\end{equation}
which together with $(\nabla^2\mathbf{u})_{t}\in L^2(\mathbb{R}^2\times(0, T^*))$ gives that for $p\in (2, +\infty)$,
\begin{equation}\label{z800}
\nabla^2\mathbf{u}\in C([\tau, T^*]; H^1\cap W^{1, p})\hookrightarrow C(\overline{\mathbb{R}^2}\times[\tau, T^*]).
\end{equation}
Lemmas \ref{r} and \ref{u} give that
\begin{equation}\label{231}
\sup_{\tau\leq t\leq T^*}(\|\mathbf{u}_{t}\|_{L^{8/\alpha}}+\|\nabla\mathbf{u}_{t}\|_{H^1})+\int_{\tau}^{T^*}\|\nabla\mathbf{u}_{tt}\|_{L^2}^2\, dt\leq C(\tau),
\end{equation}
which implies for $p\in(2, +\infty)$, we have
\begin{equation}\label{232}
\nabla\mathbf{u}_{t}\in C([\tau, T^*]; L^2\cap L^p).
\end{equation}
It follows from (\ref{231}) and (\ref{232}) that
\begin{equation}\label{z700}
\mathbf{u}_{t}\in C(\overline{\mathbb{R}^2}\times[\tau, T^*]).
\end{equation}
The proof of $n_{t}, c_{t}\in C(\overline{\mathbb{R}^2}\times[\tau, T^*])$ is the same as (\ref{z700}) and $\nabla^2 n, \nabla^2 c\in C(\overline{\mathbb{R}^2}\times[\tau, T^*])$ is the same as (\ref{z800}). Here we omit. Lemma \ref{n}, \ref{q}, and \ref{t} show that
\begin{equation*}
\nabla h\in L^{\infty}(0, T^*; H^2),\ \nabla h_{t}\in L^{\infty}(0, T^*; L^2),
\end{equation*}
which yields that for $p\in (2, +\infty)$,
\begin{equation}\label{234}
\nabla h\in C([0, T^*]; H^1\cap W^{1, p})\hookrightarrow C(\overline{\mathbb{R}^2}\times[0, T^*]).
\end{equation}
It follows from $(\ref{A})_{3}$, (\ref{z800}), and (\ref{234}) that
\begin{equation*}
h_{t}\in C(\overline{\mathbb{R}^2}\times[0, T^*]).
\end{equation*}
If $\widetilde{h}>0$, after bounding $\|\mathbf{u}\|_{L^2}$ and $\|\dot{\mathbf{u}}\|_{L^2}$ through Lemma \ref{Q1}, we obtain Theorem \ref{B} as above. Hence, the proof of Theorem \ref{B} is completed.

\section*{Acknowledgments}

The research of Zhen Luo was supported in part by National Natural Science Foundation of China (No.12271455, 12171401). The research of Yucheng Wang was supported in part by National Natural Science Foundation of China (No.12271357), Shanghai Science and Technology Innovation Action Plan (No.21JC1403600), and the CSC-DAAD Postdoc Scholarship. Yucheng Wang thanks the School of Business Informatics and Mathematics, University of Mannheim for kindly host.

\end{document}